\documentclass[a4paper,11pt,reqno]{amsart}
\RequirePackage[english]{babel}

\RequirePackage[dvipsnames]{xcolor}

\RequirePackage{fancyhdr}
\RequirePackage{float} % Correct figure placement

\RequirePackage{tikz}
\usetikzlibrary{arrows}
\usetikzlibrary{arrows.meta}
\usetikzlibrary{patterns}
\usetikzlibrary{shapes}
\usetikzlibrary{calc} % Relative position
\usetikzlibrary{decorations.pathreplacing}
\usetikzlibrary{decorations.pathmorphing}
\usetikzlibrary{decorations.markings}
\usetikzlibrary{shapes.geometric}
\usetikzlibrary{shapes.multipart}
\usetikzlibrary{angles}
\usetikzlibrary{quotes}
\usetikzlibrary{positioning}
\usetikzlibrary{babel} % Compability with babel
\usetikzlibrary{graphs}
\usetikzlibrary{fit}
\usetikzlibrary{matrix}
\usetikzlibrary{cd}
\usetikzlibrary{intersections}
\usetikzlibrary{backgrounds}
\usetikzlibrary{tikzmark}
\usetikzlibrary{fadings}

\RequirePackage{rank-2-roots}

\RequirePackage{lmodern} % Schriftart
\RequirePackage[T1]{fontenc}
\RequirePackage[utf8]{inputenc}

\RequirePackage{amsmath}
\RequirePackage{amsfonts}
\RequirePackage{amsxtra}
\RequirePackage{amssymb}
\RequirePackage{amsthm}
\RequirePackage[intlimits]{empheq}
\RequirePackage{mathtools}
\RequirePackage{mathdots}
\RequirePackage{mathrsfs} %Fuer Schreibschrift Buchstaben (use \mathscr{})
\RequirePackage{dsfont}
\RequirePackage{xfrac}
\RequirePackage{stmaryrd}
\SetSymbolFont{stmry}{bold}{U}{stmry}{m}{n} % Um Warning zu unterdruecken

\RequirePackage[shortlabels]{enumitem} % Aufzaehlungen % Advanced enumerating

\RequirePackage{setspace}
\RequirePackage{textcase}
\RequirePackage{mparhack}
\RequirePackage{gensymb} % Erlaubt einige Symbole im Text- und Mathemodus
\RequirePackage{microtype}
\RequirePackage[symbol]{footmisc}

\RequirePackage{geometry}

\RequirePackage{todonotes}

\RequirePackage{booktabs}
\RequirePackage{multirow} % Mehrere Zellen uebere mehrere Reihen einer Tabelle
\RequirePackage{tabularx}
\RequirePackage[fixlanguage]{babelbib} % Fuer babel in Bibliography

\RequirePackage{caption}

\RequirePackage{mdframed} % Fuer Boxen um Theoreme

\RequirePackage{hyperref} % Anklichbare Referenzen
\RequirePackage[nameinlink]{cleveref} % Has to be loaded after hyperref

\RequirePackage{arydshln}

\makeatletter
\newif\iftikz@shading@path

\tikzset{
	shading xsep/.store in=\tikz@pathshadingxsep,
	shading ysep/.store in=\tikz@pathshadingysep,
	shading sep/.style={shading xsep=#1, shading ysep=#1},
	shading sep=0.0cm,
}

\def\tikz@shadepath#1{% 
  \iftikz@shading@path%
  \else%
    \tikz@shading@pathtrue%
    \pgfgetpath\tikz@currentshadingpath%
    \begingroup%
      \pgfsys@beginscope% <- may not be necessary
      \tikzset{#1}%
      \xdef\tikz@tmp{\noexpand\def\noexpand\tikz@pathshadingxsep{\tikz@pathshadingxsep}%
          \noexpand\def\noexpand\tikz@pathshadingysep{\tikz@pathshadingysep}}%
      \pgfsys@endscope%
    \endgroup
    \tikz@tmp%
  \pgfextract@process\pgf@shadingpath@southwest{\pgfpointadd{\pgfqpoint{\pgf@pathminx}{\pgf@pathminy}}%
      {\pgfpoint{-\tikz@pathshadingxsep}{-\tikz@pathshadingysep}}}%%
  \pgfextract@process\pgf@shadingpath@northeast{\pgfpointadd{\pgfqpoint{\pgf@pathmaxx}{\pgf@pathmaxy}}%
      {\pgfpoint{\tikz@pathshadingxsep}{\tikz@pathshadingysep}}}%
    \pgfsetpath\pgfutil@empty%                          
    \let\tikz@options@saved=\tikz@options%
    \let\tikz@mode@saved=\tikz@mode%
    \let\tikz@options=\pgfutil@empty%
    \let\tikz@mode=\pgfutil@empty%
    \tikz@addoption{%
      \pgfinterruptpath%
      \pgfinterruptpicture%
        \begin{tikzfadingfrompicture}[name=.]
        \pgfscope%
          \tikzset{shade path/.style=}% Make absolutely sure shade path is not inherited.
          \path \pgfextra{%
            \pgfsetpath\tikz@currentshadingpath%
            \pgf@shadingpath@southwest
            \expandafter\pgf@protocolsizes{\the\pgf@x}{\the\pgf@y}%
            \pgf@shadingpath@northeast%
            \expandafter\pgf@protocolsizes{\the\pgf@x}{\the\pgf@y}%
            \let\tikz@options=\tikz@options@saved%
            \let\tikz@mode=\tikz@mode@saved%
          };
      \xdef\pgf@shadingboundingbox@southwest{\noexpand\pgfqpoint{\the\pgf@picminx}{\the\pgf@picminy}}%
      \xdef\pgf@shadingboundingbox@northeast{\noexpand\pgfqpoint{\the\pgf@picmaxx}{\the\pgf@picmaxy}}%
        \endpgfscope
      \end{tikzfadingfrompicture}%
      \endpgfinterruptpicture%
      \endpgfinterruptpath%
      \pgftransformreset%
  \pgfpathrectanglecorners{\pgf@shadingboundingbox@southwest}{\pgf@shadingboundingbox@northeast}%
      \let\tikz@path@picture=\pgfutil@empty%
      \tikz@mode@fillfalse%
      \tikz@mode@drawfalse%
      \tikz@mode@doublefalse%
      \tikz@mode@clipfalse%
      \tikz@mode@boundaryfalse%
      \tikz@mode@fade@pathfalse%
      \tikz@mode@fade@scopefalse%
      \tikzset{#1}%
      \tikz@mode%
      \def\tikz@path@fading{.}%
      \tikz@mode@fade@pathtrue%
      \tikz@fade@adjustfalse%
    \pgfpointscale{0.5}{\pgfpointadd{\pgf@shadingboundingbox@southwest}{\pgf@shadingboundingbox@northeast}}%
      \edef\tikz@fade@transform{shift={(\the\pgf@x,\the\pgf@y)}}%
      \pgfsetfading{\tikz@path@fading}{\tikz@do@fade@transform}%
      \tikz@mode@fade@pathfalse%              
    }%
  \fi%
}
\tikzset{
	shade path/.code={%
		\tikz@addmode{\tikz@shadepath{#1}}%
	}
}
\makeatother % <- To close the \makeatletter call

\newcommand{\Nnum}{\mathbb N}
\newcommand{\Znum}{\mathbb Z}
\newcommand{\Qnum}{\mathbb Q}
\newcommand{\Rnum}{\mathbb R}
\newcommand{\Cnum}{\mathbb C}

\DeclareMathOperator{\id}{id}

\newcommand{\thref}[1]{\Cref{#1}}

\pgfmathsetmacro\xMax{8}
\pgfmathsetmacro\yMax{7}
\pgfmathsetmacro\yMAX{2*\yMax}
\pgfmathsetmacro\s{2*sqrt(3)/3}

\tikzset{
font=\scriptsize,
>=stealth', % latex style arrows just look better than the default.
shape triangle/.style = {draw, regular polygon, regular polygon sides=3, minimum height=\s cm, % anchor=center,
	inner sep=0pt,
	},
Point/.style={circle, inner sep=1.5pt}}

\renewcommand{\obar}[1]{\overline{#1}}
\newcommand{\ubar}[1]{\underline{#1}}

\newcommand{\quoted}[1]{``#1''}

\providecommand\given{} % needed for \Set
\newcommand\SetSymbol[1][]{%
	\nonscript\:#1\vert
	\allowbreak
	\nonscript\:
	\mathopen{}
}
\DeclarePairedDelimiterX\Set[1]\{\}{%
	\renewcommand\given{\SetSymbol[\delimsize]}
	#1
}

\newcommand\restr[3][]{
    \left.\kern-\nulldelimiterspace % automatically resize the bar with \right
	#2 % the function
	\vphantom{\big\vert} % size is at least like that of an integral
	\right\vert_{#3}^{#1}
}

\DeclareMathOperator{\Maps}{Maps}
\DeclareMathOperator{\Hom}{Hom}
\DeclareMathOperator{\End}{End}
\DeclareMathOperator{\Aut}{Aut}

\DeclareMathOperator{\Image}{Im} % Image
\DeclareMathOperator{\Ker}{Ker}	% Kernel
\DeclareMathOperator{\Coker}{Coker}	% Cokernel

\newcommand{\eucl}{\mathbb{E}}

\newcommand{\sector}{\mathcal S}
\newcommand{\transl}{t}

\DeclareMathOperator{\conv}{conv}

\newcommand{\sectorMap}{s}
\newcommand{\secMetricFunc}{d_\vartheta} % Metric for sectors
\newcommand{\secMetric}[1]{\secMetricFunc\rbrackets*{#1}} 
\newcommand{\dircSecMetric}[2]{\secMetricFunc^{#1}\rbrackets*{#2}}
\newcommand{\secOnQuot}{s} % Sector on quotient

\newcommand{\lipschitzFunc}[1][\vartheta]{\mathcal F_{#1}}

\newcommand{\stdSimplex}{\Delta}
\newcommand{\topolSpace}{X}
\newcommand{\simplStrucMaps}{\kappa}
\newcommand{\simplexSet}{\mathfrak S}
\newcommand{\labeling}{\operatorname{col}}

\newcommand{\coweightLat}{P^\vee}
\newcommand{\innerCoweights}{\coweightLat_{++}}
\newcommand{\domCoweights}{\coweightLat_+}
\newcommand{\corootLat}{Q^\vee}
\newcommand{\rootSys}{R}
\newcommand{\rootBasis}{B}
\newcommand{\weylGr}{W_0}
\newcommand{\extWeyl}{W_{\mathrm{ext}}}
\newcommand{\affWeyl}{W_{\mathrm{aff}}}

\newcommand{\trAut}{\Aut_{\mathrm{tr}}}
\newcommand{\trAutRoot}{\Aut_{\mathrm{tr}, \rootSys}}
\newcommand{\trHom}{\Hom_{\mathrm{tr}}}

\newcommand{\chambSys}{\mathcal C} % Chamber System

\DeclareMathOperator{\rank}{rank}

\DeclareMathOperator{\Res}{Res}

\newcommand{\build}{\Delta} % Building
\newcommand{\euclbuild}{\Delta} % Euclidean Building
\newcommand{\boundary}{\Omega}

\newcommand{\apartSys}{\mathcal A}

\newcommand{\apartment}{A}
\newcommand{\chambset}{\mathcal C} % Set of Chambers of Building
\newcommand{\coxMat}{M} % Coxeter matrix
\newcommand{\coxCompl}{\chambSys(W)} % Coxeter complex

\newcommand{\euclCox}{\Sigma} % Euclidean Coxeter complex

\newcommand{\quot}{\mathcal C}
\newcommand{\secQuot}{\sector(\quot)} % Set of admissible Sectors

\newcommand{\tOp}{\mathcal L} % Transfer Operator
\newcommand{\paramA}{\mu}
\newcommand{\paramB}{\lambda}

\newcommand{\ball}{B}

\newcommand{\lift}[1]{\widetilde{#1}} % Lift of function

\newcommand{\operator}{T}
\newcommand{\operatorFam}{T}

\newcommand{\diff}{\mathrm{d}} % d for integral

\newcommand{\complVs}{V}
\newcommand{\finiteSpace}{W}
\newcommand{\banachSpace}{X}

\DeclareMathOperator{\Eig}{Eig}
\DeclareMathOperator{\Hau}{H}

\newcommand{\specRad}{r}
\newcommand{\essSpecRad}{\specRad_{\mathrm{ess}}}

\newcommand{\semigr}{H}
\newcommand{\semigrgen}{\varpi}

\newcommand{\semigrelemB}{\varpi} % Das hat Joachim angepasst (von k)

\newcommand{\nPolyAlg}{\Cnum[X_1,\ldots,X_n]}
\newcommand{\polyAlg}{\Cnum[\ubar{X}]}
\newcommand{\genVar}{X}

\newcommand{\annulus}[2]{\mathbb{A}_{#1}^{#2}}

\DeclareMathOperator{\Stab}{Stab} % Stabilizer
\DeclareMathOperator{\Tor}{Tor} % Tor functor
\DeclareMathOperator{\Ext}{Ext} % Ext functor

\DeclarePairedDelimiter\rbrackets{\lparen}{\rparen}
\DeclarePairedDelimiter\abs{\lvert}{\rvert} % Absolute value
\DeclarePairedDelimiter\spn{\langle}{\rangle} % Span
\DeclarePairedDelimiter\spr{\langle}{\rangle} % Skalarprodukt
\DeclarePairedDelimiter{\norm}{\lVert}{\rVert} % Eine Norm (Normstriche)
\newlist{proofcase}{description}{4}
\setlist[proofcase]{font=\scshape}

\newlist{prooflist}{description}{4}
\setlist[prooflist]{font=\normalfont}

\newcommand{\ximpliesy}[2]{\item[\hspace{.65em}\quoted{\( \text{#1} \Rightarrow \text{#2} \)}\hspace{.05em}]}

\makeatletter
\newcommand{\myitem}[1]{%
\item[#1]\protected@edef\@currentlabel{#1}%
}
\makeatother

\definecolor{gray75}{gray}{0.75} % Das ist die Farbe für den Kapitel Strich

\definecolor{upbblue}{RGB}{0,32,91}
\definecolor{upbgray}{RGB}{85,85,85}
\definecolor{upblightgray}{RGB}{199,201,199}
\definecolor{upbred}{RGB}{215,51,103}
\definecolor{upbgreen}{RGB}{164,196,36}
\definecolor{upbcyan}{RGB}{24,176,226}
\definecolor{upborange}{RGB}{242,149,18}
\definecolor{upbcassis}{RGB}{169,57,131}
\definecolor{upblightblue}{RGB}{0,127,185}

\colorlet{remarkColor}{blue!20}

\definecolor{blue1}{HTML}{03045e}
\definecolor{blue2}{HTML}{023e8a}
\definecolor{blue3}{HTML}{0077b6}
\definecolor{blue4}{HTML}{0096c7}
\definecolor{blue5}{HTML}{00b4d8}
\definecolor{blue6}{HTML}{48cae4}
\definecolor{blue7}{HTML}{90e0ef}
\definecolor{blue8}{HTML}{ade8f4}

\theoremstyle{plain}
\newtheorem{thm}{Theorem}[section]
\newtheorem{prop}[thm]{Proposition}
\newtheorem{lem}[thm]{Lemma}
\newtheorem{cor}[thm]{Corollary}

\theoremstyle{definition}
\newtheorem{defi}[thm]{Definition}
\newtheorem{exmp}[thm]{Example}

\theoremstyle{remark}
\newtheorem{rem}[thm]{Remark}

\theoremstyle{empty}

\AddToHook{env/defi/begin}{\crefalias{thm}{defi}}
\AddToHook{env/exmp/begin}{\crefalias{thm}{exmp}}
\AddToHook{env/lem/begin}{\crefalias{thm}{lem}}
\AddToHook{env/rem/begin}{\crefalias{thm}{rem}}
\AddToHook{env/cor/begin}{\crefalias{thm}{cor}}
\AddToHook{env/prop/begin}{\crefalias{thm}{prop}}

\crefname{lem}{lemma}{lemmata}
\crefname{rem}{remark}{remarks}
\crefname{subsection}{subsection}{subsections}
\crefname{defi}{definition}{definitions}

\begin{document}
\title[Transfer operators on buildings]{Spectral theory for transfer operators on compact quotients of Euclidean buildings}

\author{Joachim Hilgert}
\email{hilgert@mail.upb.de}
\address{Institut f\"ur Mathematik, Universit\"at Paderborn, Paderborn, Germany}
\author{Daniel Kahl}
\email{daniel.kahl@mail.upb.de}
\address{Institut f\"ur Mathematik, Universit\"at Paderborn, Paderborn, Germany}
\author{Tobias Weich}
\email{tobias.weich@mail.upb.de}
\address{Institut f\"ur Mathematik und Insittut für Photonische Quantensysteme (PhoQS), Universit\"at Paderborn, Paderborn, Germany}

\begin{abstract}
	In this paper we generalize the geodesic flow on (finite) homogeneous graphs to a multiparameter flow on compact quotients of Euclidean buildings.
	Then we study the joint spectra of the associated transfer operators acting on suitable Lipschitz spaces.
	The main result says that outside an arbitrarily small neighborhood of zero in the set of spectral parameters the Taylor spectrum of the commuting family of transfer operators is contained in the joint point spectrum. 
\end{abstract}

\maketitle

%%%%%%%%%%%%%%%%%%%%%%%%%%%%%%%%%%%%%%%%%%%%%%%%%%%%%%%%%%%%%%%%%%

\setcounter{tocdepth}{2}
\tableofcontents

%%%%%%%%%%%%%%%%%%%%%%%%%%%%%%%%%%%%%%%%%%%%%%%%%%%%%%%%%%%%%%%%%%

\section{Introduction}
A spectral theoretic invariant that can be associated to a large class of dynamical systems that exhibit in some sense enough hyperbolicity is the set of dynamical resonances.
They had their origin in the thermodynamic formalism developed by Ruelle, Bowen, and Pollicott \cite{Ruelle1978, Bowen1975, Pollicott1985} and others in the 70s and were particularly aiming to understand the leading resonance, the associated invariant Gibbs measures and their statistical properties (e.g.\ exponential mixing, central limit theorems, etc.).
In the past 20 years it has however become clear that not only the leading resonance is of interest, but that the whole spectrum of dynamical resonances provides an interesting spectral invariant of the dynamical systems.
The foundation of this paradigm shift was the development of anisotropic Sobolev spaces \cite{BlankKellerLiverani2002, Liverani2004, BaladiTsujii2007, ButterleyLiverani2007, FaureSjostrand2011, DyatlovGuillarmou2016, BonthonneauWeich2022} which allowed to define, for smooth hyperbolic flows, a resonance spectrum in the whole complex plane.
These spectral foundations quickly lead to important applications, such as the proof of Smale's conjecture on the meromorphic continuation of Ruelle's zeta function \cite{GiuliettiLiveraniPollicott2013, DyatlovZworski2016, DyatlovGuillarmou2016, DyatlovGuillarmou2018}, a proof of Fried's conjecture in small dimension \cite{DangGuillarmouRiviereShen2020}, length spectrum rigidity \cite{GL19, GPL25}, and semiclassical zeta functions \cite{FaureTsujii2015}.

A development of central significance for the present article is the study of dynamical resonances on locally symmetric spaces:
Given a rank one real semisimple Lie group \( G \) and a discrete torsion free subgroup \( \Gamma\subset G \), the geodesic flow on the locally symmetric space \( \Gamma\backslash G/K \) is given by the right \( A \cong \Rnum \)-action on \( \Gamma\backslash G/M \), where \( P=MAN \) is the Langlands decomposition of the minimal parabolic \( P\subset G \).
In this setting the geodesic flow can thus be studied via homogeneous dynamics and the algebraic structure has strong implications on the dynamical resonances such as an exact band structure and an exact correspondence to the spectrum of the Laplace--Beltrami operator on the locally symmetric space (often called quantum-classical correspondence, QCC).
We refer to \cite{DyatlovFaureGuillarmou2015} for the pioneering work and to \cite{GuillarmouHilgertWeich2018, DelarueWeich2021} for extensions to convex cocompact groups and general rank one locally symmetric spaces (see also \cite{Hil25} for a current review on this topic).
A natural question concerning these results are the generalizations to higher rank Lie groups \( G \).
In these cases, \( A \cong \Rnum^r \) is a higher dimensional abelian group and the right \( A \)-action on \( \Gamma\backslash G/M \), which is called the Weyl chamber flow, a higher rank Anosov action.
In \cite{BonthonneauGuillarmouHilgertWeich2024} dynamical resonances for higher rank Anosov actions on compact manifolds could be established.
A major challenge in this work was to merge tools from spectral theory of commuting operators and tools for the establishing of resonances.
It turned out that Taylor's framework for spectral theory of commuting operators is particularly useful as its homological approach allows the essential construction of parametrices that allow to transform the problem into a Fredholm problem.
This work quickly lead to applications of new SRB and Gibbs measures as well as equidistribution of torus orbits \cite{BonthonneauGuillarmouWeich_JDG, Humpert_Preprint} and quantum classical correspondence \cite{HWW23}.
Note furthermore that very recently, dynamical resonances could also be defined for higher rank actions related to Anosov representations \cite{BLW26}, building on a breakthrough on the existence of good domains of discontinuity by Delarue, Monclair and Sanders \cite{DMS25GAFA, DMS25}.

The present article aims to extend this theory when transitioning from real Lie groups \( G \) to \( p \)-adic Lie groups such as \( \mathrm{SL}_n(\Qnum_p) \).
The analogues of symmetric spaces for \( p \)-adic Lie groups are the Bruhat--Tits buildings of \( G \), which are simplicial complexes with a high degree of symmetry \cite{kalethaPrasad23}.
These Bruhat--Tits buildings, obtained from \( p \)-adic groups, possess a rich combinatorial structure of chambers, apartments, sectors, galleries, and so on.
This combinatorial structure has been successfully abstracted and axiomatized, leading to the more general framework of Euclidean buildings \cite{abramenkoBrown08}.
In this paper, we will work entirely within the axiomatic framework of Euclidean buildings and will not directly refer to \( p \)-adic groups, although this point of view remains a major motivation for this work. 

Indeed, the translation of analytic and dynamical questions from the smooth setting of symmetric spaces to the discrete geometry of buildings is an active and highly fruitful area of current research.
Notable recent examples include results on quantum ergodicity in the Benjamini--Schramm limit \cite{Pet23}, estimates on heat kernels \cite{AST25} and caloric functions \cite{PT25}, as well as the study of wave equations \cite{ART23}.
Regarding dynamical resonances, to the best of our knowledge, investigations have so far been restricted to the rank-one setting, namely graphs and trees.
In this context, the dynamics under consideration is the non-backtracking random walk.
While the spectral theory of these dynamics has long been understood in the framework of subshifts of finite type, it remains of considerable interest, particularly in the context of Ramanujan graphs and their mixing properties, with applications in computer science \cite{ABLS07, LP16, EFMP25}.
Moreover, aspects of the quantum--classical correspondence between the non-backtracking walk and the vertex averaging operator---an incarnation of the Ihara--Bass identity---have recently seen significant refinements \cite{BHW25, AHF25, AHF25b, AP26}.
Note that in \cite{LLP20} Lubetzky, Lubotzky, and Parzanchevski are also interested in a higher-rank generalization of non-backtracking random walks by introducing $j$ dimensional geodesic flows.
For a relation to the dynamics studied in this work see \thref{rem:llp20}.

Here is a brief description of our setting.
Let \( \euclbuild \) be a Euclidean building, which we will view as a topological space with a simplicial structure satisfying a specific list of axioms (see \thref{sec:EuclBuild}).
Any such building is associated with a Euclidean Coxeter complex \( \euclCox \), which is a tessellation of a Euclidean space \( \eucl \) into simplices by the action of an affine Weyl group \( \affWeyl \), a reflection group arising from a root system \( \rootSys \).
Note that the set of zero-dimensional simplices (the vertices) of this Coxeter complex contains the coweight lattice \( \coweightLat \subset \eucl \)\footnote{In the \( \widetilde{A}_n \) case, the coweight lattice even coincides with the vertices of the Coxeter complex.}.
Choosing a basis (a set of simple positive roots) \( \rootBasis \subset \rootSys \) defines a positive Weyl chamber, respectively a fundamental sector
\[
\sector_0 \coloneqq \Set{x\in\eucl \given \spr{\alpha,x} \geq 0 \text{ for all } \alpha \in \rootBasis}\subset \eucl,
\]
and this sector inherits a simplicial structure from \( \euclCox \).
In particular, \( \domCoweights\coloneqq \sector_0 \cap \coweightLat \) is an abelian monoid of \( \coweightLat \). 
As in \cite{BonthonneauGuillarmouHilgertWeich2024}, we want to work on compact spaces and therefore require that $\quot$ is a compact simplicial complex whose universal cover is a Euclidean building \( \euclbuild \), and that the covering map \( p \colon \euclbuild \to \quot \) is a type-preserving simplicial morphism. We call such a \( \quot \) a compact local building (see \thref{chap:genSetting} for an exact exposition of the geometric setting, as well as a definition of type-preserving morphisms). 
We can then define the space of sectors in \( \quot \) as
\[
\secQuot \coloneqq \Set{s \colon \sector_0 \to \quot \given \text{locally injective, type-rotating simplicial morphism}}.
\]
Note that this space is in complete analogy to the Weyl chamber flow described above:
for a locally symmetric space, the biquotient \( \Gamma \backslash G/M \) on which the Weyl chamber flow acts is in bijection with the space of local isometries of a positive Weyl chamber into \( \Gamma\backslash G/K \)\footnote{In a forthcoming paper that adopts the Lie-theoretic point of view, we will explain, among other things, that for buildings \( \euclbuild \) derived from \( p \)-adic Lie groups, the space \( \secQuot \) is in bijection with a biquotient analogous to \( \Gamma\backslash G/M \).}.

Now, a natural dynamic is described by the action of the dominant coweights.
Note that \( \domCoweights \) acts on \( \euclCox \) by type-rotating morphisms and maps the fundamental sector \( \sector_0 \) into itself.
We can thus, for any $\mu \in \domCoweights$, define the shift
\[
\sigma_\mu \colon \secQuot \to \secQuot,\; \secOnQuot \mapsto \secOnQuot(\bullet + \mu).
\]
This provides us with an action of the abelian monoid \( \domCoweights \) on the space \( \secQuot \), whose spectral theory is our principal object of study.
Note furthermore very similar dynamical systems were studied by Mozes in \cite{mozes95} order to obtain results on Cartan actions on buildings.
While Mozes studied multidimensional subshifts (which implicitly also play a role in our work) he did not study dynamical resonances.

In order to define dynamical resonances, we must endow this space with some structure.
A major difference, compared to the case of real symmetric spaces, is that \( \secQuot \) is no longer a manifold.
However, it can be endowed with a family of natural metrics \( \secMetricFunc \), which are ultrametrics and endow the space \( \secQuot \) with a Cantor-like topology.
The metric leads, in turn, to well-defined Banach spaces \( \lipschitzFunc \) of complex-valued Lipschitz functions on the metric space \( (\secQuot, \secMetricFunc) \).
It is this space that takes the role of the anisotropic Sobolev spaces defined via microlocal analysis in \cite{BonthonneauGuillarmouHilgertWeich2024}.
For any \( \mu\in \domCoweights \), we can define a transfer operator
\[
\lipschitzFunc \ni \varphi \mapsto \tOp_\mu \varphi(\secOnQuot) \coloneqq \frac{1}{M_\mu} \sum_{\secOnQuot^\prime \in \sigma^{-1}_\mu(\secOnQuot)} \varphi(\secOnQuot^\prime).
\]
By the commutativity of the \( \domCoweights \)-action, this yields a family of commuting operators on \( \lipschitzFunc \).
What remains to be discussed is the space in which we naturally define a spectrum of dynamical resonances in this setting.
Note that \( \domCoweights \) is freely generated by the fundamental weights \( \varpi_1,\ldots, \varpi_n \) associated with the basis \( \rootBasis \subset \rootSys \), so we obtain \( n \) commuting operators \( \tOp_{\varpi_1},\ldots, \tOp_{\varpi_n} \), and a joint spectrum is naturally expressed in terms of tuples in \( \Cnum^n \).

As this tuple might depend on choices, it is more convenient to express a spectrum as a character \( \chi\in \Hom_{\Cnum\text{-}\mathrm{alg.}} (\Cnum[\domCoweights],\Cnum) \), where \( \Cnum[\domCoweights] \) is the semigroup algebra of \( \domCoweights \).
Such a character associates to any \( \mu\in \domCoweights \) a spectral value \( \chi(\mu) \) for the operator \( \tOp_\mu \).

As in \cite{BonthonneauGuillarmouHilgertWeich2024}, we want to study the Taylor spectrum of the family of operators \( \tOp_\mu \), and we refer to \thref{sec:SpecTheory} for an introduction to Taylor spectra.
So far, we only emphasize that, in general, the Taylor spectrum \( \sigma_{\mathrm{T}}(\tOp) \) of \( \tOp \coloneqq (\tOp_{\varpi_1},\ldots, \tOp_{\varpi_n}) \) is not discrete and far from being fully described by the spectrum of joint eigenvalues, which is given by
\[
\sigma_{\mathrm{p}}(\tOp) \coloneqq \Set{\chi\in \Hom_{\Cnum\text{-alg.}} (\Cnum[\domCoweights],\Cnum)\given \exists\varphi\in\lipschitzFunc\setminus \{0\}, \ \tOp_\mu\varphi = \chi(\mu) \varphi \ \forall \mu\in\domCoweights}.
\]
Our main theorem, however, establishes that for characters large enough with respect to the parameter \( \vartheta \) (which can be chosen arbitrarily small to make this a very weak condition), the entire Taylor spectrum is given by the discrete eigenvalue spectrum of \( \tOp_\mu \) on \( \lipschitzFunc \):
\begin{thm}
	Let \( \quot \) be a compact local building.
	Then, for any $ 0 < \vartheta < 1 $, we have
	\[
	\sigma_{\mathrm{T}}(\tOp) \cap \Set{\chi \in \Hom_{\Cnum\text{-}\mathrm{alg.}}(\Cnum[\domCoweights],\Cnum) \given \abs{\chi} \geq \vartheta} \subset \sigma_{\mathrm{p}}(\tOp)
	\]
	where we define
	\[
	\abs{\chi} \coloneqq \sup \Set{\abs{\chi(\mu)}\given \mu\in \coweightLat\text{ with } \spr{\alpha,\mu} > 0 \text{ for all } \alpha \in \rootBasis}.
	\]
\end{thm}
We thus have obtained a discrete set of dynamical resonances that are an intrinsic spectral invariant associated to the \( \domCoweights \)-action on \( \secQuot \) for any compact local building \( \quot \).
We obtain even much stronger statements that also imply the finiteness of all joint eigenspaces and exactly relate the Taylor cohomologies to the cohomology on a finite-dimensional complex.
As these statements, however, require a precise introduction of the Taylor complexes, we refrain from stating them in the introduction and refer to \thref{sec:SpecTheory} for a statement of these results.

Let us shortly outline the article:
In \thref{chap:genSetting}, we recall the foundational definitions of simplicial structures, labeled chamber systems, and their morphisms.
\thref{sec:ECC} and \thref{sec:EuclBuild} review the necessary geometric background on Euclidean and spherical Coxeter complexes, root systems, and buildings, formalizing the notions of sectors and type-rotating morphisms.
In \thref{sec:dynaSetting}, we introduce the primary dynamical system: we define the space of sectors \( \secQuot \) on a compact local building, equip it with a natural ultrametric, and define the shift dynamics driven by the monoid of dominant coweights.
\thref{sec:Transfer Operator} is dedicated to the analytic properties of the transfer operators associated with these shifts. We define appropriate Banach spaces of Lipschitz functions, establish a crucial Lasota--Yorke-type key inequality, and prove that transfer operators corresponding to strongly dominant coweights are quasi-compact.
Finally, in \thref{sec:SpecTheory}, we merge these analytic results with Taylor's homological spectral theory for commuting operators. By constructing suitable parametrices, we bound the essential Taylor spectrum and prove our main theorem, establishing that the dynamical resonances outside a strictly bounded radius are given precisely by a discrete joint eigenvalue spectrum.

\textbf{Acknowledgements} We thank Dominik Brennecken, Carsten Peterson and Kai-Uwe Bux for numerous stimulating discussions about buildings and Margit Rösler and Lukas Langen for helpful suggestions on the spectral theory of Banach algebras. We furthermore thank Nikhil Srivastava for encouraging discussions and hints to the literature about higher rank analoga of non-backtracking random walks.
All three authors acknowledge support from the Deutsche Forschungsgemeinschaft (DFG) Grant No. SFB-TRR 358/1 2023 - 491392403 (CRC \quoted{Integral Structures in Geometry and Representation Theory}).

\section{General Setting}\label{chap:genSetting}

\subsection{Simplicial Complexes}\label{sec:simpCompl}

The definitions for the simplices are taken from \cite[Section 2.1]{hatcher02}.

For \( n \in \Nnum_0 \) the \emph{standard \( n \)-simplex}, denoted \( \stdSimplex^n \) is the set
\[
\stdSimplex^n \coloneqq \Set*{(t_0,\ldots,t_n) \in \Rnum^{n+1} \given t_i \geq 0 \text{ and } \sum_{i=0}^n t_i = 1} \subseteq \Rnum^{n+1}
\]
together with the subspace topology. A \emph{face} of \( \stdSimplex^n \) is a subset \( F \) defined by a subset of indices \( K \subseteq \Set{0,\ldots,n} \) of the form
\[
F = \Set*{(t_0,\ldots,t_n) \in \Rnum^{n+1} \given t_i \geq 0, \; \sum_{i=0}^n t_i = 1 \text{ and } t_k = 0 \text{ for all } k \in K}.
\]

Obviously any face \( F \) of \( \stdSimplex^n \) is linearly homeomorphic to \( F \cong \stdSimplex^\ell \), where \( \ell = n - \text{\quoted{number of indices equal zero}} \).
The faces linearly homeomorphic to \( \stdSimplex^{n-1} \) of \( \stdSimplex^n \) are called \emph{panels} and the union of all panels of \( \stdSimplex^n \) is the \emph{boundary} of \( \stdSimplex^n \), written \( \partial \stdSimplex^n \).

The \emph{open simplex} \( \mathring{\stdSimplex}^n \coloneqq \stdSimplex^n \setminus \partial \stdSimplex^n \) is the interior of \( \stdSimplex^n \).

\begin{defi}[Simplicial structure]\label{defi:simplStruc}
	A \emph{simplicial structure} on a topological space \( \topolSpace \) is a collection of continuous maps \( (\simplStrucMaps_\alpha \colon \stdSimplex^{n_\alpha} \to \topolSpace)_{\alpha \in J} \) such that
	\begin{enumerate}[(i)]
		\item
			the restriction \( \restr{\simplStrucMaps_\alpha}{\mathring{\stdSimplex}^{n_\alpha}} \) is injective, and each point of \( \topolSpace \) is in the image of exactly one such restriction \( \restr{\simplStrucMaps_\alpha}{\mathring{\stdSimplex}^{n_\alpha}} \).
		\item
			each restriction of \( \simplStrucMaps_\alpha \) to a panel of \( \stdSimplex^{n_\alpha} \) is one of the maps \( \simplStrucMaps_\beta \colon \stdSimplex^{n_\beta} \to \topolSpace \) with \( n_\beta = n_\alpha - 1 \).
			Here we are identifying the panel of \( \stdSimplex^{n_\alpha} \) with \( \stdSimplex^{n_\alpha-1} \) by the canonical linear homeomorphism between them.
			
			Consequently each restriction of \( \simplStrucMaps_\alpha \) to a face of \( \stdSimplex^{n_\alpha} \) is one of the maps \( \simplStrucMaps_\beta \) with \( \beta \in J \).
			In this case we call \( \simplStrucMaps_\beta \) a \emph{face} of \( \simplStrucMaps_\alpha \) and write \( \simplStrucMaps_\beta \leq \simplStrucMaps_\alpha \).
			This defines a partial ordering.
		\item
			for any \( \simplStrucMaps_\alpha, \simplStrucMaps_\beta \) with \( \Image \simplStrucMaps_\alpha \cap \Image \simplStrucMaps_\beta \neq \emptyset \).
			We find a map \( \simplStrucMaps_\gamma \colon \stdSimplex^{n_\gamma} \to \topolSpace \) with \( \Image \simplStrucMaps_\gamma = \Image \simplStrucMaps_\alpha \cap \Image \simplStrucMaps_\beta \)
			and \( \simplStrucMaps_\gamma \) is a face of both \( \simplStrucMaps_\alpha \) and \( \simplStrucMaps_\beta \). By abuse of notation we will also write $\simplStrucMaps_\gamma=\simplStrucMaps_\alpha\cap \simplStrucMaps_\beta$.
		\item
			a set \( A \subseteq \topolSpace \) is open if and only if \( \simplStrucMaps_\alpha^{-1}(A) \) is open in \( \stdSimplex^{n_\alpha} \) for each \( \simplStrucMaps_\alpha \).
	\end{enumerate}
	We will usually write \( (\topolSpace, (\simplStrucMaps_\alpha)) \) for a simplicial structure.
	
	Furthermore, we introduce the following terms:
	\begin{enumerate}[(a)]
		\item
			We call a map \( \simplStrucMaps_\alpha \) \emph{simplex} and write \( \simplexSet(\topolSpace, (\simplStrucMaps_\alpha)) \coloneqq \simplexSet(\topolSpace) \coloneqq \simplexSet \) for the set of all simplices.
		\item
			If a simplex \( \simplStrucMaps \) is no face of any simplex we call it a \emph{maximal simplex}.
			This is equivalent to being maximal with respect to the above defined partial ordering.
		\item
			The \emph{dimension} \( \dim \simplStrucMaps \) of a simplex \( \simplStrucMaps \colon \stdSimplex^{n} \to X \) is \( n \).
			
			For the set of all simplices of \( (\topolSpace, (\simplStrucMaps_\alpha)) \) of a fixed dimension \( \ell \) we write
			\[
			\simplexSet_\ell(\topolSpace, (\simplStrucMaps_\alpha)) \coloneqq \Set{\simplStrucMaps \given \dim \simplStrucMaps = \ell}.
			\]
			We call the elements of \( \simplexSet_0(\topolSpace, (\simplStrucMaps_\alpha)) \) \emph{vertices}.
		\item
			Let \( \simplStrucMaps \colon \stdSimplex^{n} \to \topolSpace \) be a simplex.
			We define the \emph{vertex set} \( \simplexSet_0(\simplStrucMaps) \) of \( \simplStrucMaps \) as the set of all dimension \( 0 \) faces of \( \simplStrucMaps \), i.e.,
			\[
			\simplexSet_0(\simplStrucMaps) \coloneqq \Set{\simplStrucMaps^\prime \given \simplStrucMaps^\prime \leq \simplStrucMaps} \cap \simplexSet_0(\topolSpace, (\simplStrucMaps_\alpha)).
			\]
	\end{enumerate}
\end{defi}

\begin{rem}\label{rem:simplexDescrByVertizes}
	One can check that every simplex \( \simplStrucMaps \) is uniquely determined by its vertex set \( \simplexSet_0(\simplStrucMaps) \).
	That means we have \( \simplexSet_0(\simplStrucMaps_\alpha) = \simplexSet_0(\simplStrucMaps_\beta) \) if and only if \( \simplStrucMaps_\alpha = \simplStrucMaps_\beta \).
	
	As a consequence we can conclude that every simplicial structure gives us an (abstract) simplicial complex as defined in \cite[Section 3.1]{spanier82}.
	To see this we take the vertex set \( \simplexSet_0(\topolSpace, (\simplStrucMaps_\alpha)) \) and the simplices are given by \( \Set{\simplexSet_0(\simplStrucMaps_\alpha) \given \alpha \in J} \).
	
	Every (abstract) simplicial complex has a geometric realization, 
	the construction of which can be found in \cite[Section 2.1]{hatcher02} or \cite[Section 3.1]{spanier82}.
	As one can see the geometric realization agrees with our definition of a simplicial structure.
	The benefit of our approach is that we can use all the terminology and results from the theory of (abstract) simplicial complexes and apply it to our setting.
	Later on we want to use coverings and covering spaces.
	In this setting it is useful to have the topology in the background to apply the results from the covering theory of topological spaces (see \cite[Section 1.3]{hatcher02}).
\end{rem}

\begin{rem}\label{rem:namingAndConventions}
	We use the same naming conventions for \( \simplStrucMaps_\alpha \) and \( \Image \simplStrucMaps_\alpha \).
	Moreover for the simplex sets \( \simplexSet_\ell(\topolSpace, (\simplStrucMaps_\alpha)) \) we write \( \simplexSet_\ell \) whenever the simplicial structure is clear from the context.
	
	Moreover, we want to view \( \simplexSet_0 \) as a subset of \( \topolSpace \).
	This can be done by identifying the one element set \( \Image \simplStrucMaps \subseteq \topolSpace \) with its element (\( \simplStrucMaps \in \simplexSet_0 \)).
		Hence we can view \( \simplexSet_\ell \) as an element of the power set of \( \simplexSet_0 \).
\end{rem}

\begin{defi}[Simplicial Substructure]
	Let \( (\topolSpace, (\simplStrucMaps_\alpha)_{\alpha \in J}) \) be a simplicial structure.
	A \emph{simplicial substructure} is a simplicial structure \( (\topolSpace^\prime, (\simplStrucMaps_\beta)_{\beta \in K}) \) such that
	\begin{enumerate}[(i)]
		\item
			\( \topolSpace^\prime \subseteq \topolSpace \), and the topology on $\topolSpace^\prime$ coincides with the subspace topology and
		\item
			\( (\simplStrucMaps_\beta)_{\beta \in K} \subseteq (\simplStrucMaps_\alpha)_{\alpha \in J} \).
	\end{enumerate}
	In this case we write \( (\topolSpace^\prime, (\simplStrucMaps_\beta)_{\beta \in K}) \subseteq (\topolSpace, (\simplStrucMaps_\alpha)_{\alpha \in J}) \).
\end{defi}

\begin{defi}[Morphisms of Simplicial Structures]\label{defi:morphSimplStruc}
	Let \( (\topolSpace, (\simplStrucMaps_\alpha)_{\alpha \in J}) \) and \( (\topolSpace^\prime, (\simplStrucMaps^\prime_\beta)_{\beta \in K}) \) be two simplicial structures.
	A continuous map \( \varphi \colon \topolSpace \to \topolSpace^\prime \) is a \emph{morphism of simplicial structures} if
			for all simplices \( \simplStrucMaps_\alpha \) we have \( \varphi \circ \simplStrucMaps_\alpha \in \simplexSet(\topolSpace^\prime, (\simplStrucMaps^\prime_\beta)) \).
			In particular, the restriction \( \restr{\varphi}{\simplexSet_0(\topolSpace, (\simplStrucMaps_\alpha))} \colon \simplexSet_0(\topolSpace, (\simplStrucMaps_\alpha)) \to \simplexSet_0(\topolSpace^\prime, (\simplStrucMaps^\prime_\beta)) \) induces a map on the vertices.
\end{defi}

\begin{rem}
	By comparing \thref{defi:morphSimplStruc} to the definition of a simplicial map in \cite[Section 3.1]{spanier82} one sees that every morphism of simplicial structures gives a morphism between the underlying abstract simplicial complexes.
	But we could also start with a morphism between the abstract simplicial complexes in the sense of \cite{spanier82} and obtain a morphism of simplicial structures.
	The only thing that is not immediate is the continuity of the map.
	In \cite[p.~110]{spanier82} it is explained that the construction of the geometric realization gives a covariant functor which maps the simplicial maps to continuous maps.
	As explained in \cite[Theorem 3.1 and the definition following it]{munkres84} any two such realizations are homeomorphic.
	
	Thus whenever needed we can specify all our data combinatorically using (abstract) simplicial complexes.
	Moreover, whenever we want to define a morphism between two simplicial structures we only need to check that it maps vertices to vertices and the image of any simplex is again a simplex of the combinatorial complex.
\end{rem}

From now on we assume that the simplices of our simplicial structures have bounded dimension.
Hence every simplex is contained in a maximal simplex.
Moreover, we want to assume that all maximal simplices of the simplicial structure \( (\topolSpace, (\simplStrucMaps_\alpha)) \) have the same dimension.
With this convention we introduce the following terminology.

\begin{defi}
	Let \( (\topolSpace , (\simplStrucMaps_\alpha))\) be a simplicial structure.
	\begin{enumerate}[(a)]
		\item
			We call a simplex of maximal dimension a \emph{chamber}.
		\item
			The \emph{rank} of the simplicial structure \( (\topolSpace, (\simplStrucMaps_\alpha)) \) is \( \rank (\topolSpace, (\simplStrucMaps_\alpha)) = \max \dim \simplStrucMaps_\alpha \).
			As our simplicial structure is bounded this value is finite and equal to the dimension of a chamber.
		\item
			A simplex \( \simplStrucMaps \) of codimension one, i.e., \( \dim \simplStrucMaps = \rank (\topolSpace,(\simplStrucMaps_\alpha)) - 1 \), is called a \emph{panel}.
	\end{enumerate}
\end{defi}

\begin{defi}[Labeling for Simplicial Structures]\label{def:labsimpcomp}
	A \emph{labeled simplicial structure} is a simplicial structure \( (\topolSpace, (\simplStrucMaps_\alpha)) \) equipped with a finite set \( I \) of \emph{types}, and a \emph{type map} \( \labeling \colon \simplexSet_0 \to I \), such that, for all chambers \( C \)
	\[
	\restr{\labeling}{C} \colon \simplexSet_0(C) \to I
	\]
	is bijective.
	Here we view \( \simplexSet_0 \) as a subset of \( \topolSpace \) as described in \thref{rem:namingAndConventions}.
	
	The \emph{type of a simplex} \( s \) is \( \labeling(\simplexSet_0(s)) \subset I\), and the \emph{cotype} of \( s \) is \( I \setminus \labeling(\simplexSet_0(s)) \).
\end{defi}

From now on we want \( I \) to be a finite index set.

\begin{defi}[Labeled Simplicial Substructures]
	Let \( (\topolSpace, (\simplStrucMaps_\alpha), \labeling) \) be a simplicial structure.
	A labeled simplicial structure \( (\topolSpace^\prime, (\simplStrucMaps_\beta), \labeling^\prime) \) is a \emph{labeled simplicial substructure} if \( (\topolSpace^\prime, (\simplStrucMaps_\beta)) \) is a simplicial substructure of \( (\topolSpace, (\simplStrucMaps_\alpha)) \) and the restriction of \( \labeling \) to \( (\topolSpace^\prime, (\simplStrucMaps_\beta)) \) agrees with \( \labeling^\prime \).
\end{defi}

\begin{defi}[Morphism of Labeled Simplicial Structures]
	Let \( (\topolSpace, (\simplStrucMaps_\alpha)_{\alpha \in J}, \labeling) \) and \( (\topolSpace^\prime, (\simplStrucMaps^\prime_\beta)_{\beta \in K}, \labeling^\prime) \) be two labeled simplicial structures over the same index set \( I \).
	A continuous map \( \varphi \colon \topolSpace \to \topolSpace^\prime \) is a \emph{morphism of labeled simplicial structures} if it is a morphism between the simplicial structures \( (\topolSpace, (\simplStrucMaps_\alpha)) \) and \( (\topolSpace^\prime, (\simplStrucMaps^\prime_\beta)) \) and \( \labeling = \labeling^\prime \circ \varphi \).
\end{defi}

Obviously in a labeled simplicial structure all maximal simplices have the same dimension \( \# I - 1 \).

For the rest of this section, if not otherwise noted, the index set is \( I = \Set{0,1,\ldots,n} \).
Hence \( \simplexSet_n \) is the set of chambers.

Whenever the simplicial structure and its labeling is clear from the context we write \( \topolSpace \) for our simplicial structure instead of \( (\topolSpace, (\simplStrucMaps_\alpha), \labeling) \) and \( \chambset(\topolSpace) \coloneqq \simplexSet_n \) for the set of all chambers of \( \topolSpace \).
While we use capital letters for chambers we will use small letters for simplices of arbitrary dimension.

For a set \( M \) we denote by \( F(M) \) the free monoid over \( M \).

\begin{defi}[Galleries and Residues]\label{defi:terminologyChamberSys}
	Let \( (\topolSpace, (\simplStrucMaps_\alpha), \labeling) \) be a labeled simplicial structure.
	\begin{enumerate}[(a)]
		\item\label{item:iEquivChambers}
			Let \( i \in I \).
			We call two chambers \( C,D \in \chambset(\topolSpace) \) \emph{\( i \)-equivalent} if \( I \setminus \Set{i} \subseteq \labeling(\simplexSet_0(C) \cap \simplexSet_0(D)) \).
			In this case we write \( C \sim_i D \).
			If \( C \sim_i D \) and \( C \neq D \) we say that \( C \) and \( D \) are \emph{\( i \)-adjacent}. This is equivalent to the assumption that \( C \cap D \) has cotype \( i \). 
		\item
			A \emph{gallery} \( \Gamma = (C_0,\ldots,C_k) \) is a finite sequence of chambers in \( \chambset(\topolSpace) \) together with a word \( w(\Gamma) = i_1i_2 \ldots i_{k} \in F(I) \), such that \( C_{j-1} \sim_{i_j} C_j \).
			Sometimes we call a sequence of this kind \emph{pregallery} and only speak about a gallery if \( C_{j-1} \neq C_j \) for all \( 1 \leq j \leq k \).
			In this case we also say the gallery is \emph{non-stuttering}, whereas if \( C_{j-1} = C_j \) for some \( j \) we say the gallery is \emph{stuttering}.
			We say a gallery \emph{connects} \( C \) and \( D \) if its extremities are \( C_0 = C \) and \( C_k = D \) or \( C_0 = D \) and \( C_k = C \).
		\item
			The \emph{type} of a gallery \( \Gamma \) is the word \( w(\Gamma) \).
		\item
			The \emph{length} of a gallery \( \Gamma \) is \( k \).
		\item
			A gallery \( \Gamma = (C_1,\ldots,C_k) \) is called \emph{minimal} if the length of the gallery is minimal for all galleries connecting \( C_1 \) and \( C_k \).
		\item
			For a subset \( J \subseteq I \) we call a gallery \( \Gamma \) a \emph{\( J \)-gallery} if \( w(\Gamma) \in F(J) \).
		\item
			Let \( J \subseteq I \) and \( C \in \chambSys \) be a chamber.
			The \emph{\( J \)-residue} of \( C \) is the set
			\[
			\Res_J(C) \coloneqq \Set{D \in \chambset(\topolSpace) \given \text{there is a \( J \)-gallery from \( C \) to \( D \)}}.
			\]
		\item
			We call \( \chambset(\topolSpace) \) \emph{gallery-connected} if for any two chambers there exists a gallery connecting \( C \) and \( D \).
		\item
			We call \( \chambset(\topolSpace) \) \emph{thin} if for each \( i \in I \) and every chamber \( C \) there exists exactly one chamber \( i \)-adjacent to \( C \).
		\item
			We call \( \chambset(\topolSpace) \) \emph{thick} if for each \( i \in I \) and every chamber \( C \) there exist at least two distinct chambers \( i \)-adjacent to \( C \).
	\end{enumerate}
\end{defi}

\begin{defi}[Chamber System]
	We call a finite dimensional gallery-connected labeled simplicial structure \( (\topolSpace, (\simplStrucMaps_\alpha), \labeling) \) a \emph{chamber system over \( I \)}.
	We denote this chamber system by \( (\chambset(\topolSpace), \labeling) \) or \( (\chambset(\topolSpace), I) \).
\end{defi}

\begin{rem}\label{rem:abstract_chamb1}
		There is an abstract notion of a \emph{chamber system} which is closely related but not equivalent to labeled simplicial complexes or in our case labeled simplicial structures.
		The gap can be closed by adding more assumptions on the simplicial complexes as well as the chamber systems.
		As every building and every Coxeter complex satisfy these conditions we do not go into the details here and refer to \cite[Proposition 1.4]{cartwright97}.
		Later on in the paper we are looking at chamber systems which are quotients of Euclidean buildings and therefore satisfy the extra conditions as well.
		Nevertheless, using the basic notation of chamber complexes already now is useful as it allows intuitive descriptions.
\end{rem}

\begin{rem}\label{rem:onlyChambAndEquiRelImp}
	One important fact about chamber systems is that the data needed to describe it uniquely is the set of chambers as well as the equivalence relations for the index set \( I \).
	In our setting the chamber set is \( \chambset(\topolSpace) \) and the \( i \)-equivalence relations are described by the labeling map \( \labeling \).
	
	As we only need the set of chambers we often write \( \chambSys \) for a chamber system over \( I \) and use the equivalence relations implicitly.
\end{rem}

\begin{rem}
	Let \( \chambSys \) be a chamber system over \( I \), \( J \subseteq I \) and \( C \) a chamber in \( \chambSys \).
	We can equip the \( J \)-residue \( \Res_J(C) \) of \( C \) with the structure of a chamber system over \( J \).
\end{rem}

\begin{defi}[Subcomplex]\label{defi:subcomplex}
	Let \( \chambSys \) be a chamber system over \( I \).
	A labeled simplicial substructure \( \mathcal D \subseteq \chambSys \) is called a \emph{subcomplex} if \( \mathcal D \) is a chamber system of the same dimension as \( \chambSys \).
\end{defi}

\begin{defi}[Types of Morphisms of Chamber Systems]\label{defi:morphiOfCS}
	Let \( \chambSys \) and \( \mathcal D \) be two chamber systems over \( I \).
	We call a morphism of simplicial structures \( \psi \colon \chambSys \to \mathcal D \) that takes chambers to chambers
	\begin{enumerate}[(a)]
		\item\label{item:morphismOfCS}
			a \emph{morphism of chamber systems} if there exists a permutation \( \sigma \in S(I) \) such that for two chambers \( C,D \in \chambSys \) with \( C \sim_i D \) the image is \( \sigma(i) \)-equivalent \( \psi(C) \sim_{\sigma(i)} \psi(D) \).
			
			We write \( \Hom(\chambSys, \mathcal D) \) for the set of a morphisms of chamber systems between \( \chambSys \) and \( \mathcal D \).
		\item
			an \emph{isomorphism of chamber systems} if there exists a permutation \( \sigma \in S(I) \) such that \( \psi \) is a bijection satisfying \( C \sim_i D \) if and only if \( \psi(C) \sim_{\sigma(i)} \psi(D) \).
			
			If \( \mathcal D = \chambSys \) we call an isomorphism \emph{automorphism} and write \( \Aut(\chambSys) \) for the group of all automorphisms.
		\item
			a \emph{type-preserving morphism of chamber systems} if \( \psi \) is a morphism of chamber systems for which we can choose \( \sigma = \id_I \).
			
			We write \( \Hom_0(\chambSys, \mathcal D) \) for the set of type-preserving morphisms.
		\item
			a \emph{type-preserving isomorphism of chamber systems} if \( \psi \) is an isomorphism of chamber systems for which we can choose \( \sigma = \id_I \).
			
			If \( \mathcal D = \chambSys \) we write \( \Aut_0(\chambSys) \) for the group of all type-preserving automorphisms of the chamber system \( \chambSys \).
	\end{enumerate}
\end{defi}

\begin{rem}\label{rem:chamberMorphisms}
	\mbox{}
	\begin{enumerate}[(i)]
		\item
			According to \cite[Proposition A.14]{abramenkoBrown08}, for every morphism of simplicial structures \( \psi \) that preserves the chambers one can find a suitable permutation \( \sigma \in S(I) \).
			This ensures that our definition \thref{defi:morphiOfCS}\ref{item:morphismOfCS} covers the same morphisms as the literature.
		\item
			Given a simplicial morphism \( \psi \colon \chambSys \to \mathcal D \) between two chamber systems over \( I \) which preserves chambers there might be more than one permutation \( \sigma, \sigma^\prime \in S(I) \) such that \( \sigma \) and \( \sigma^\prime \) both describe the equivalence relation of the morphism.

			As an example, take a chamber morphism \( \chambSys \to \chambSys \) which maps all chambers to a fixed chamber.
		\item\label{item:uniquePremuOfMorph}
			Let \( \chambSys \) be a chamber system such that for every \( i \in I \) there exist distinct chambers \( C, D \in \chambSys \) such that \( C \sim_i D \).
			
			One can check that if a morphism of chamber systems \( \psi \colon \chambSys \to \mathcal D \) is injective on all residues of the form \( \Res_{\Set{i}}(C) \subseteq \chambSys \), i.e., \( \psi \) maps adjacent chambers to adjacent chambers, the permutation \( \sigma \) is uniquely determined.
			This condition is equivalent to \( \psi \) mapping galleries to galleries.
			
			To see this recall that the \( i \)-equivalence is given by the labeling (see \thref{defi:terminologyChamberSys}\ref{item:iEquivChambers}).
			Let \( \sigma, \sigma^\prime \in S(I) \) be two different permutations that describe the \( i \)-equivalence of the morphism \( \psi \), i.e., for any chambers \( C, D \in \chambSys \) and any index \( i \in I \) with \( C \sim_i D \) we have
			\[
			\psi(C) \sim_{\sigma(i)} \psi(D) \quad \text{and} \quad \psi(C) \sim_{\sigma^\prime(i)} \psi(D).
			\]
			As \( \sigma \neq \sigma^\prime \) we find an index \( j \in I \) such that \( \sigma(j) \neq \sigma^\prime(j) \).
			Let \( C, D \in \chambSys \) be two distinct chambers with \( C \sim_j D \).
			Their images are \( \sigma(j) \) and \( \sigma^\prime(j) \) adjacent, i.e.,
			\begin{align*}
				I \setminus \Set{\sigma(j)} &\subseteq \labeling(\simplexSet_0(\psi(C)) \cap \simplexSet_0(\psi(D))) \\
				\shortintertext{and}
				I \setminus \Set{\sigma^\prime(j)} &\subseteq \labeling(\simplexSet_0(\psi(C)) \cap \simplexSet_0(\psi(D))).
			\end{align*}
			As \( \sigma(j) \in I \setminus \Set{\sigma^\prime(j)} \) and \( \sigma^\prime(j) \in I \setminus \Set{\sigma(j)} \) we can conclude that \( I \subseteq \labeling(\simplexSet_0(\psi(C)) \cap \simplexSet_0(\psi(D))) \).
			That means the right hand side has \( \#I \) elements and therefore \( \psi(C) \) and \( \psi(D) \) agree on all vertices.
			Since in our setting (see \thref{rem:simplexDescrByVertizes}) a simplex is uniquely determined by its vertices, we have \( \psi(C) = \psi(D) \).
			Hence \( \psi \) is not injective on \( \Res_{\Set{j}}(C) \) which is a contradiction.
		\item
			If \( \mathcal D \subseteq \chambSys \) is a subcomplex the inclusion \( \mathcal D \hookrightarrow \chambSys \) is a type-preserving morphism of chamber systems.
	\end{enumerate}
\end{rem}

\begin{defi}[Combinatorial Convexity]\label{defi:combConvex}
	Let \( \chambSys \) be a chamber system over \( I \).
	\begin{enumerate}[(a)]
		\item
			We call a subset of chambers \( \mathcal D \subseteq \chambSys \) \emph{combinatorially convex} if for any two chambers \( C,D \in \mathcal D \) all chambers in any minimal gallery connecting \( C \) and \( D \) lie in \( \mathcal D \).
		\item
			Let \( C, D \) be two chambers in \( \chambSys \).
			The \emph{combinatorial convex hull} \( \conv \Set{C,D} \) of \( C, D \) is the smallest subset of \( \chambSys \) which contains all minimal galleries connecting \( C \) and \( D \).
	\end{enumerate}
\end{defi}

\section{Euclidean and Spherical Coxeter Complexes}\label{sec:ECC}
In this section we will introduce all the notions and notations of Euclidean and spherical Coxeter complexes that are necessary to describe the Weyl chamber flow dynamics in \Cref{sec:dynaSetting}.
We start by introducing general Coxeter complexes.

\subsection{Coxeter Systems}\label{subsec:Cox}

A reference for the basics on Coxeter groups and related topics is \cite[Chap.~IV]{bourbaki02}.
As before \( I \) is a finite set.

\begin{defi}[Coxeter Matrix]\label{defi:coxeterMat}
	For a finite index set \( I \) we call \( \coxMat \coloneqq (m_{i,j})_{i,j \in I} \) a \emph{Coxeter matrix} if for all \( i,j \in I \)
	\[
	m_{i,j} = m_{j,i} \in
	\begin{cases}
		\Nnum_{\geq 2} \cup \Set{\infty} & i \neq j ,\\
		\Set{1} & i = j.
	\end{cases}
	\]
\end{defi}

\begin{defi}[Coxeter Group]
	Let \( M = (m_{i,j})_{i,j \in I} \) be a Coxeter matrix and \( S \coloneqq \Set{s_i \given i \in I} \) with fixed symbols \( s_i \).
	\begin{enumerate}[(a)]
		\item
			A group \( W \) is a \emph{Coxeter group of type \( \coxMat \)} if it has a presentation
			\begin{equation}\label{eq:W-presentation}
				W = \spr{S \mid (s_is_j)^{m_{i,j}} = 1 \text{ for all } i,j \in I},
			\end{equation}
			following the usual notation for relations on free groups.
			Here \( m_{i,j} = \infty \) means there is no relation between \( s_i \) and \( s_j \) and is therefore omitted.
		\item
			The pair \( (W,S) \) is called a \emph{Coxeter system of type \( M \)}.
		\item
			For a subset \( J \subseteq I \) we write \( W_J \subseteq W \) for the subgroup of \( W \) generated by \( S_J \coloneqq \Set{s_j \given j \in J} \) and \( M_J \) for the Coxeter matrix obtained from \( M \) by deleting all rows and columns not in \( J \).
		\item
			We call a subset \( J \subseteq I \) \emph{spherical} if the subgroup \( W_J \subseteq W \) is finite.
	\end{enumerate}
\end{defi}

\begin{rem}\label{rem:Coxgroup}
	\mbox{}
	\begin{enumerate}[(i)]
		\item
			\( (W_J, S_J) \) is a Coxeter system of type \( M_J \).

	\end{enumerate}
\end{rem}

As mentioned in \thref{rem:onlyChambAndEquiRelImp} to define a chamber system over an index set \( I \) it is enough to define the set of chambers and the equivalence relations for each element in \( I \).
This leads to the following definition of a \emph{Coxeter complex}.

\begin{defi}[Coxeter Complexes]\label{defi:CoxeterComplex}
	Let \( (W,S) \) be a Coxeter system of type \( M \).
	A \emph{Coxeter complex} is a thin labeled chamber system \( (\chambSys, I) \), with a bijection \( \chambSys \cong W \).
	Furthermore the labeling of the chamber system has the property, that \( i \)-equivalence is given by
	\begin{equation}\label{eq:coxeter_adjacent}
		w \sim_i w \quad \text{and} \quad w \sim_i w s_i.
	\end{equation}
\end{defi}

\begin{rem}
	For every Coxeter system \( (W,S) \) \thref{eq:coxeter_adjacent} defines a notion of \( i \)-equivalence on the set \( W \), and with respect to this the set \( W \) is an abstract chamber system (cf.\ \thref{rem:abstract_chamb1,rem:onlyChambAndEquiRelImp}).
	As discussed in those remarks these abstract chamber systems are nice enough such that one can associate a simplicial structure in a unique way.
	Consequently for any Coxeter System \( (W,S) \) there is a (up to isomorphism) unique Coxeter complex (with all its properties of a labeled simplicial structure) which we denote by \( \coxCompl \).
\end{rem}

\begin{defi}\label{defi:sphericalResidue}
	For a spherical subset \( J \subseteq I \) we call the corresponding residue \( \Res_J(C) \subseteq \coxCompl \) \emph{spherical}.
\end{defi}

\begin{rem}\label{rem:sphericalResChambSys}
	\mbox{}
	\begin{enumerate}[(i)]
		\item
			All chamber systems and buildings we are looking at are indexed over some set \( I \) with a Coxeter group \( (W,S) \) in the background.
			By our convention \( S \) is indexed by \( I \).
			This convention allows us to extend the notion of \emph{spherical residues} to chamber systems \( (\chambSys, I) \) over \( I \).
		\item\label{item:linkOfVertex}
			Moreover, all residues of the form \( \Res_{I \setminus \Set{i}}(C) \) will be spherical in our setting.
			If \( x \in C \) is the type \( i \) vertex we call \( \Res_{I \setminus \Set{i}}(C) \) the \emph{link of \( x \)}.
	\end{enumerate}
\end{rem}

\subsection{Root Systems}

All Coxeter complexes that we will consider come from root systems.
We recall their construction and the essential notions.
For references and proofs see \cite[Chap.~VI]{bourbaki02}.

Let \( (\eucl, \spr{\cdot, \cdot}) \) be a finite-dimensional Euclidean vector space with scalar product \( \spr{\cdot,\cdot} \).
We denote the dimension of the Euclidean space by \( n = \dim \eucl \).
Let \( \rootSys \subseteq \eucl \) be a root system.
For \( \alpha \in \eucl \) we define \( \alpha^\vee \coloneqq \frac{2}{\spr{\alpha, \alpha}} \alpha \).

We quickly recall that \( \rootSys \) being a root system implies that \( \rootSys \) is finite and \( \rootSys \) spans \( \eucl \) as a vector space, i.e., \( \spr{\rootSys} = \eucl \).
An element \( \alpha \in \rootSys \) will be called a \emph{root} and for a root \( \alpha \in \rootSys \) the element \( \alpha^\vee \) is called a \emph{dual root}.

A root system \( \rootSys \) is \emph{reduced} if for every \( \alpha, \beta \in \rootSys \) which are linearly-dependend we have \( \alpha = \pm \beta \).
Else we say \( \rootSys \) is \emph{non-reduced}.
In the non-reduced case the only possible multiplies of a root \( \alpha \) are \( \pm \alpha \) and \( \pm 2 \alpha \).
Lastly we call a root system \emph{irreducible} if it cannot be decomposed into orthogonal parts.

\begin{rem}
	In \cite[Chap.~VI, \S 4]{bourbaki02} one can find a classification of all irreducible root systems.
	These are labeled by symbols \( A_n \) (\( n \geq 1 \)), \( B_n \) (\( n \geq 2 \)), \( C_n \) (\( n \geq 2 \)), \( D_n \) (\( n \geq 4 \)), \( E_6 \), \( E_7 \), \( E_8 \), \( F_4 \) and \( G_2 \).
	Except \( B_2 \) and \( C_2 \) non of these systems are isomorphic.
	
	Moreover, there is exactly one (up to isomorphism) irreducible non-reduced root system for each \( n \geq 1 \) given by \( BC_n \).
\end{rem}

For every root \( \alpha \in \rootSys \) and \( k \in \Znum \) we can define an affine hyperplane
\[
H_{\alpha, k} \coloneqq \Set{x \in \eucl \given \spr{\alpha, x} = k}.
\]
We define two collections of hyperplanes
\begin{align*}
	\mathcal H_0 &\coloneqq \Set{H_{\alpha,0} \given \alpha \in \rootSys} \\
	\mathcal H &\coloneqq \Set{H_{\alpha, k} \given \alpha \in \rootSys, k \in \Znum}
\end{align*}
as well as \emph{orthogonal reflections} along these hyperplanes
\[
s_{\alpha, k} \coloneq \eucl \to \eucl, \; x \mapsto x - \left( \spr{x, \alpha} - k \right) \alpha^\vee.
\]
By the definition of a root system the reflections \( s_{\alpha,0} \) leave \( \rootSys \) invariant.
To make the notation easier we write \( s_\alpha \coloneqq s_{\alpha,0} \) and \( H_{\alpha} \coloneqq H_{\alpha,0} \).

These hyperplane arrangements naturally lead to two simplicial complexes:
First, the collection \( \mathcal H \) tesselates the Euclidean space \( \eucl \) and equips it with a simplicial structure.
Second, the collection of codimension one subspaces \( \mathcal H_0 \) leads to a simplicial structure on the \( n-1 \)-dimensional sphere \( \mathbb{S}^{n-1} \).
Both simplicial structures turn out to be Coxeter complexes (see \thref{prop:euclCoxeterCompl}), the \emph{Euclidean Coxeter complex}, respectively the \emph{spherical Coxeter complex} associated to the root system \( \rootSys \).

In order to endow these simplicial complexes with a Coxeter complex structure we choose a basis \( \rootBasis \) for the root system \( \rootSys \).
Recall that a subset \( \rootBasis \subseteq \rootSys \) is a \emph{basis} of the root system \( \rootSys \) if the following hold
\begin{enumerate}[(i)]
	\item
		\( \rootBasis \subseteq \eucl \) is a basis for \( \eucl \),
	\item
		for all \( \alpha \in \rootSys \) we have
		\begin{equation}\label{equ:basisEqu}
			\alpha = \sum\limits_{\beta \in \rootBasis} c_\beta \beta,
		\end{equation}
		where the \( c_\beta \) are either all non-negative or all non-positive integers.
\end{enumerate}
The second condition allows us to speak about \emph{positive} and \emph{negative} roots.
A root \( \alpha \in \rootSys \) is \emph{positive} (with respect to the basis \( \rootBasis \)) if all \( c_\beta \) in \eqref{equ:basisEqu} are non-negative integers and \emph{negative} if all \( c_\beta \) are non-positive.
Given a root \( \alpha = \sum_{\beta \in \rootBasis} c_\beta \beta \), the \emph{height} of \(\alpha\) is defined by
\[
\mathrm{ht}(\alpha) = \sum\limits_{\beta \in \rootBasis} c_\beta.
\]
One can show that there exists a unique \emph{highest root}
\begin{equation}\label{equ:heighestRootRep}
	\check{\alpha} = \sum\limits_{\beta \in \rootBasis} m_\beta \beta \in \rootSys,
\end{equation}
with all \( m_\beta \geq 1 \).

\begin{defi}[Affine Weyl Group]
	Let \( \rootSys \) be a root system.
	We define the \emph{Weyl group} of \( \rootSys \) as
	\[
	\weylGr \coloneqq \weylGr(\rootSys) \coloneqq \spr{s_\alpha \mid \alpha \in \rootSys}_{\mathrm{group}}.
	\]
	The \emph{affine Weyl group} of \( \rootSys \) is the group
	\[
	\affWeyl \coloneqq \affWeyl(\rootSys) \coloneqq \spr{s_{\alpha, k} \mid \alpha \in \rootSys, k \in \Znum}_{\mathrm{group}}.
	\]
\end{defi}

One can deduce from the discussion of Euclidean reflection groups in \cite[Chapter 10]{abramenkoBrown08} that there are two fundamentally different kinds of Euclidean reflection groups, the finite and the infinite Euclidean reflection groups.
This is summarized in the following proposition.

\begin{prop}\label{prop:euclCoxeterCompl}
	Let \( \rootSys \) be an irreducible root system, \( \rootBasis \) a basis of \( \rootSys \). 
	We define \( s_0 \coloneqq s_{\check{\alpha}, 1} \), \( I_0 \coloneqq \Set{1,\ldots,n} \) and \( I \coloneqq I_0 \cup \Set{0} \).
	Moreover, let \( S_0 \coloneqq \Set{s_\beta \given \beta \in \rootBasis} \) and \( S \coloneqq S_0 \cup \Set{s_0} \).
	Then \( (\affWeyl, S) \) and \( (\weylGr, S_0) \) are Coxeter systems.
	Furthermore, the simplicial structure on \( \eucl \) induced by \( \mathcal H \) is a Coxeter complex for \( (\affWeyl, S) \) and the simplicial structure on the sphere \( \mathbb{S}^{n-1} \) induced by \( \mathcal H_0 \) is a Coxeter complex for \( (W_0,S_0) \).
\end{prop}

\begin{rem}
	Given an irreducible root system \( \rootSys \subseteq \eucl \) and a basis \( \rootBasis \) we denote the corresponding Euclidean Coxeter complex by \( \euclCox(\rootSys, \rootBasis) \) or shorter \( \euclCox(\rootSys) \) or \( \euclCox \) if \(\rootBasis\) and \(\rootSys\) are clear from the context.
	Note that the simplicial structure is determined by \( \rootSys \), whereas the basis \( \rootBasis \) defines the labeling of the complex (see \thref{rem:euclCoxEquivRel,rem:euclCoxAndCoweights}\ref{item:euclCoxLabeling}) as choosing a basis is equivalent to choosing a generating set \( S \) and an enumeration for it.
	
	As two of these choices lead to isomorphic complexes we usually do not specify the basis \( \rootBasis \).
\end{rem}

Note that in addition to the Euclidean topology on \( \euclCox \) the Euclidean scalar product also induces a metric on a Euclidean Coxeter complex \cite[Lemma 10.36]{abramenkoBrown08}.

Every hyperplane defines two closed halfspaces \( \mathcal H_{\alpha, k,+} \) and \( \mathcal H_{\alpha, k,-} \) defined by
\begin{align*}
	\mathcal H_{\alpha, k,+} &\coloneqq \Set{x \in \eucl \given \spr{\alpha, x} \geq k} \\
	\mathcal H_{\alpha, k,-} &\coloneqq \Set{x \in \eucl \given \spr{\alpha, x} \leq k}.
\end{align*}

\begin{defi}[Halfspace-Convexity]\label{defi:halfspaceConvex}
	We call a subset \( X \subseteq \Sigma \) of a Euclidean Coxeter complex \emph{halfspace-convex} if it is equal to the intersection of closed halfspaces.
\end{defi}

\begin{rem}\label{rem:notionsOfConvex}
	If \( X \subseteq \Sigma \) contains a chamber, then halfspace-convexity coincides with the definition of combinatorial convexity in \thref{defi:combConvex} (see \cite[Section 3.6.6 and Section 11.5]{abramenkoBrown08}).
	Moreover, from the discussion in \cite[Section 11.5]{abramenkoBrown08} it follows that any halfspace or combinatorially convex subset in a Euclidean Coxeter complex is convex in the Euclidean sense, i.e., it contains the lines connecting any two points in the set.
	Conversely, whenever a convex set $X\subset \eucl$ is additionally a union of simplices then this set is halfspace-convex.
\end{rem}

\begin{defi}[Fundamental Sector and Chamber]\label{defi:fundamentalSectorAndChamb}
	Let \( \rootSys \subseteq \eucl \) be an irreducible root system and \( \rootBasis = \Set{\beta_i \given i \in I_0} \) a basis.
	\begin{enumerate}[(a)]
		\item
			The \emph{fundamental chamber} is given by
			\begin{equation}\label{equ:fundaChamb}
				C_0 \coloneqq \Set{x \in \eucl \given \spr{x, \beta_i} \geq 0 \text{ for all } \beta_i \in \rootBasis \text{ and } \spr{x, \check{\alpha}} \leq 1}.
			\end{equation}
		\item
			The \emph{fundamental sector} is defined by
			\[
			\sector_0 \coloneqq \Set{x \in \eucl \given \spr{x, \beta_i} \geq 0 \text{ for all } \beta_i \in \rootBasis}.
			\]
	\end{enumerate}
\end{defi}

\begin{rem}\label{rem:euclCoxEquivRel}
	The chambers of \( \euclCox \) are the elements \( w C_0 \) for \( w \in \affWeyl \).
	So by \thref{defi:CoxeterComplex} the \( i \)-equivalence relation is given by
	\[
	w C_0 \sim_i w C_0 \quad \text{and} \quad w C_0 \sim_i w s_i C_0.
	\]
\end{rem}

In \cite[Chapter 10]{abramenkoBrown08} one can find that the following holds.
\begin{prop}\label{prop:weylGroupActions}
	Let \( \rootSys \subseteq \eucl \) be an irreducible root system.
	\begin{enumerate}[(i)]
		\item
			The fundamental sector \( \sector_0 \) is a fundamental domain for the action of the Weyl group \( \weylGr \) on \( \eucl \).
		\item
			\( \affWeyl \) acts simply transitively and by type-preserving morphisms on the chambers of \( \euclCox \) and \( C_0 \) is a fundamental domain for the action of \( \affWeyl \) on \( \eucl \).
		\item\label{item:weylGrActsOnZeroClosure}
			The Weyl group \( \weylGr \) acts simply transitively on the chambers having \( 0 \) in their closure.
	\end{enumerate}
\end{prop}

\subsection{Coroots and Coweights}

\begin{defi}[Coroot Lattice]
	Let \( \rootSys \subseteq \eucl \) be an irreducible root system.
	Recall that for a root \( \alpha \in \rootSys \) the corresponding dual root is given by \( \alpha^\vee \coloneqq \frac{2}{\spr{\alpha, \alpha}} \alpha \).
	The \emph{coroot lattice} \( \corootLat \) is the \( \Znum \)-span of the dual roots \( \Set{\alpha^\vee \given \alpha \in R} \).
\end{defi}

\begin{rem}
	The dual roots form a root system.
	Moreover, if \( \rootSys \) is an irreducible root system, the dual root system \( \corootLat \) is an irreducible reduced root system if and only if \( \rootSys \) is reduced.
\end{rem}

In \cite[Chap.~VI, \S 2, No.~1, Prop.~1]{bourbaki02} we find that there is a decomposition of the affine Weyl group as follows.

\begin{prop}
	Let \( \rootSys \subseteq \eucl \) be an irreducible root system.
	The affine Weyl group of \( \rootSys \) is the semi-direct product
	\[
	\affWeyl(\rootSys) = \corootLat \rtimes \weylGr(\rootSys).
	\]
\end{prop}

Another important lattice for the root system \( \rootSys \) is the \emph{coweight lattice} \( \coweightLat \).

\begin{defi}[Coweight Lattice]\label{defi:coweightLattice}
	Let \( \rootSys \subseteq \eucl \) be an irreducible root system and \( \rootBasis = \Set{\beta_i \given i \in I_0} \subseteq \rootSys \) a basis for \( \rootSys \).
	Let \( \Set{\varpi_i \given i \in I_0} \) be the dual basis, i.e., \( \spr{\varpi_i, \beta_j} = \delta_{i,j} \).
	We call the elements of the dual basis \( \varpi_i \) \emph{fundamental coweights}.
	The \( \Znum \)-module
	\[
	\coweightLat \coloneqq \spr{\varpi_1,\ldots,\varpi_n}_{\Znum} 
	\]
	is called the \emph{coweight lattice}.
	The coweight lattice is independent of the basis \( \rootBasis \).
	
	A coweight of the form \( \paramA = \sum_{i=1}^n a_i \varpi_i \) with \( a_i \in \Nnum_0 \) is called \emph{dominant}.
	We write \( \domCoweights \) for the set of all dominant coweights.
	The coweights of the form \( \paramA = \sum_{i=1}^n a_i \varpi_i \) with \( a_i \in \Nnum \) are called \emph{strongly dominant}.
	We write \( \innerCoweights \) for the set of all strongly dominant coweights.
	
	Moreover, we define a partial order \( \leq \) on \( \domCoweights \).
	We write \( \mu \leq \lambda \) if and only if \( \lambda - \mu \in \domCoweights \).
\end{defi}

\begin{rem}\label{rem:euclCoxAndCoweights}
	\mbox{}
	\begin{enumerate}[(i)]
		\item\label{item:euclCoxLabeling}
			Let \( \rootSys \subseteq \eucl \) be an irreducible root system.
			Let \( C_0 \in \chambset(\euclCox) \) be the fundamental chamber of the Euclidean Coxeter complex.
			The vertices of \( C_0 \) are given by (see \cite[Chap.~VI, \S 2, No.~2]{bourbaki02})
			\[
			\Set{0} \cup \Set*{\frac{\varpi_i}{m_i} \given i \in I_0},
			\]
			where the \( m_i \) are the multiplicities of basis elements \( \beta_i \) in the heighest root \eqref{equ:heighestRootRep}.
			The labeling of \( C_0 \) is given by
			\[
			\labeling(0) = 0 \quad \text{and} \quad \labeling\left(\frac{\varpi_i}{m_i}\right) = i.
			\]
			This labeling extends to a unique labeling on all of \( \euclCox \).
		\item
			From the definitions of the coroot and coweight lattice we observe that \( \corootLat \subseteq \coweightLat \).
		\item\label{item:domCoweightsInSector}
			From the definition of the fundamental sector \( \sector_0 \) (\thref{defi:fundamentalSectorAndChamb}) we see that \( \domCoweights \subseteq \sector_0 \).
	\end{enumerate}
\end{rem}

\begin{defi}[Extended Affine Weyl Group]
	Let \( \rootSys \subseteq \eucl \) be an irreducible root system and \( \weylGr \) its Weyl group.
	The group
	\[
	\extWeyl \coloneqq \coweightLat \rtimes \weylGr
	\]
	is the \emph{extended affine Weyl group}.
\end{defi}

\begin{rem}\label{rem:rootSysNotFromEuclCox}
	Given two root systems \( \rootSys, \rootSys^\prime \subseteq \eucl \) it is possible that \( \affWeyl(\rootSys) = \affWeyl(\rootSys^\prime) \) for example if \( \rootSys = C_n \) and \( \rootSys^\prime = BC_n \).
	The reason for this is that if we take all indivisible roots in \( BC_n \) we obtain a root system of type \( C_n \).
	As the affine Weyl group does not detect the different scaling the groups are the same.
	
	This is where the extended affine Weyl groups makes a difference as \( \extWeyl(C_n) \not\cong \extWeyl(BC_n) \).
\end{rem}

\begin{rem}
	As \( \corootLat \subseteq \coweightLat \) we have the following inclusions of subgroups
	\[
	\weylGr \leq \affWeyl \leq \extWeyl \leq \Aut(\euclCox).
	\]
	In other words extended affine Weyl group also acts on \( \euclCox \) and preserves the simplicial structure.
	But in contrast to \( \affWeyl = \Aut_0(\euclCox) \) not all morphisms coming from \(\extWeyl\) are type-preserving.
	We call the morphisms coming from \( \extWeyl \) \emph{type-rotating}.
\end{rem}

By our definition every \( \psi \in \Aut(\euclCox) \) induces a map \( \sigma \in S(I) \) and by \thref{rem:chamberMorphisms}\ref{item:uniquePremuOfMorph} this permuation is unique.
We use this observation for the following definition.

\begin{defi}[Type-rotating Permutation]\label{def:Aut_tr}
	Let \( \rootSys \subseteq \eucl \) be an irreducible root system and \( \rootBasis \) a basis.
	A permuation \( \sigma \in S(I) \) is called \emph{type-rotating} (for type \( \rootSys \) and basis \( \rootBasis \)) if it belongs to a type-rotating automorphism of the Euclidean Coxeter complex \( \euclCox(\rootSys) \).
	In other words, there exists an element \( \psi \in \extWeyl \) such that \( \sigma \) describes the change of the equivalence relations.
	
	We write \( \Aut_{\mathrm{tr}, \rootSys, \rootBasis}(I) \) for the type-rotating permutations belonging to the root system \( \rootSys \).
\end{defi}

\begin{rem}
	\mbox{}
	\begin{enumerate}[(i)]
		\item
			Note that the definition of a type-rotating permutation does not depend on the choice of the basis \( \rootBasis \) for the root system since any two bases are conjugate under \( \weylGr \).
			Hence we omit the basis \( \rootBasis \) and write \( \trAutRoot(I) \).
		\item
			Each type-rotating permutation \( \sigma \in \trAutRoot(I) \) induces an automorphism of Coxeter systems \( \sigma \colon (\affWeyl, S) \to (\affWeyl, S) \) via \( s_i \mapsto s_{\sigma(i)} \).
			One needs to check that we have \( m_{i,j} = m_{\sigma(i),\sigma(j)} \).
			This follows from the discussion in \cite[Section 3.5 \& 3.6]{parkinsonDiss} as our \( \trAutRoot(I) \) corresponds to the group \( \trAut(D) \) (see \cite[Proposition 3.6.1]{parkinsonDiss}).
		\item
			From the same discussion in \cite{parkinsonDiss} we can conclude that in the non-reduced case we have \( \trAutRoot(I) = \Set{\id} \) as \( \trAutRoot(I) \cong \coweightLat/\corootLat \) (see \cite[Chap.~VI, \S 2, No.~3]{bourbaki02} for the reduced case and \cite[Section 3.5]{parkinsonDiss} for the non-reduced case).
	\end{enumerate}
\end{rem}

\section{Euclidean and Spherical Buildings}\label{sec:EuclBuild}

\subsection{Buildings}

We work with the definition of a building as presented in \cite[Definition 4.1]{abramenkoBrown08}.

\begin{defi}[Building]\label{def:building}
	A \emph{building} is a labeled simplicial structure \( \build = (\topolSpace, (\simplStrucMaps_\alpha), \labeling) \) that can be expressed as the union of simplicial substructures \( \apartment \) (called \emph{apartments}) satisfying the following axioms
	\begin{enumerate}
		\myitem{(B0)}\label{item:buildAx0}
			each apartment is isomorphic to a Coxeter complex,
		\myitem{(B1)}\label{item:buildAx1}
			for any two simplices \( a, b \in \simplexSet(\build) \), there is an apartment \( \apartment \) such that \( a,b \in \apartment \).
		\myitem{(B2)}\label{item:buildAx2}
			if \( \apartment, \apartment^\prime \) are two apartments containing \( a \) and \( b \), then there is an isomorphism \( \apartment \to \apartment^\prime \) fixing \( a \) and \( b \) pointwise.
	\end{enumerate}
\end{defi}

\begin{rem}\label{rem:basicsOnBuildings}
	\mbox{}
	\begin{enumerate}[(i)]
		\item\label{item:allApartmentsIsom}
			An easy consequence of \ref{item:buildAx1} and \ref{item:buildAx2} is that any two apartments are isomorphic.
			
			The buildings we are interested in have apartments isomorphic to Euclidean Coxeter complexes hence we can speak of a \emph{Euclidean} respectively \emph{spherical building of type \( \rootSys \)} where \( \rootSys \) represents a root system, if the type of the Coxeter complex is a Euclidean, respectively spherical Coxeter complex for the root System \( \rootSys \).
		\item
			Any collection \( \apartSys \) of subcomplexes \( \apartment \) satisfying the axioms is called a \emph{system of apartments}.
			
			Throughout the paper we assume that \( \apartSys \) is the (unique) maximal system of apartments (see \cite[Theorem 4.54]{abramenkoBrown08}).
	\end{enumerate}
\end{rem}

Even though the following definitions can mostly be made for arbitrary buildings, we assume our buildings to be Euclidean buildings coming from an irreducible root system.
The central geometric object of this paper will be local buildings, defined as follows:

\begin{defi}[Local Building]\label{defi:localBuilding}
	We call a chamber system \( \chambSys \) over \( I \) a \emph{local building (of type \( \rootSys \))} if there exists a Euclidean building \( \euclbuild \) of type \( \rootSys \) and a type-preserving simplicial morphism \( p \colon \euclbuild \to \chambSys \) which is a topological covering.
\end{defi}

\begin{rem}\label{rem:localBijSphericalRes}
	Let \( p \colon \euclbuild \to \chambSys \) be a type-preserving simplicial morphism which is also a topological covering.
	Then \( p \) restricted to spherical residues of \( \euclbuild \) is an isomorphism.
\end{rem}

From now on we denote our Euclidean buildings by \( \euclbuild \) and only specify the rest of the datum when needed.

Let \( \chambset(\euclbuild) \) be the set of chambers of a Euclidean building \( \euclbuild \) of type \( \rootSys \).
According to \cite[Proposition 4.81]{abramenkoBrown08} there exists a function
\begin{equation}\label{equ:weylDistance}
	\delta_{\affWeyl} \colon \chambset(\euclbuild) \times \chambset(\euclbuild) \to \affWeyl
\end{equation}
with the following properties:
\begin{enumerate}[(i)]
	\item
		Given a minimal gallery \( \Gamma = (C_0,\ldots,C_d) \) of type \( w(\Gamma) = (s_1,\ldots,s_d) \), \( \delta_{\affWeyl}(C_0, C_d) \) is the element \( w = s_1 \cdots s_d \) represented by \( w(\Gamma) \).
	\item
		Let \( C \) and \( D \) be chambers, and let \( w = \delta_{\affWeyl}(C,D) \).
		The function \( \Gamma \mapsto w(\Gamma) \) gives a one-to-one correspondence between minimal galleries from \( C \) to \( D \) and reduced decompositions of \( w \).
\end{enumerate}
The function \( \delta_{\affWeyl} \) is called the \emph{Weyl distance function} associated to \( \euclbuild \).
For more properties of \( \delta_{\affWeyl} \) see \cite[Section 4.8]{abramenkoBrown08}.

\begin{defi}[Combinatorial Balls]\label{def:comb-balls}
	Let \( \chambset(\euclbuild) \) be the chamber set of a Euclidean building \( \euclbuild \) of type \( \rootSys \).
	For \( C \in \chambset(\euclbuild) \) and \( s = s_i \in S\) define the set
	\[
	\mathcal C_s(C) \coloneqq \Set{D \in \chambset(\euclbuild) \given C \sim_i D \text{ and } C \neq D}.
	\]
	We make the following definitions:
	\begin{enumerate}[(a)]
		\item
			We call the building \emph{locally finite} if \( \# \mathcal C_s(C) < \infty \) for all \( s \in S \) and \( C \in \chambset \).
		\item
			We call the building \emph{regular} if \( \# \mathcal C_s(C) \) is independent of \( C \in \chambset \).

			In this case we set \( q_s \coloneqq \# \mathcal C_s(C) \) for each \( s \in S \).
		\item
			We call the building \emph{strongly regular} if it is regular and \( q_s = q_{\sigma(s)} \) for all \( \sigma \in \trAutRoot(I) \).
	\end{enumerate}
\end{defi}

\begin{rem}
	For a Euclidean building of type \( \rootSys \) being strongly regular is not neccesarily a restriction.
	In \cite[Section 3.8]{parkinsonDiss} one can find that \quoted{choosing the right root system} overcomes this restriction.
	For example any semi-homogeneous tree is a type \( BC_1 \) building and only homogeneous trees are buildings of type \( A_1 \).
	A similar convention is needed for buildings of type \( C_n \) and \( BC_n \) (\( n \geq 2 \)).
	
	If one sticks to these conventions one always chooses the right root systems and every regular building is automatically strongly regular.
\end{rem}

If \( s = s_i \in S \) we set \( q_i \coloneqq q_s \) and call \( \Set{q_i}_{i \in I} \) a \emph{parameter system} of the Euclidean building \( \euclbuild \).

We use the Weyl distance function \( \delta_{\affWeyl} \) to extend the definition of the combinatorial balls to all of \( \affWeyl \) via
\[
\mathcal C_w(C) \coloneqq \Set{D \in \chambset \given \delta_{\affWeyl}(C,D) = w}.
\]

\begin{prop}[{\cite[Proposition 1.7.1]{parkinsonDiss}}]
	Let \( \euclbuild \) be a locally finite regular Euclidean building.
	\begin{enumerate}[(i)]
		\item
			\( \# \mathcal C_w(C) = q_{i_1} q_{i_2} \cdots q_{i_n} \) whenever \( w = s_{i_1} \cdots s_{i_n} \) is a reduced expression, and
		\item
			\( q_i = q_j \) whenever \( m_{i,j} < \infty \) is odd.
	\end{enumerate}
\end{prop}

\begin{rem}
	This justifies the notation \( q_w \coloneqq \# \mathcal C_w(C) \) for \( w \in \affWeyl \).
	Moreover, one can see that \( q_w = q_{w^{-1}} \) for all \( w \in \affWeyl \).
	This follows from the fact that \( q_{s_i} = q_{s_i^{-1}} \) because \( s_i = s_i^{-1} \).
\end{rem}

If \( \affWeyl \) is the affine Weyl group of a locally finite, strongly regular Euclidean building, then we can extend the definition of \( q_w \) to the extended affine Weyl group as follows.

\begin{defi}[Extended Parameters]\label{defi:extendedParameters}
	Let \( \affWeyl \) be the affine Weyl group of a locally finite, strongly regular Euclidean building.
	As any \( w \in \extWeyl \) acts on the chambers of \( \euclCox \), \( wC_0 \) is again a chamber of \( \euclCox \) and as \( \affWeyl \) acts simply transitively on the chambers, there is a unique \( w^\prime \in \affWeyl \) such that \( w C_0 = w^\prime C_0 \) and we define \( q_w \coloneqq q_{w^\prime} \).
\end{defi}

These parameters coincide with the ones in \cite[p.~44 and Definition 5.1.2]{parkinsonDiss} and are independent of the chamber \( C \in \chambset(\euclCox) \).
This is a consequence of the strong regularity which implies \( q_{w^\prime} = q_{\sigma(w^\prime)} \) for all \( \sigma \in \trAutRoot(I) \) and \( w^\prime \in \affWeyl \).

For further use we note that \( q_w = q_{w^{-1}} \) also holds for all \( w \in \extWeyl \).

\begin{lem}\label{lem:extWeylParaInverse}
	Given a locally finite, strongly regular Euclidean building \( \euclbuild \) and \( w \in \extWeyl \) we have \( q_w = q_{w^{-1}} \).
\end{lem}

\begin{proof}
	This follows from the Definition in \cite{parkinsonDiss} of the parameters \( q_w \) for \( w \in \extWeyl \).
	We use the decomposition \( \extWeyl = \affWeyl \ltimes \Stab_{\extWeyl}(C_0) \) and that \( q_{w} = q_{w^\prime} \) when \( w = w^\prime g \) with \( w^\prime \in \affWeyl \) and \( g \in \Stab_{\extWeyl}(C_0) \).
	It holds
	\[
	w^{-1} = (w^\prime g)^{-1} = g^{-1} {w^\prime}^{-1} = g^{-1} {w^\prime}^{-1} g g^{-1} = \sigma({w^\prime}^{-1}) g^{-1}
	\]
	by using the relation \cite[(3.6.3)]{parkinsonDiss} and the type-rotating permutation \( \sigma \) associated to \( g \).
	Hence we have
	\[
	q_{w^{-1}} = q_{\sigma({w^\prime}^{-1})} = q_{{w^\prime}^{-1}} = q_{w^\prime} = q_w \quad \text{for all } w \in \extWeyl.
	\]
	Here we used that our building is strongly regular.
\end{proof}

\begin{cor}
	From \thref{lem:extWeylParaInverse} we obtain \( q_{\transl_{-\paramA}} = q_{\transl_\paramA} \) for all \( \paramA \in \domCoweights \).
\end{cor}

The buildings in this paper are assumed to be locally finite and strongly regular Euclidean buildings.
\begin{rem}
	In \cite[Theorem 1.7.4]{parkinsonDiss} it is shown that a thick building with \( m_{i,j} < \infty \) for all \( i,j \in I \) is already regular.
	The condition \( m_{i,j} < \infty \) means that the building has no rank two residues of type \( \widetilde{A}_1 \).
\end{rem}

The points \( x \in \coweightLat \) are special points in the tesselation of the Euclidean space given by the hyperplanes.
They have the property that for every hyperplane \( H \in \mathcal H \) there exists a hyperplane \( H^\prime \in \mathcal H \) parallel to \( H \) going through \( x \) (see \cite[Section 3.4]{parkinsonDiss}).
As discussed in \cite[Section 3.4]{parkinsonDiss} this property is often used to define \emph{special} vertices of a Euclidean Coxeter complex.
If the root system of the Euclidean building is reduced, the special vertices are the elements of the coweight lattice.
But in the cases where the root system of the Euclidean building is non-reduced \( \coweightLat \) is a proper subset of these special vertices.

This leads to the notion of \emph{good} vertices.

\begin{defi}[Good Vertices]\label{defi:good vertices}
	The elements of \( \coweightLat \) are called \emph{good vertices} in \( \euclCox \). 
	To extend the notion to all vertices in a Euclidean building \( \euclbuild \) we use our labeling \( \labeling \) and define
	\[
	I_{\coweightLat} \coloneqq \Set{\labeling(\lambda) \given \lambda \in \coweightLat \subseteq \euclCox}.
	\]
	This is the set of all labels of vertices in the coweight lattice and it holds \( x \in \coweightLat \) if and only if \( \labeling(x) \in I_{\coweightLat} \).
	Hence \( I_{\coweightLat} \) are precisely the labels of the good vertices.
	This allows us to extend the notion of \emph{good} vertices to the Euclidean building \( \euclbuild \).
	A vertex \( x \in \euclbuild \) is a \emph{good} vertex if \( \labeling(x) \in I_{\coweightLat} \).
	To make it easier when speaking about good vertices we denote the set of all good vertices of a Euclidean building by \( V_{\coweightLat} \).
\end{defi}

This definition of good vertices agrees with the one given in \cite[Definition 3.8.1]{parkinsonDiss}.
A useful characterization of a good vertex of a Euclidean building is the following.

\begin{prop}[{{\cite[Proposition 3.8.2]{parkinsonDiss}}}]\label{prop:characGoodVertexBuild}
	A vertex \( x \in \euclbuild \) is good if and only if there exists an apartment \( \apartment \) containing \( x \) and a type-preserving isomorphism \( \psi \colon \apartment \to \euclCox \) such that \( \psi(x) \in \coweightLat \).
\end{prop}

\subsection{Morphisms and their Properties}\label{subsec:MaP}

We want to extend the definition of type-rotating morphisms to morphisms between local buildings.
\begin{defi}[Type-rotating Morphisms]\label{def:tr-morphisms}
	Let \( \rootSys \subseteq \eucl \) be an irreducible root system with a fixed basis \( \rootBasis \).
	Let \( \chambSys \) and \( \mathcal D \) be two local buildings that come from the same Euclidean building of type \( \rootSys \).
	Then we call a morphism of chamber sytems \( \psi \colon \chambSys \to\mathcal D \) \emph{type-rotating} morphism if we can find a corresponding permutation \( \sigma \in \trAutRoot(I) \subseteq S(I) \) (where \( \trAutRoot(I) \) is defined in \thref{def:Aut_tr}).
	We denote the set of type-rotating morphisms from \( \chambSys \) to \( \mathcal D \) by \( \trHom(\chambSys, \mathcal D) \).
\end{defi}

Note that we can directly extend this definition to subcomplexes of local buildings, e.g. to the fundamental sector \( \sector_0 \subset \Sigma \).

\begin{rem}{{{\cite[Definition 4.1.1]{parkinsonDiss}}}}\label{rem:typeRotApartments}
	Let \( \apartment_1, \apartment_2 \subseteq \euclbuild \) be two apartments in a Euclidean building.
	An isomorphism \( \psi \colon \apartment_1 \to \apartment_2 \) is type-rotating if and only if we find type-preserving morphisms \( \psi_i \colon \apartment_i \to \euclCox \) (\( i = 1,2 \)) and an element \( w \in \extWeyl \) such that
	\[
	\psi = \psi_2^{-1} \circ w \circ \psi_1.
	\]
	This follows from our definition that \( \extWeyl \) are the type-rotating automorphisms of \( \euclCox \).
\end{rem}

As a consequence we get the following lemma in \cite[Lemma 4.1.3]{parkinsonDiss}.

\begin{lem}\label{lem:typRotFixWeylGr}
	Suppose \( x \in \euclbuild \) is a good vertex contained in two apartments \( \apartment_1 \) and \( \apartment_2 \) and suppose that \( \psi_i \colon \apartment_i \to \euclCox \) (\( i = 1,2 \)) are type-rotating isomorphisms such that \( \psi_1(x) = 0 = \psi_2(x) \).
	Let \( \phi \colon \apartment_1 \to \apartment_2 \) be a type-preserving isomorphism mapping \( x \) to \( x \).
	Then the composition \( \psi_2 \circ \phi \circ \psi_1^{-1} \in W_0 \).
\end{lem}

\begin{rem}
	The idea behind the type-rotating morphisms is that we want to view them as \quoted{orientation-preserving} maps.
\end{rem}

\begin{rem}\label{rem:loc-inj-morph}
	As we have a topology on all buildings and chamber systems the term \emph{locally injective} makes sense.
	Recall \thref{rem:localBijSphericalRes} which stated that local topological properties imply properties for the spherical residues.
	In particular, for a Euclidean Coxeter complex \( \euclCox \) and a chamber system \( \chambSys \), a morphism of chamber systems \( \psi \colon \euclCox \to \chambSys \) is locally injective if and only if \( \psi \) is injective on all spherical residues.
	Moreover, for a subcomplex \( \mathcal D \subseteq \euclCox \) a morphism \( \psi \colon \mathcal D \to \chambSys \) is locally injective if and only if it is injective on all sets \( \mathcal D \cap \Res_J(C) \) for any chamber \( C \in \mathcal D \) and \( J \subseteq I \) spherical.
\end{rem}

\begin{rem}
	As a consequence of \thref{rem:chamberMorphisms}\ref{item:uniquePremuOfMorph} any locally injective morphism has a unique permutation associated to it.
\end{rem}

\begin{lem}\label{lem:localInjMapsGalToGal}
	Let \( \euclCox \) be a Euclidean Coxeter complex of type \( \rootSys \), \( \mathcal D \subseteq \euclCox \) a subcomplex and \( \chambSys \) a local building of the same type.
	A locally injective type-rotating morphism \( \psi \colon \mathcal D \to \chambSys \) maps galleries in \( \mathcal D \) to galleries in \( \chambSys \).
	In other words, \( i \)-adjacent chambers are mapped to \( \sigma(i) \)-adjacent chambers for some \( \sigma \in \trAutRoot(I) \).
\end{lem}

\begin{proof}
	Every morphism of chamber systems maps galleries to pre-galleries.
	Recall \thref{defi:terminologyChamberSys} which tells us that in a gallery two consecutive chambers must be different.
	We need to check that if \( C \sim_i D \) are two \( i \)-adjacent chambers their images are not equal, i.e., \( \psi(C) \neq \psi(D) \).
	As \( C \sim_i D \) we have \( C, D \in \Res_{\Set{i}}(C) \) and this is a spherical residue as \( m_{i,i} = 1 \) in the Coxeter matrix.
\end{proof}

Recall the definition of \( \sigma \)-isometries from \cite[Definition 5.59]{abramenkoBrown08}.

\begin{defi}[\( \sigma \)-Isometries]
	Let \( \euclbuild, \euclbuild^\prime \) be two Euclidean buildings of type \(\rootSys \) and \( \delta_{\affWeyl} , \delta_{\affWeyl}^\prime \) the Weyl distance functions on \( \euclbuild \) and \( \euclbuild^\prime \) as defined in \Cref{equ:weylDistance}.
	Let \( \sigma \in \Aut(\affWeyl,S) \) be an automorphism of the Euclidean Coxeter group \( (\affWeyl,S) \), i.e., \( \sigma \) maps \( S \) bijectively onto \( S \) and preserves the values of the Coxeter matrix \( M \).
	A \emph{\( \sigma \)-isometry} is a morphism \( \phi \colon \euclbuild \to \euclbuild^\prime \) satisfying
	\[
	\delta_{\affWeyl}^\prime(\phi(C), \phi(D)) = \sigma(\delta_{\affWeyl}(C,D)).
	\]
	Obviously this notion can be extended to subcomplexes \( \mathcal D \subseteq \euclbuild \) and maps \( \phi \colon \mathcal D \to \euclbuild^\prime \).
	
	If \( \sigma = \id \) or equivalently \( \phi \) is type-preserving, we call \( \phi \) an isometry.
\end{defi}

\begin{rem}\label{rem:sigmaIsomMetricIsom}
	\mbox{}
	\begin{enumerate}[(i)]
		\item
			As explained in \cite[Examples 5.60(c)]{abramenkoBrown08} a \( \sigma \)-isometry is just another characterization of automorphisms of a (Euclidean) building.
		\item\label{item:sigmaIsomIsMetricIsom}
			Any \( \sigma \)-isometry from a subset of a Euclidean Coxeter complex of type \( \rootSys \) into a Euclidean building of type \( \rootSys \) is an isometry in the (Euclidean) metric sense.
			This follows from the discussions in \cite[p.~532 and p.~570]{abramenkoBrown08}.
	\end{enumerate}
\end{rem}

\begin{lem}\label{lem:localInjIsSigIsom}
	Let \( \euclCox \) be a Euclidean Coxeter complex of type \( \rootSys \), \( \mathcal D \subseteq \euclCox \) be a combinatorially convex subcomplex and \( \euclbuild \) be a Euclidean building of type \( \rootSys \) with apartments isometric to \( \euclCox \).
	Then a locally injective morphism of chamber systems \( \psi \colon \mathcal D \to \euclbuild \) is a \( \sigma \)-isometry.
\end{lem}

\begin{proof}
	This is \cite[Proposition 1.1]{mozes95}.
\end{proof}

\begin{lem}[{{\cite[Lemma 5.62]{abramenkoBrown08}}}]\label{lem:sigmaImageOfConvex}
	Let \( \euclbuild, \euclbuild^\prime \) be two Euclidean buildings of type \( \rootSys \).
	Suppose we have chosen two different bases \( \rootBasis, \rootBasis^\prime \) of \( \rootSys \).
	This gives us different sets of generators \( S, S^\prime \) of \( \affWeyl \).
	Let \( \sigma \colon (\affWeyl,S) \to (\affWeyl, S^\prime) \) be an isomorphism of (Euclidean) Coxeter systems.
	
	If \( \mathcal M \subseteq \euclbuild \) is a combinatorially convex subset and \( \phi \colon \mathcal M \to \euclbuild^\prime \) is a \( \sigma \)-isometry, then the image \( \phi(\mathcal M) \subseteq \euclbuild^\prime \) is combinatorially convex.
\end{lem}

Combining \thref{lem:localInjIsSigIsom,lem:sigmaImageOfConvex} we obtain the following corollary.

\begin{cor}\label{cor:localInjConvexSets}
	Let \( \euclCox \) be a Euclidean Coxeter complex of type \( \rootSys \) and \( \euclbuild \) be a Euclidean building of type \( \rootSys \).
	That means the apartments of \( \euclbuild \) are isometric to \( \euclCox \).
	Then the image of a combinatorially convex subset of \( \coxCompl \) under a locally injective map is combinatorially convex in \( \build \).
\end{cor}

\begin{thm}[{{\cite[Theorem 11.53]{abramenkoBrown08}}}]\label{thm:subsetOfApartment}
	Let \( \euclbuild \) be a Euclidean building and \( Y \subseteq \euclbuild \) a subset that is either convex (in the Euclidean sense) or has nonempty interior.
	If \( Y \) is isometric to a subset of an apartment (i.e., subset of a Euclidean space), then \( Y \) is contained in an apartment.
\end{thm}

\section{Dynamics on the Space of Sectors}\label{sec:dynaSetting}

From now on we only consider buildings \( \euclbuild \) which are locally finite, strongly regular Euclidean buildings of type \( \rootSys \), where \( \rootSys \) is an irreducible root system.
Whenever we talk about local buildings, Euclidean buildings or Euclidean Coxeter complexes the whole structure we use is given by the corresponding root system even though we might not specify the root system explicitly.

\subsection{Sectors in Local Buildings}

The space for the Weyl chamber dynamics on Euclidean buildings is defined as follows.

\begin{defi}[Sectors]
	Let \( \quot \) be a local building of type \( \rootSys \) and \( \sector_0 \) the fundamental sector associated to the corresponding Euclidean Coxeter complex of type \( \rootSys \), then we define
	\[
	\secQuot \coloneqq \Set{\secOnQuot \in \trHom(\sector_0, \quot) \given \secOnQuot \text{ locally injective}}
	\]
	as the set of \emph{sectors} in \( \quot \).
\end{defi}

\begin{exmp}\label{exmp:nonbacktrackingGraphPath}
	A quotient of a Euclidean building of type \( A_1 \) is a regular graph.
	In this case \( \secQuot \) is the set of all non-backtracking paths.
	The sector \( \sector_0 \) consists of the non-negative reals \( \Rnum_{\geq 0} \), the coweight lattice is isomorphic to \( \Znum \) and the dominant coweights can be identified with the natural numbers \( \Nnum_0 \).
	To clarify the nature of the image in the quotient we enumerate the edges/chambers by natural numbers \( \Nnum \).
	\begin{figure}[H]
		\centering
		{
\pgfmathsetmacro\side{3} % Length of triangle side
\pgfmathsetmacro\arrowSep{0.05cm} % Space of Edge path to nodes
\pgfmathsetmacro\bendAngle{12.5} % Bending angle for edges
\begin{tikzpicture}[
	font=\tiny, % Fontsize
	Point/.style={circle,inner sep=1.5pt}, % The Nodes
	>={Stealth[length=1.5mm]}, % Arrow style
	description/.style={fill=white,inner sep=1pt} % Can be used if background is white
	]
	% Draw the triangle
	\node[Point, fill=black] (l) at (0,0) {};
	\node[Point, fill=black] (r) at (\side,0) {};
	\node[Point, fill=black] (m) at ($(\side/2,0) + {sqrt(3)/6}*(0,\side)$) {};
	\node[Point, fill=black] (o) at ($(\side/2,0) + {sqrt(3)/2}*(0,\side)$) {};

	% Add the lines of the triangle
	% Add them now so they are over the description boxes
	\draw[-] (l) -- (r);
	\draw[-] (l) -- (m);
	\draw[-] (l) -- (o);
	\draw[-] (r) -- (m);
	\draw[-] (r) -- (o);
	\draw[-] (o) -- (m);

	% Edge Path
	\draw[->,shorten >=\arrowSep,shorten <=\arrowSep,color=blue1] (l) edge[bend left=-\bendAngle] node [description] {\( 1 \)} (r);

	\draw[->,shorten >=\arrowSep,shorten <=\arrowSep,color=blue2] (r) edge[bend left=-\bendAngle] node [description] {\( 2 \)} (m);
	% Add the line again so they are over the description boxes
	\draw[-] (r) -- (m);

	\draw[->,shorten >=\arrowSep,shorten <=\arrowSep,color=blue3] (m) to[bend right=\bendAngle] node [description] {\( 3 \)} (o);

	\draw[->,shorten >=\arrowSep,shorten <=\arrowSep,color=blue4] (o) to[bend left=\bendAngle] node [description] {\( 4 \)} (r);

	\draw[->,shorten >=\arrowSep,shorten <=\arrowSep,color=blue5] (r) to[bend right=\bendAngle] node [description] {\( 5 \)} (l);

	\draw[->,shorten >=\arrowSep,shorten <=\arrowSep,color=blue6] (l) to[bend left=\bendAngle] node [description] {\( 6 \)} (m);
	% Add the lines of the triangle
	% Add them now so they are over the description boxes
	\draw[-] (l) -- (m);

	\draw[->,shorten >=\arrowSep,shorten <=\arrowSep,color=blue7] (m) to[bend left=\bendAngle] node [description] {\( 7 \)} (o);

	\draw[->,shorten >=\arrowSep,shorten <=\arrowSep,transparent!0,shade path={left color=white, right color=blue8}] (o) to[bend right=\bendAngle] (l);
\end{tikzpicture}

}
		\caption{Non-backtracking path}
	\end{figure}
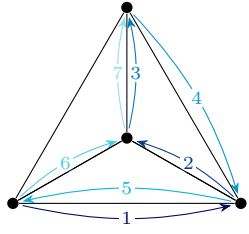
\end{exmp}

Let us note that whenever \( \quot \) is simply-connected, e.g.\ when \( \quot \) is itself a Euclidean building \( \euclbuild \) or the Euclidean Coxeter complex \( \euclCox \), then a sector \( \sectorMap \colon \sector_0 \to \quot \) defines a type-rotating automorphism onto its image \( \Image(\sectorMap) \subset \quot \).
We thus obtain the following lemma.

\begin{lem}\label{lem:maps_sets_are_the_same}
	\mbox{}
	\begin{enumerate}[(i)]
		\item\label{item:sectorsInEuclCox}
			If \( \euclCox \) is a Euclidean Coxeter complex, we have
			\[
			\sector(\Sigma) = \Set{f \colon \sector_0 \to \euclCox, \; x \mapsto w.x \given w \in \extWeyl}.
			\]
			In other words, our definition of sectors of \( \euclCox \) describes all geometric sectors in \( \euclCox \) with a base point in \( \coweightLat \).
		\item\label{item:sectorsInEuclBuild}
			Let \( \euclbuild \) be a Euclidean building.
			\begin{enumerate}[(a)]
				\item\label{item:uniqueSectorImage}
					A sector \( \sectorMap \in \sector(\euclbuild) \) is uniquely determined by its image.
				\item\label{item:apartmentTypeRotSector}
					For every image \( \Image \sectorMap \) of a sector there is an apartment \( \apartment \subseteq \euclbuild \) and a unique type-rotating isomorphism \( \psi \colon \apartment \to \euclCox \) with \( \psi(\Image \sectorMap) = \sector_0 \).
			\end{enumerate}
	\end{enumerate}
\end{lem}

\begin{proof}
	For \ref{item:sectorsInEuclCox} it suffices to prove \ref{item:sectorsInEuclBuild}\ref{item:apartmentTypeRotSector} as any Coxeter complex has only one apartment and the type-rotating isomorphisms of \( \euclCox \) are \( \trAut(\euclCox) = \extWeyl \) by definition.
	
	To show \ref{item:sectorsInEuclBuild}\ref{item:uniqueSectorImage} we take \( \sectorMap_1, \sectorMap_2 \in \sector(\euclbuild) \) with \( \Image \sectorMap_1 = \Image \sectorMap_2 \).
	In particular, we have \( \sectorMap_1(0) = \sectorMap_2(0) \) and \( \sectorMap_1(C_0) = \sectorMap_2(C_0) \).
	
	By \cite[Theorem 5.73]{abramenkoBrown08} we can extend both maps to maps \( \obar{\sectorMap_1}, \obar{\sectorMap_2} \colon \euclCox \to \euclbuild \).
	From \thref{thm:subsetOfApartment} we know that their images lie in an apartment \( \apartment \).
	Now we can apply \thref{lem:typRotFixWeylGr} and find that \( \obar{\sectorMap_1}^{-1} \circ \obar{\sectorMap_2} \in \weylGr \).
	Moreover, we know \( \obar{\sectorMap_1}^{-1} \circ \obar{\sectorMap_2}(C_0) = C_0 \) and as \( \weylGr \) acts simply transitively on the chambers having \( 0 \) in their closure (see \thref{prop:weylGroupActions}\ref{item:weylGrActsOnZeroClosure}) we get \( \obar{\sectorMap_1}^{-1} \circ \obar{\sectorMap_2} = \id \).
	
To show 	 \ref{item:sectorsInEuclBuild}\ref{item:apartmentTypeRotSector}, we take \( \sectorMap \colon \sector_0 \to \euclbuild \) and extend it to an injective morphism \( \obar{\sectorMap} \colon \euclCox \to \euclbuild \).
	By the same argument as before the image is contained in an apartment \( \apartment \).
	Hence we have an injective morphism \( \obar{\sectorMap} \colon \euclCox \to \apartment \).
	Inverting the map gives the desired morphism.
	The uniqueness now follows from \ref{item:uniqueSectorImage}.
\end{proof}

\begin{cor}
	By using the description of \thref{lem:maps_sets_are_the_same}\ref{item:sectorsInEuclBuild}.\ref{item:apartmentTypeRotSector} we find that every image of a sector defines a sector in a Euclidean building as defined in the literature (\cite[Definition 11.45 and Definition 11.61]{abramenkoBrown08}).
\end{cor}
	
Based on \thref{lem:maps_sets_are_the_same} we can change between the point of view of a sector in \( \sector(\euclbuild) \) as a locally injective map and their images as subcomplexes.
The latter point of view is usually adopted when speaking of sectors in Euclidean buildings, e.g., in \cite{abramenkoBrown08}.
Note, however, that by our assumption of having type-rotating maps as locally injective morphisms, the base point of our sectors seen as subcomplexes are necessarily good vertices.
In the particular cases where good and special vertices do not coincide, one could also consider the larger class of sectors starting at any special vertex.
But from the point of view of the Weyl chamber dynamics that we will consider this is not helpful, because the space of sectors would then simply split into dynamically disjoint subspaces, which are all equivalent to the space of sector with good vertices as base points.

In a Euclidean building, sectors are used to define the spherical building at infinity.
This can be done as follows.
On the set of sectors \( \sector(\euclbuild) \) of a Euclidean building \( \euclbuild \) one defines an equivalence relation where two sectors \( \sectorMap_1 \sim \sectorMap_2 \) are equivalent if and only if there exists a sector \( \sectorMap_3 \in \sector(\euclbuild) \) such that
\[
\Image(\sectorMap_3) \subset \Image(\sectorMap_1) \cap \Image(\sectorMap_2).
\]
The quotient \( \boundary \coloneqq \sector(\euclbuild)/\sim \) is called the \emph{spherical building at infinity}.
We collect some properties from \cite{abramenkoBrown08}:

\begin{rem}[Chambers at Infinity]\label{rem:propertiesOfBoundary}
	\mbox{}
	\begin{enumerate}[(i)]
		\item\label{item:sphericalBuildingsAtInfty}
			The elements of \( \boundary \) can be given the structure of a combinatorial chamber system and as such it is a spherical building (see \cite[Theorem 11.79]{abramenkoBrown08}).
			Consequently, we call the \( \omega \in \boundary \) \emph{chambers at infinity}.
		\item\label{item:oppoChambersDefineApartment}
			In any spherical building there is a notion of \emph{opposite chambers} (see \cite[Definition 4.68]{abramenkoBrown08}).	Any pair of opposite chambers in a spherical building defines a unique apartment in the building (c.f.~\cite[Theorem 4.70]{abramenkoBrown08}).
			For a Euclidean building \( \euclbuild \) the apartments in the building at infinity \( \boundary \) and the apartments of \( \euclbuild \) are in one-to-one correspondence (see \cite[Theorem 11.79]{abramenkoBrown08}).
			Hence any pair \( (\omega,\omega^\prime) \in \boundary^\prime \) of opposite chambers defines a unique apartment \( \apartment(\omega, \omega^\prime) \subseteq \euclbuild \).
		\item\label{item:sphereFromPoint}
			For any good vertex \( x \in \euclbuild \) and any \( \omega \in \boundary \) there is a unique sector \( \sectorMap \in \sector(\euclbuild) \) such that \( \sectorMap(0)= x \) and \( [s]_\sim = \omega \) (\cite[Lemma 11.75]{abramenkoBrown08}).
		For convenience we write \( \sector^x(\omega) \) to specify a sector with these properties.
	\end{enumerate}
\end{rem}

\begin{rem}\label{rem:sectorLifts}
	Let \( \quot \) be a local building of type \( \rootSys \).
	Hence we have a covering map which is a simplicial morphism \( p \colon \euclbuild \to \quot \) from a Euclidean building \( \euclbuild \) of type \( \rootSys \).
	As both \( \sector_0 \) and \( \euclbuild \) are simply-connected we can lift any morphism \( \secOnQuot \in \secQuot \) to a map \( \lift{\secOnQuot} \colon \sector_0 \to \euclbuild \) such that \( p \circ \lift{\secOnQuot} = \secOnQuot \).
	After choosing base points, e.g. \( 0 \in \sector_0 \), \( \secOnQuot(0) \in \quot \) and \( x_0 \in p^{-1}(\secOnQuot(0)) \) this lift is unique.
	As \( p \) is locally bijective, the local injectivity is preserved and therefore \( \lift{\secOnQuot} \in \sector(\euclbuild) \) defines a sector on the Euclidean building.
	Thus, any lift of a sector \( \secOnQuot \in \secQuot \) defines a unique sector in \( \sector(\euclbuild) \).
	
	On the other hand we can take a sector \( \sectorMap \in \sector(\euclbuild) \) and compose it with \( p \).
	In this way we get a map \( p \circ \sectorMap \colon \sector_0 \to \quot \) which is type-rotating and locally injective, i.e., \( p \circ \sectorMap \in \secQuot \).
\end{rem}

\subsection{Ultrametric on Sectors}

\begin{defi}
	Let \( \secOnQuot_1, \secOnQuot_2 \in \secQuot \) be two sectors.
	We define
	\begin{align*}
		\mathcal E(\secOnQuot_1, \secOnQuot_2) &\coloneqq \Set{x \in \sector_0 \given \secOnQuot_1(x) = \secOnQuot_2(x)} 
	\end{align*}
	and \( \mathcal E_0(\secOnQuot_1, \secOnQuot_2) \) to be the connected component of \( 0 \) in \(\mathcal E(\secOnQuot_1, \secOnQuot_2)\).
	If \( 0 \notin \mathcal E(\secOnQuot_1, \secOnQuot_2) \), then we set \( \mathcal E_0(\secOnQuot_1, \secOnQuot_2) = \emptyset \).
\end{defi}

\begin{exmp}
	As \( \mathcal E_0(\secOnQuot_1, \secOnQuot_2) \) plays an important role in the definition of the metric we take a moment to discuss the difference between \( \mathcal E_0(\secOnQuot_1, \secOnQuot_2) \) and \( \mathcal E(\secOnQuot_1, \secOnQuot_2) \) in the case of a graph.
	In \thref{exmp:nonbacktrackingGraphPath} we already saw an example of a sector.
	We take the 3-regular graph \( \mathcal G \):
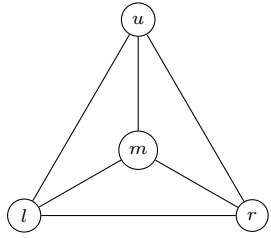
	\begin{figure}[H]
		\centering
		{
\pgfmathsetmacro\side{3} % Length of triangle side
\pgfmathsetmacro\arrowSep{0.05cm} % Space of Edge path to nodes
\pgfmathsetmacro\bendAngle{12.5} % Bending angle for edges
\begin{tikzpicture}[
	font=\tiny, % Fontsize
	Point/.style={circle,inner sep=2.5pt, draw}, % The Nodes
	>={Stealth[length=1.5mm]}, % Arrow style
	description/.style={fill=white,inner sep=1pt} % Can be used if background is white
	]
	% Draw the triangle
	\node[Point, fill=white] (l) at (0,0) {\( l \)};
	\node[Point, fill=white] (r) at (\side,0) {\( r \)};
	\node[Point, fill=white] (m) at ($(\side/2,0) + {sqrt(3)/6}*(0,\side)$) {\( m \)};
	\node[Point, fill=white] (o) at ($(\side/2,0) + {sqrt(3)/2}*(0,\side)$) {\( u \)};

	% Add the lines of the triangle
	% Add them now so they are over the description boxes
	\draw[-] (l) -- (r);
	\draw[-] (l) -- (m);
	\draw[-] (l) -- (o);
	\draw[-] (r) -- (m);
	\draw[-] (r) -- (o);
	\draw[-] (o) -- (m);
\end{tikzpicture}
}
	\end{figure}
	Observe that any chamber is uniquely determined by its vertices (this holds for every building).
	Hence we can use the vertices to define a map \( \secOnQuot \colon \sector_0 \to \mathcal G \).
	Recall that the vertices (and in this case the dominant coweights) of \( \sector_0 \) are the natural numbers \( \Nnum_0 \).
	To understand the difference between \( \mathcal E_0(\secOnQuot_1, \secOnQuot_2) \) and \( \mathcal E(\secOnQuot_1, \secOnQuot_2) \) we take two maps \( \sector_1, \sector_2 \colon \Nnum_0 \to \Set{a,b,c,d} \).
	Let
	\[
	\secOnQuot_1 = (l, r, m, u, r, m, u, l, \ldots) \quad \text{and} \quad \secOnQuot_2 = (l, r, m, u, l, m, u, l, \ldots).
	\]
	This looks as follows.
	
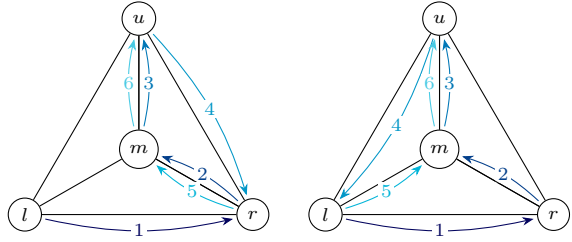
\begin{figure}[H]
		\centering
		{
\pgfmathsetmacro\side{3} % Length of triangle side
\pgfmathsetmacro\arrowSep{0.05cm} % Space of Edge path to nodes
\pgfmathsetmacro\bendAngle{12.5} % Bending angle for edges
\begin{tikzpicture}[
	font=\tiny, % Fontsize
	Point/.style={circle,inner sep=2.5pt, draw}, % The Nodes
	>={Stealth[length=1.5mm]}, % Arrow style
	description/.style={fill=white,inner sep=1pt} % Can be used if background is white
	]
	% Draw the triangle
	\node[Point, fill=white] (l) at (0,0) {\( l \)};
	\node[Point, fill=white] (r) at (\side,0) {\( r \)};
	\node[Point, fill=white] (m) at ($(\side/2,0) + {sqrt(3)/6}*(0,\side)$) {\( m \)};
	\node[Point, fill=white] (o) at ($(\side/2,0) + {sqrt(3)/2}*(0,\side)$) {\( u \)};

	% Add the lines of the triangle
	% Add them now so they are over the description boxes
	\draw[-] (l) -- (r);
	\draw[-] (l) -- (m);
	\draw[-] (l) -- (o);
	\draw[-] (r) -- (m);
	\draw[-] (r) -- (o);
	\draw[-] (o) -- (m);

	% Edge Path
	\draw[->,shorten >=\arrowSep,shorten <=\arrowSep,color=blue1] (l) edge[bend left=-\bendAngle] node [description] {\( 1 \)} (r);

	\draw[->,shorten >=\arrowSep,shorten <=\arrowSep,color=blue2] (r) edge[bend left=-\bendAngle] node [description] {\( 2 \)} (m);
	% Add the line again so they are over the description boxes
	\draw[-] (r) -- (m);

	\draw[->,shorten >=\arrowSep,shorten <=\arrowSep,color=blue3] (m) to[bend right=\bendAngle] node [description] {\( 3 \)} (o);

	\draw[->,shorten >=\arrowSep,shorten <=\arrowSep,color=blue4] (o) to[bend left=\bendAngle] node [description] {\( 4 \)} (r);

	\draw[->,shorten >=\arrowSep,shorten <=\arrowSep,color=blue5] (r) to[bend left=\bendAngle] node [description] {\( 5 \)} (m);

	\draw[->,shorten >=\arrowSep,shorten <=\arrowSep,color=blue6] (m) to[bend left=\bendAngle] node [description] {\( 6 \)} (o);
	% Add the lines of the triangle
	% Add them now so they are over the description boxes
	\draw[-] (m) -- (o);

	\draw[->,shorten >=\arrowSep,shorten <=\arrowSep,transparent!0,shade path={left color=white, right color=blue8}] (o) to[bend right=\bendAngle] (l);
\end{tikzpicture}
\quad
\begin{tikzpicture}[
	font=\tiny, % Fontsize
	Point/.style={circle,inner sep=2.5pt, draw}, % The Nodes
	>={Stealth[length=1.5mm]}, % Arrow style
	description/.style={fill=white,inner sep=1pt} % Can be used if background is white
	]
	% Draw the triangle
	\node[Point, fill=white] (l) at (0,0) {\( l \)};
	\node[Point, fill=white] (r) at (\side,0) {\( r \)};
	\node[Point, fill=white] (m) at ($(\side/2,0) + {sqrt(3)/6}*(0,\side)$) {\( m \)};
	\node[Point, fill=white] (o) at ($(\side/2,0) + {sqrt(3)/2}*(0,\side)$) {\( u \)};

	% Add the lines of the triangle
	% Add them now so they are over the description boxes
	\draw[-] (l) -- (r);
	\draw[-] (l) -- (m);
	\draw[-] (l) -- (o);
	\draw[-] (r) -- (m);
	\draw[-] (r) -- (o);
	\draw[-] (o) -- (m);

	% Edge Path
	\draw[->,shorten >=\arrowSep,shorten <=\arrowSep,color=blue1] (l) edge[bend left=-\bendAngle] node [description] {\( 1 \)} (r);

	\draw[->,shorten >=\arrowSep,shorten <=\arrowSep,color=blue2] (r) edge[bend left=-\bendAngle] node [description] {\( 2 \)} (m);
	% Add the line again so they are over the description boxes
	\draw[-] (r) -- (m);

	\draw[->,shorten >=\arrowSep,shorten <=\arrowSep,color=blue3] (m) to[bend right=\bendAngle] node [description] {\( 3 \)} (o);

	\draw[->,shorten >=\arrowSep,shorten <=\arrowSep,color=blue4] (o) to[bend left=\bendAngle] node [description] {\( 4 \)} (l);

	\draw[->,shorten >=\arrowSep,shorten <=\arrowSep,color=blue5] (l) to[bend right=\bendAngle] node [description] {\( 5 \)} (m);

	\draw[->,shorten >=\arrowSep,shorten <=\arrowSep,color=blue6] (m) to[bend left=\bendAngle] node [description] {\( 6 \)} (o);
	% Add the lines of the triangle
	% Add them now so they are over the description boxes
	\draw[-] (m) -- (o);

	\draw[->,shorten >=\arrowSep,shorten <=\arrowSep,transparent!0,shade path={left color=white, right color=blue8}] (o) to[bend right=\bendAngle] (l);
\end{tikzpicture}
}
	\end{figure}
	The set \( \mathcal E(\secOnQuot_1, \secOnQuot_2) \) is given by
	\[
	\mathcal E(\secOnQuot_1, \secOnQuot_2) = [0,3] \cup [5,7] \cup \ldots.
	\]
	and \( \mathcal E_0(\secOnQuot_1, \secOnQuot_2) \) is given by
	\[
	\mathcal E_0(\secOnQuot_1, \secOnQuot_2) = [0,3].
	\]
	Both sets can be equiped with the simplicial structure of \( \Rnum_{\geq 0} \).
\end{exmp}

This allows us to define a distance on \( \secQuot \) as follows.

\begin{defi}[Distance on \( \secQuot \)]
	Let \( \lambda \in \coweightLat \) with \( \lambda = \sum_{i=1}^n a_i \varpi_i \).
	We set \( \abs{\lambda} \coloneqq \norm{\lambda}_1 \coloneqq \sum_{i=1}^n \abs{a_i} \).
	For \( \secOnQuot_1, \secOnQuot_2 \in \secQuot \) we define the value
	\[
	k(\secOnQuot_1,\secOnQuot_2) \coloneqq \inf\limits_{\lambda \in \domCoweights} \Set*{\abs{\lambda} \given \lambda \notin \mathcal E_0(\secOnQuot_1, \secOnQuot_2)} \in \Nnum_0 \cup \Set{\infty}.
	\]
	For \( 0 < \vartheta < 1 \) we define the function
	\[
	\secMetricFunc \colon \secQuot \times \secQuot \to \Rnum, \; (\secOnQuot_1, \secOnQuot_2) \mapsto \vartheta^{k(\secOnQuot_1,\secOnQuot_2)}
	\]
	with the convention \( \vartheta^\infty \coloneqq 0 \).
	
	If \( k(\secOnQuot_1, \secOnQuot_2) \in \Nnum_0 \) we call an element \( \lambda \in \domCoweights \) which satisfies \( \abs{\lambda} = k(\secOnQuot_1, \secOnQuot_2) \) and \( \lambda \notin \domCoweights \cap \mathcal E_0(\secOnQuot_1, \secOnQuot_2) \) \emph{distance minimizing}.
\end{defi}

The key observation regarding \( \mathcal E_0(\secOnQuot_1, \secOnQuot_2) \) is the following (recall \thref{rem:sectorLifts}).
Given two sectors \( \secOnQuot_1, \secOnQuot_2 \in \secQuot \) which have the same base point, we can lift these along the covering map \( p \colon \euclbuild \to \quot \).
Let us denote these lifts by \( \lift{\secOnQuot}_1, \lift{\secOnQuot}_2 \).
First we note that the image of such a map defines a (unique) sector in \( \euclbuild \) (\thref{lem:maps_sets_are_the_same}\ref{item:sectorsInEuclBuild}.\ref{item:uniqueSectorImage}).
The second observation we make is
\[
\mathcal E_0(\secOnQuot_1, \secOnQuot_2) = \mathcal E_0(\lift{\secOnQuot}_1, \lift{\secOnQuot}_2) = \mathcal E(\lift{\secOnQuot}_1, \lift{\secOnQuot}_2).
\]
We collect some consequences of this in the following remark.

\begin{rem}\label{rem:distance of sectors}
	\mbox{}
	\begin{enumerate}[1.]
		\item
			The distance between two sectors \( \secOnQuot_1, \secOnQuot_2 \in \secQuot \) (with the same base point \( \secOnQuot_1(0)= \secOnQuot_2 (0) \)) is the same as the distance between the corresponding lifts \( \lift{\secOnQuot}_1, \lift{\secOnQuot}_2 \in \sector(\euclbuild) \).
		
			In the Euclidean building \( \euclbuild \) we do not need the set \( \mathcal E_0 \) and instead can take the minimal \( \mu \in \domCoweights \) such that the image of the two sectors is not equal.
		\item\label{item:conCompZeroConvex}
			The set \( \mathcal E_0(\secOnQuot_1, \secOnQuot_2) \) is a halfspace-convex subset of \( \sector_0 \), as the intersection of two sectors in a Euclidean building is halfspace-convex (in the sense of \thref{defi:halfspaceConvex}).
		\item
			Let \( \secOnQuot_1, \secOnQuot_2 \in \secQuot \) be two sectors and \( \paramA \in \mathcal E_0(\secOnQuot_1, \secOnQuot_2) \cap \domCoweights \).
			For all \( \paramB \in \domCoweights \) with \( \paramB \leq \paramA \) we have \( \paramB \in \mathcal E_0(\secOnQuot_1, \secOnQuot_2) \) as \( \mathcal E_0(\secOnQuot_1, \secOnQuot_2) \) is halfspace-convex (see \cite[Section 7.1]{parkinsonDiss}).
	\end{enumerate}
\end{rem}

\begin{prop}
	The function \( \secMetricFunc \) defines an ultrametric on \( \secQuot \).
\end{prop}

\begin{proof}
	Let \( \secOnQuot_1, \secOnQuot_2, \secOnQuot_3 \in \secQuot \) be three sectors.
	\begin{enumerate}[(i)]
		\item
			\( \secMetric{\secOnQuot_1, \secOnQuot_2} \geq 0 \) follows from the definition.
			
			To show that the map is definite we observe that for \( \secOnQuot_1 = \secOnQuot_2 \) we have \( \mathcal E_0(\secOnQuot_1, \secOnQuot_2) = \sector_0 \).
			Hence \( k(\secOnQuot_1, \secOnQuot_2) = \infty \).
			
			Let \( \secMetric{\secOnQuot_1, \secOnQuot_2} = 0 \).
			That means \( \Set{\abs{\lambda} \given \lambda \notin \mathcal E_0(\secOnQuot_1, \secOnQuot_2) \cap \domCoweights} = \emptyset \).
			Which is equivalent to \( \mathcal E_0(\secOnQuot_1, \secOnQuot_2) \cap \domCoweights = \domCoweights \).
			But as \( \mathcal E_0(\secOnQuot_1, \secOnQuot_2) \) is convex and \( \conv \domCoweights = \sector_0 \) we have \( \mathcal E_0(\secOnQuot_1, \secOnQuot_2) = \sector_0 \).
			Thus we have shown \( \secOnQuot_1 = \secOnQuot_2 \).
		\item
			Symmetry follows from the definition of \( k(\secOnQuot_1, \secOnQuot_2) \).
		\item
			We want to show
			\[
			\secMetric{\secOnQuot_1, \secOnQuot_3} \leq \max \Set{\secMetricFunc(\secOnQuot_1, \secOnQuot_2), \secMetricFunc(\secOnQuot_2, \secOnQuot_3)}
			\]
			or equivalently
			\begin{equation}\label{equ:kInequalityOnQuot}
				k(\secOnQuot_1, \secOnQuot_3) \geq \min \Set{k(\secOnQuot_1, \secOnQuot_2), k(\secOnQuot_2, \secOnQuot_3)}.
			\end{equation}
			By \( x_i \coloneqq \secOnQuot_i(0) \) we denote the base points of the sectors.
			\begin{enumerate}[(i)]
				\item
					If \( x_1 \neq x_3 \), then \( x_1 \neq x_2 \) or \( x_2 \neq x_3 \).
					In this case we have \( k(\secOnQuot_1, \secOnQuot_3) = 0 \) and moreover \( k(\secOnQuot_1, \secOnQuot_2) = 0 \) or \( k(\secOnQuot_2, \secOnQuot_3) = 0 \) as desired.
				\item
					If \( x_1 = x_3 \) we again have two cases.
					\begin{proofcase}
						\item[Case 1 \( x_1 \neq x_2 \)]
							In this case we are done since the right hand side of \eqref{equ:kInequalityOnQuot} is \( 0 \).
						\item[Case 2 All base points are equal]
							In this case, suppose
							\[
							k(\secOnQuot_1, \secOnQuot_3) < \min \Set{k(\secOnQuot_1, \secOnQuot_2), k(\secOnQuot_2, \secOnQuot_3)}.
							\]
							Find \( \lambda \notin \mathcal E_0(\secOnQuot_1, \secOnQuot_3) \cap \domCoweights \) such that \( k(\secOnQuot_1, \secOnQuot_3) = \abs{\lambda} \) and \( \lambda \in \mathcal E_0(\secOnQuot_1, \secOnQuot_2) \cap \domCoweights \) and \( \lambda \in \mathcal E_0(\secOnQuot_2, \secOnQuot_3) \).
							As we have
							\[
							\mathcal E_0(\secOnQuot_1, \secOnQuot_2) \cap \mathcal E_0(\secOnQuot_2, \secOnQuot_3) \subseteq \mathcal E_0(\secOnQuot_1, \secOnQuot_3),
							\]
							this is a contradiction as this shows \( \lambda \in \mathcal E_0(\secOnQuot_1, \secOnQuot_3) \). \qedhere
					\end{proofcase}
			\end{enumerate}
	\end{enumerate}
\end{proof}

\begin{rem}\label{rem:controlOfCoweights}
	Let \( \secOnQuot_1, \secOnQuot_2 \in \secQuot \) be two sectors with \( \secMetric{\secOnQuot_1, \secOnQuot_2} < \vartheta^k \).
	Then every \( \paramA \in \domCoweights \) with \( \abs{\paramA} \leq k \) satisfies \( \mu \in \mathcal E_0(\secOnQuot_1, \secOnQuot_2) \).
\end{rem}

\begin{defi}\label{defi:ballSizeN}
	Given a natural number \( n \in \Nnum_0 \) and a sector \( \secOnQuot \in \secQuot \) we define the \emph{ball around \( \secOnQuot \) of size \( n \)} as
	\[
	\ball_n(\secOnQuot) \coloneqq \overline{\ball}_{\vartheta^n}(\secOnQuot) \coloneqq \Set{\secOnQuot^\prime \in \secQuot \given d_\vartheta(\secOnQuot, \secOnQuot^\prime) \leq \vartheta^n}.
	\]
\end{defi}

\begin{rem}\label{rem:propOfUltrametric}
	The following remarks are either straightforward from the definitions or standard facts about ultrametric spaces (see \cite[Chapter 1]{perezSchikhof10}).
	\begin{enumerate}[(i)]
		\item
			For any sector \( \secOnQuot \in \secQuot \) the ball \( \ball_n(\secOnQuot) \) does not depend on the parameter \( \vartheta \in (0,1) \).
		\item
			We have the following inclusions
			\[
			\ldots \subseteq \ball_2(\secOnQuot) \subseteq \ball_1(\secOnQuot) \subseteq \ball_0(\secOnQuot) = \secQuot.
			\]
		\item
			Let \( n \in \Nnum_0 \) and \( \secOnQuot \in \secQuot \) be fixed.
			If \( \secOnQuot^\prime \in \ball_n(\secOnQuot) \), then \( \ball_n(\secOnQuot) = \ball_n(\secOnQuot^\prime) \).
			In other words every point in a ball is the center of the ball.
		\item
			Two balls around \( \secOnQuot \) and \( \secOnQuot^\prime \) of size \( n \) are either equal or distinct.
		\item\label{item:diffBallsSameDist}
			Let \( \secOnQuot_1, \secOnQuot_2 \in \secQuot \) be two sectors, \( n \in \Nnum_0 \) and \( \secOnQuot_1^\prime \in \ball_n(\secOnQuot_1), \secOnQuot_2^\prime \in \ball_n(\secOnQuot_2) \).
			If \( \ball_n(\secOnQuot_1) \neq \ball_n(\secOnQuot_2) \), then \( \secMetric{\secOnQuot_1^\prime, \secOnQuot_2^\prime} = \secMetric{\secOnQuot_1, \secOnQuot_2} \).
		\item
			Let \( n \in \Nnum_0 \).
			There exists a family of sectors \( (\secOnQuot_\iota)_{\iota \in J} \) such that
			\begin{equation}\label{equ:ballCovering}
				\secQuot = \bigsqcup\limits_{\iota \in J} \ball_n(\secOnQuot_\iota).
			\end{equation}
	\end{enumerate}
\end{rem}

\begin{lem}\label{lem:finitelyBallsCover}
	Let \( \quot \) be a compact local building.
	The index set \( J \) in \eqref{equ:ballCovering} can be chosen finite, i.e, for every \( n \in \Nnum_0 \) there exists a finite number of balls of size \( n \) such that \( \secQuot \) is covered by the balls.
\end{lem}

\begin{proof}
	For \( n = 0 \) we have \( \secQuot = \ball_0(\secOnQuot) \) for any \( \secOnQuot \in \secQuot \).
	As \( \quot \) is compact it has a finite number of vertices.
	In particular only a finite number of good vertices which can be base points of sectors.
	
	Due to the simplicial structure a map \( \secOnQuot \colon \sector_0 \to \quot \) is determined by the image of the vertices of \( \sector_0 \).
	Take all vertices \( Y_n \coloneqq \Set{\paramA \in \domCoweights \given \abs{\paramA} \leq n} \) and form the halfspace-convex hull of \( Y_n \) in \( \sector_0 \).
	This subset is bounded and has finitely many vertices in \( \sector_0 \) (note not all vertices of \( \sector_0 \) are in \( \domCoweights \)).
	
	For two sectors \( \secOnQuot_1, \secOnQuot_2 \in \secQuot \) to have distance \( \secMetric{\secOnQuot_1, \secOnQuot_2} > \vartheta^n \) they have to differ at least on a subset of \( Y_n \).
	But as this is a finite set and \( \quot \) only has finitely many vertices, there are only finitely many possible maps from \( Y_n \to \quot \).
	Hence we can cover \( \secQuot \) by finitely many balls of size \( n \).
\end{proof}

Finally we can introduce the following helpful shorthand notation:
If \( \sectorMap \in \sector(\euclbuild) \) is a sector and \( \apartment \supset \Image \sectorMap \) an apartment, then let \( \psi \colon \apartment \to\euclCox \) be the unique type-rotating isomorphism such that \( \psi(\Image \sectorMap) = \sector_0 \) given by \thref{lem:maps_sets_are_the_same}\ref{item:sectorsInEuclBuild}.\ref{item:apartmentTypeRotSector} (in other words, \( \psi\circ \sectorMap = \id_{\sector_0} \)).
Then recall that for \( \mu \in \coweightLat \) the translation \( \transl_\mu \colon \euclCox \to \euclCox, \; x \mapsto x + \mu \) operates on \( \euclCox \).
As \( \transl_\mu \in \extWeyl \), this is a type-rotating automorphism and conjugating this action by \( \psi \) we obtain a type-rotating automorphism of the apartment \( \apartment \).
We write for \( x \in \apartment \), \( \mu \in \coweightLat \):
\begin{equation}\label{eq:apartment_+}
	x +_{\apartment, \sectorMap} \mu = \psi^{-1}(\transl_\mu(\psi(x))) \quad \text{and} \quad x -_{\apartment, \sectorMap} \mu = \psi^{-1}(\transl_{-\mu}(\psi(x))).
\end{equation}

\subsection{Shift Operators}

Note that by \thref{rem:euclCoxAndCoweights}\ref{item:domCoweightsInSector} we have \( \domCoweights \subseteq \sector_0 \).
Hence every translation \( \transl_\mu \) for \( \mu \in \domCoweights \) maps \( \transl_\mu(\sector_0) \subseteq \sector_0 \).
We use this to define shift operators.

\begin{defi}[Shift Operators on \( \secQuot \)]
	For \( \paramA \in \domCoweights \) we define the \emph{shift operator on \( \secQuot \)} as
	\[
	\sigma_\mu \colon \secQuot \to \secQuot, \; \secOnQuot \mapsto \secOnQuot \circ \transl_\paramA.
	\]
\end{defi}

\begin{lem}\label{lem:shift_image}
	Let \( \secOnQuot \in \secQuot \) and \( \mu, \mu_1, \mu_2 \in \domCoweights \).
	Then
	\begin{enumerate}[(i)]
		\item\label{item:shiftContainedInSector}
			\( \Image \sigma_\mu(s) \subseteq \Image s \), and
		\item 
			\(
			\sigma_{\mu_1} \circ \sigma_{\mu_2} = \sigma_{\mu_1 + \mu_2} = \sigma_{\mu_2} \circ \sigma_{\mu_1}.
			\)
	\end{enumerate}
\end{lem}

\begin{proof}
	\mbox{}
	\begin{enumerate}[(i)]
		\item
			For \( y \in \Image \sigma_\paramA(\secOnQuot) \) let \( x \in \sector_0 \) such that \( \sigma_\mu(s)(x) = y \).
			That means
			\[
			y = \secOnQuot \circ \transl_\paramA(x) = \secOnQuot(\transl_\paramA(x)).
			\]
			But as \( \transl_\paramA(x) \in \sector_0 \) we have \( y \in \Image \secOnQuot \).
		\item
			Let \( \paramA_1, \paramA_2 \in \domCoweights \).
			As \( \transl_{\paramA_1} \circ \transl_{\paramA_2} = \transl_{\paramA_1 + \paramA_2} \) we find
			\[
			\sigma_{\paramA_1} \circ \sigma_{\paramA_2}(\secOnQuot) = \sigma_{\paramA_1}(\secOnQuot \circ \transl_{\paramA_2}) = \secOnQuot \circ \transl_{\paramA_2} \circ \transl_{\paramA_1} = \secOnQuot \circ \transl_{\paramA_2 + \paramA_1} = \sigma_{\paramA_1 + \paramA_2}(\secOnQuot).
			\]
			Now the rest follows. \qedhere
	\end{enumerate}
\end{proof}

\begin{rem}\label{rem:connectedCompShift}
	Let \( \secOnQuot_1, \secOnQuot_2 \in \secQuot \) be two sectors and \( \paramA \in \domCoweights \).
	Then
	\[
	\mathcal E(\sigma_\paramA(\secOnQuot_1), \sigma_\paramA(\secOnQuot_2)) = (\mathcal E(\secOnQuot_1, \secOnQuot_2) - \paramA) \cap \sector_0.
	\]
	Moreover, for \( \paramA \in \mathcal E_0(\secOnQuot_1, \secOnQuot_2) \) we have
	\begin{equation}\label{eq:domain_under_shift}
	\mathcal E_0(\sigma_\paramA(\secOnQuot_1), \sigma_\paramA(\secOnQuot_2)) = (\mathcal E_0(\secOnQuot_1, \secOnQuot_2) - \paramA) \cap \sector_0.
	\end{equation}
\end{rem}

\begin{cor}\label{cor:mathcalDSmaller}
	For two sectors \( \secOnQuot_1, \secOnQuot_2 \in \secQuot \) and \( \mu \in \mathcal E_0(\secOnQuot_1, \secOnQuot_2) \cap \domCoweights \) we have
	\begin{equation}\label{equ:shiftBereichKleiner}
		\mathcal E_0(\sigma_\paramA(\secOnQuot_1), \sigma_\paramA(\secOnQuot_2)) \subseteq \mathcal E_0(\secOnQuot_1, \secOnQuot_2).
	\end{equation}
	In particular, for \( \secMetric{\secOnQuot_1, \secOnQuot_2} < \vartheta^k \) and \( \abs{\paramA} \leq k \) \Cref{equ:shiftBereichKleiner} holds.
\end{cor}

\begin{cor}\label{cor:expandingShift}
	For two sectors \( \secOnQuot_1, \secOnQuot_2 \in \secQuot \) and \( \paramA \in \mathcal E_0(\secOnQuot_1, \secOnQuot_2) \cap \domCoweights \) we have
	\[
	\secMetric{\secOnQuot_1, \secOnQuot_2} \leq \secMetric{\sigma_\paramA(\secOnQuot_1), \sigma_\paramA(\secOnQuot_2)}.
	\]
\end{cor}

The first main goal of this section is to prove the following proposition.

\begin{prop}\label{prop:shiftKeyInequality}
	Let \( \secOnQuot_1, \secOnQuot_2 \in \secQuot \) be two sectors and \( \paramA \in \mathcal E_0(\secOnQuot_1, \secOnQuot_2) \cap \innerCoweights \) a strongly dominant coweight.
	Then
	\[
	\secMetric{\secOnQuot_1, \secOnQuot_2} \leq \vartheta \secMetric{\sigma_\paramA(\secOnQuot_1), \sigma_\paramA(\secOnQuot_2)}.
	\]
\end{prop}

\begin{lem}\label{lem:distanceAttainedOnRay}
	Let \( \secOnQuot_1, \secOnQuot_2 \in \secQuot \) be two sectors and \( \varpi_1,\ldots,\varpi_n \in \domCoweights \) a basis of the coweight lattice \( \coweightLat \) which generates the monoid \( \domCoweights \).
	Then there exists \( k \in \Nnum_0 \) and some \( i = 1,\ldots,n \) such that
	\begin{enumerate}[(i)]
		\item
			\( \secMetric{\secOnQuot_1, \secOnQuot_2} = \vartheta^k \) and
		\item
			\( k \varpi_i \notin \mathcal E_0(\secOnQuot_1, \secOnQuot_2) \cap \domCoweights \).
	\end{enumerate}
	In other words, there exists a distance minimizing element on some ray \( \Nnum_0 \varpi_i \).
\end{lem}

\begin{proof}
	The crucial part of the proof is the observation that \( \mathcal E_0(\secOnQuot_1, \secOnQuot_2) \) is a (halfspace-) convex subset of \( \sector_0 \) (see \thref{rem:distance of sectors}).
	Hence \( \mathcal E_0(\secOnQuot_1, \secOnQuot_2) \cap \domCoweights \) is a (halfspace-) convex subset of \( \domCoweights \) as described in \cite[Section 7.1]{parkinsonDiss}.
	
	Let us assume that \( \mu \notin \mathcal E_0(\secOnQuot_1, \secOnQuot_2) \cap \domCoweights \) is an element minimizing the distance and that there is no distance minimizing element in some ray \( \Nnum_0 \varpi_i \).
	Write \( \mu = \sum_{i=1}^n a_i \varpi_i \) with \( a_i \in \Nnum_0 \) and set \( k = \sum_{i=1}^n a_i \).
	By assumption \( k \varpi_j \in \mathcal E_0(\secOnQuot_1, \secOnQuot_2) \cap P_+^\vee \) for all \( j = 1,\ldots,n \).
	As \( \mathcal E_0(\secOnQuot_1, \secOnQuot_2) \) is (halfspace) convex the convex hull \( \conv \Set{0, k \varpi_j \given j = 1,\ldots,n} \subseteq \mathcal E_0(\secOnQuot_1, \secOnQuot_2) \).
	But as \( \mu \) is in the metric convex hull of \( \Set{0, k \varpi_j \given j = 1,\ldots,n} \) it is in the halfspace convex hull as well, i.e., \( \mu \in \conv \Set{0, k \varpi_j \given j=1,\ldots,n} \).
\end{proof}

The previous lemma gives rise to the following definition.

\begin{defi}
	Let \( \varpi_1,\ldots,\varpi_n \in \domCoweights \) be a basis of the coweight lattice \( \coweightLat \) which generates the monoid \( \domCoweights \).
	We define the \emph{distance in the direction \( \varpi_i \)} as follows:
	\[
	\dircSecMetric{\varpi_i}{\secOnQuot_1, \secOnQuot_2} \coloneqq \vartheta^{k_i(\secOnQuot_1, \secOnQuot_2)},
	\]
	where
	\[
	k_i(\secOnQuot_1, \secOnQuot_2) \coloneqq \inf \Set*{\ell \in \Nnum_0 \given \ell \cdot \varpi_i \notin \mathcal E_0(\secOnQuot_1, \secOnQuot_2)} \in \Nnum_0 \cup \Set{\infty}.
	\]
	As before we set \( \vartheta^\infty \coloneqq 0 \).
\end{defi}

\begin{rem}\label{rem:metricMinOnBoundary}
	\mbox{}
	\begin{enumerate}[(i)]
		\item
			The map \( \secMetricFunc^{\varpi_i} \) is not a metric.
		\item
			By \thref{lem:distanceAttainedOnRay} we can use the different directions to compute the distance of two sectors as follows
			\[
			\secMetric{\secOnQuot_1, \secOnQuot_2} = \max_{i=1,\ldots,n} \dircSecMetric{\varpi_i}{\secOnQuot_1, \secOnQuot_2}.
			\]
	\end{enumerate}
\end{rem}

In order to prove \thref{prop:shiftKeyInequality} we need another lemma which gives us information about how the distance in the different directions changes when applying a shift along on the rays spanning the cone \( \domCoweights \).

\begin{lem}\label{lem:splitShiftInDirections}
	As before let \( \varpi_1,\ldots,\varpi_n \in \domCoweights \) be a basis of the coweight lattice \( \coweightLat \) which generates the monoid \( \domCoweights \).
	Take \( i \in \Set{1,\ldots,n} \) and let \( \secOnQuot_1, \secOnQuot_2 \in \secQuot \) be two sectors with \( \varpi_i \in \mathcal E_0(\secOnQuot_1, \secOnQuot_2) \).
	It holds
	\begin{enumerate}[(i)]
		\item
			\( \dircSecMetric{\varpi_i}{\secOnQuot_1, \secOnQuot_2} = \vartheta \dircSecMetric{\varpi_i}{\sigma_{\varpi_i}(\secOnQuot_1), \sigma_{\varpi_i}(\secOnQuot_2)} \) and
		\item
			\( \dircSecMetric{\varpi_j}{\secOnQuot_1, \secOnQuot_2} \leq \dircSecMetric{\varpi_j}{\sigma_{\varpi_i}(\secOnQuot_1), \sigma_{\varpi_i}(\secOnQuot_2)} \) for all \( j \neq i \).
	\end{enumerate}
\end{lem}

\begin{proof}
	Let \( \secOnQuot_1, \secOnQuot_2 \in \secQuot \) two sectors with \( \varpi_i \in \mathcal E_0(\secOnQuot_1, \secOnQuot_2) \).
	To shorten the notation we write \( k_j \coloneqq k_j(\secOnQuot_1, \secOnQuot_2) \).
	Our assumption tells us that \( k_i > 1 \).
	\begin{enumerate}[(i)]
		\item
			The case \( k_i = \infty \) is trivial as we have \( k_i(\sigma_{\varpi_i}(\secOnQuot_1), \sigma_{\varpi_i}(\secOnQuot_2)) = \infty \) as well and both sides are equal to \( 0 \).
			In the case where \( 1 < k_i < \infty \) we find by \eqref{eq:domain_under_shift} \( (k_i-1) \varpi_i \notin \mathcal E_0(\sigma_{\varpi_i}(\secOnQuot_1), \sigma_{\varpi_i}(\secOnQuot_2)) \) and this element is minimal.
		\item
			Follows from \thref{cor:mathcalDSmaller}. \qedhere
	\end{enumerate}
\end{proof}

Now everything is set up to give a proof of \thref{prop:shiftKeyInequality}.

\begin{proof}[Proof of \thref{prop:shiftKeyInequality}]
	Let \( \secOnQuot_1, \secOnQuot_2 \in \secQuot \) be two sectors with \( \paramA \in \mathcal E_0(\secOnQuot_1, \secOnQuot_2) \).
	
	First we show the claim for \( \paramA = \sum_{i=1}^n \varpi_i \) which is the unique smallest element in \( \innerCoweights \).
	As \( \varpi_i \leq \paramA \) for every \( i = 1,\ldots, n \) we find \( \varpi_i \in \mathcal E_0(\secOnQuot_1, \secOnQuot_2) \).
		For \( j \in \Set{1,\ldots,n} \) fixed we decompose
	\[
	\sigma_\paramA = \sigma_{\varpi_j} \circ \prod\limits_{i \neq j} \sigma_{\varpi_i}.
	\]
	 Note that by \eqref{eq:domain_under_shift} after every application of \( \sigma_{\varpi_i} \) to the sectors the element \( \varpi_j \) stays in the component \( \mathcal E_0 \).
\thref{lem:splitShiftInDirections} implies
	\[
	\dircSecMetric{\varpi_j}{\secOnQuot_1, \secOnQuot_2} \leq \vartheta \dircSecMetric{\varpi_j}{\sigma_{\paramA}(\secOnQuot_1), \sigma_{\paramA}(\secOnQuot_2)}.
	\]
We combine the estimates for all \( j = 1,\ldots,n \) and obtain
	\begin{align*}
		\secMetric{\secOnQuot_1, \secOnQuot_2} &= \max \dircSecMetric{\varpi_j}{\secOnQuot_1, \secOnQuot_2} 
		\leq \max \vartheta \dircSecMetric{\varpi_j}{\sigma_{\paramA}(\secOnQuot_1), \sigma_{\paramA}(\secOnQuot_2)} 
		= \vartheta \max \dircSecMetric{\varpi_j}{\sigma_{\paramA}(\secOnQuot_1), \sigma_{\paramA}(\secOnQuot_2)} \\
		&= \vartheta \secMetric{\sigma_{\paramA}(\secOnQuot_1), \sigma_{\paramA}(\secOnQuot_2)}.
	\end{align*}
	As \( \mu \leq \lambda \) for all \( \lambda \in \innerCoweights \) we can use this inequality and apply \thref{cor:expandingShift} to get the proposition.
\end{proof}

The second main goal of this section is to get a better understanding of the preimages of the shift maps.
We extend the notation \eqref{eq:apartment_+} to sectors on buildings:
Let \( \sectorMap \in \sector(\euclbuild) \) be a sector, \( \apartment \) an apartment with \( \Image \sectorMap \subset \apartment \) and \( \mu \in \coweightLat \).
Then we define a new sector
\[
\sectorMap \pm_{\apartment} \mu \colon \sector_0 \to \euclbuild, \; x \mapsto \sectorMap(x) \pm_{\apartment, \sectorMap} \mu.
\]
With this notation we can describe the preimages of the shift maps on \( \euclbuild \) as follows.

\begin{lem}\label{lem:preimage_by_apartments}
	Let \( \sectorMap \in \sector(\euclbuild) \) a sector in the Euclidean building \( \euclbuild \) and \( \mu \in \domCoweights \).
	Then
	\[
	\sigma_\mu^{-1}(\sectorMap) = \bigcup_{\apartment \in\apartSys, \, \Image s\subset \apartment} \sectorMap -_{\apartment} \mu
	\]
\end{lem}

\begin{proof}
	The inclusion \quoted{\( \supseteq \)} is clear by definition of \( -_{\apartment} \).
	For the converse inclusion let \( \sectorMap^\prime \in\sigma^{-1}_\mu(\sectorMap) \).
	We choose an apartment \( \apartment \) that contains \( \Image(\sectorMap^\prime) \) and thus by \thref{lem:shift_image} also \( \Image(\sectorMap) \). Then by definition of \( +_{\apartment} \) we have \( \sectorMap^\prime +_{\apartment} \mu = \sectorMap \).
\end{proof}

\begin{lem}\label{lem:preimageOfSectorInBuild}
	Let \( \sectorMap \in \sector(\euclbuild) \) be sector in the Euclidean building \( \euclbuild \).
	For any \( \paramA \in \domCoweights \) the cardinality \( \# \sigma_\mu^{-1}(\sectorMap) \) does not depend on \( \sectorMap \) but only on \( \paramA \).
	It is given by
	\begin{equation}\label{equ:preimageNumber}
		M_\paramA \coloneqq \# \sigma_\paramA^{-1}(\sectorMap) = q_{\transl_{-\paramA}} = q_{\transl_\paramA}.
	\end{equation}
\end{lem}

\begin{proof}
	Let \( \sectorMap \colon \sector_0 \to \euclbuild \) be a sector and \( \sigma \in \trAutRoot(I) \) the corresponding permutation.
Recall the combinatorial balls. As the building is strongly regular we have \( \mathcal C_w(C) = \mathcal C_{\sigma(w)}(C) \) for all \( w \in \affWeyl \) and this does not depend on the chamber \( C \), i.e., \( q_w = q_{\sigma(w)} \).
	By \thref{lem:preimage_by_apartments} we know that the preimage of the shift can be computed by taking an apartment \( \apartment \) with \( \Image \sectorMap \subseteq \apartment \).
	
	Within any apartment containing \( \Image \sectorMap \) there is a unique chamber which gets mapped to the base chamber of \( \sectorMap \) when the translation by \( \paramA \) is applied.
	As this chamber has distance \( \transl_{-\paramA} \) we know that there are exactly \( q_{\transl_{-\paramA}} \) such chambers in the entire building (this follows from the \thref{defi:extendedParameters}).
	From \thref{lem:extWeylParaInverse} we know that \( q_{\transl_{-\paramA}} = q_{\transl_{-\paramA}^{-1}} = q_{\transl_{\paramA}} \).
\end{proof}

\begin{rem}\label{rem:oppositeSectors}
	Let \( \euclbuild \) be a Euclidean building and \( \omega, \widetilde{\omega} \in \boundary \) be two opposite chambers in the building at infinity.
	By \thref{rem:propertiesOfBoundary}\ref{item:oppoChambersDefineApartment} this pair defines a unique apartment \( \apartment(\omega, \widetilde{\omega}) \subseteq \euclbuild \).
	Therefore, we can find a good vertex \( x \in \apartment(\omega, \widetilde{\omega}) \subseteq \euclbuild \) which is the base vertex for two sectors in the classes of \( \omega \) and \( \widetilde{\omega} \) which lie entirely in \( \apartment(\omega, \widetilde{\omega}) \).
	We denote them by \( \sector^x(\omega), \sector^x(\widetilde{\omega}) \).
	
	In this case we say that \( \sector^x(\omega) \) and \( \sector^x(\widetilde{\omega}) \) are \emph{opposite sectors}.
	More generally we call two sectors \( \sectorMap_1, \sectorMap_2 \in \sector(\euclbuild) \) \emph{opposite} sectors if their corresponding chambers at infinity are opposite, their images lie in a common apartment and they have the same base vertex \( \sectorMap_1(0) = x_0 = \sectorMap_2(0) \).
	
	In this case the base chambers \( \sectorMap_1(C_0) \) and \( \sectorMap_2(C_0) \) are opposite chambers in the link of \( x_0 \) (see \thref{rem:sphericalResChambSys}\ref{item:linkOfVertex}) \( \Res_{I \setminus \Set{\tau(x_0)}}(\sectorMap_1(C_0)) \).
\end{rem}

\begin{lem}\label{lem:jointHalfapartmentForSectors}
	Let \( \sectorMap_1, \sectorMap_2 \in \sector(\euclbuild) \) be sectors in \( \euclbuild \) with \( \sectorMap_1(0) = \sectorMap_2(0)= x_0 \) and \( \sectorMap_1(C_0) = \sectorMap_2(C_0) \).
	Then for every apartment \( \apartment_1 \subseteq \euclbuild \) with \( \Image \sectorMap_1 \subseteq \apartment_1 \) there exists an apartment \( \apartment_2 \) with \( \Image \sectorMap_2 \subseteq \apartment_2 \) such that \( x_0 -_{\apartment_1, \sectorMap_1} \mu = x_0 -_{\apartment_2, \sectorMap_2}\mu \) for all \( \mu \in \domCoweights \).
\end{lem}

\begin{proof}
	Let \( \omega_1,\omega_2 \in \boundary \) be the chambers at infinity of \( \sectorMap_1, \sectorMap_2 \).
	Choose an apartment \( \apartment \subseteq \euclbuild \) with \( \Image \sectorMap_1 \subseteq \apartment_1 \) and take the sector \( \widetilde{\sectorMap} \) opposite to \(\Image \sectorMap_1 \) in \( \apartment \).
	Let \( \widetilde{\omega} \) be the chamber at infinity of this sector.
	As \( \sectorMap_2 \) and \( \widetilde{\sectorMap} \) have the same base vertex and their base chambers are opposite chambers in the link of \( \sectorMap_1(0) \) we conclude that \( \widetilde{\omega} \) and \( \omega_2 \) are opposite chambers in the boundary.
	The desired apartment \( \apartment_2 \) is the unique apartment induced by \( (\omega_2, \widetilde{\omega}) \).
	As both apartments \( \apartment_1 \) and \( \apartment_2 \) contain \( \Image \widetilde{\sectorMap} \) these are precisely the points \( x_0 +_{\apartment_1} (-\mu) \), respectively \( x_0 +_{\apartment_2} (-\mu) \) and thus the points coincide.
\end{proof}

\begin{lem}
	Let \( \sectorMap_1, \sectorMap_2 \in \sector(\euclbuild) \) be two sectors in \( \euclbuild \) with \( \sectorMap_1(0) = \sectorMap_2(0) \) and \( \sectorMap_1(C_0) = \sectorMap_2(C_0) \).
	Let \( \paramA \in \domCoweights \).
	For every \( \sectorMap_1^\prime \in \sigma_\paramA^{-1}(\sectorMap_1) \) there exists a unique \( \sectorMap_2^\prime \in \sigma_\paramA^{-1}(\sectorMap_2) \) such that \( \paramA \in \mathcal E_0(\sectorMap_1^\prime, \sectorMap_2^\prime) \).
\end{lem}

\begin{proof}
	Take an arbitrary \( \sectorMap_1^\prime \in \sigma_\paramA^{-1}(\sectorMap_1) \).
	Let \( \apartment_1 \) be an apartment containing \( \Image \sectorMap_1 \) and \( \sectorMap_1^\prime(0) \).
	Then \( \sectorMap_1^\prime = \sectorMap_1 +_{\apartment_1} (-\mu) \).
	Now let \( \apartment_2 \) be an apartment as given by \thref{lem:jointHalfapartmentForSectors} (with respect to \( \sectorMap_1, \sectorMap_2 \)).
	Then we can define \( \sectorMap_2^\prime \coloneqq \sectorMap_2+_{\apartment_2} (-\paramA) \).
	By definition \( \sectorMap_2^\prime \in\sigma_\paramA^{-1}(\sectorMap_2) \) and by \thref{lem:jointHalfapartmentForSectors} \( \sectorMap_1^\prime(0) = \sectorMap_2^\prime(0) \).
	Furthermore, by construction we have \( \sectorMap_1^\prime(\paramA) = \sectorMap_1(0) = \sectorMap_2(0) = \sectorMap_2^\prime(\paramA) \).
	Hence \( \paramA \in \mathcal E_0(\sectorMap_1, \sectorMap_2) \).
	The uniqueness follows from \thref{rem:propertiesOfBoundary}\ref{item:sphereFromPoint} and \thref{lem:shift_image}\ref{item:shiftContainedInSector} above.
\end{proof}

\begin{lem}\label{lem:preimageUnderLift}
	Let \( \paramA \in \domCoweights \) and \( \secOnQuot \in \secQuot \) a sector.
	Let \( \lift{\secOnQuot} \) be the lift of \( \secOnQuot \) of a chosen base point \( \lift{\secOnQuot}(0) = x_0 \in p^{-1}(\secOnQuot(0)) \).
	Then the map
	\[
	\sigma_\paramA^{-1}(\lift{\secOnQuot}) \to \sigma_\paramA^{-1}(\secOnQuot), \; \sectorMap^\prime \mapsto p \circ \sectorMap^\prime.
	\]
	is a bijection.
	In particular, the cardinalities \( \# \sigma_\paramA^{-1}(\lift{\secOnQuot}) = \# \sigma_\paramA^{-1}(\secOnQuot) \) in the quotient and in the building agree.
\end{lem}

\begin{proof}
	The lemma is a consequence of the fact that \( p \) is a covering map and \( \sector_0 \) and \( \euclbuild \) are simply-connected.
	Let us describe the details:

	First note that the map is well-defined because
	\[
	p \circ \sectorMap^\prime \circ \transl_\paramA = p \circ \lift{\secOnQuot} = \secOnQuot.
	\]
	For the injectivity, let \( \tilde \sectorMap_1, \tilde \sectorMap_2 \in \sigma_\paramA^{-1}(\lift{\secOnQuot}) \) be two sectors which map to the same sector \( \secOnQuot^\prime\in\mathcal S(\mathcal C) \).
	But this means by definition that they are both lifts of \( \secOnQuot^\prime \). 
	As the points \( \sectorMap_1(\paramA) = \sectorMap_2(\paramA) = x_0 \) agree and both \( \sector_0 \) and \( \euclbuild \) are simply-connected they must coincide.
	Take the point \( \paramA \in \sector_0 \) as the base point for the lift.
	As all spaces we are considering are simply-connected the lifts are unique.
	By construction we have \( \sectorMap_1(\paramA) = \sectorMap_2(\paramA) \) and therefore the lifts have to be equal.
	
	For the surjectivity let \( \secOnQuot^\prime \in \sigma_\paramA^{-1}(\secOnQuot) \) and \( \lift{\secOnQuot^\prime}\) be the unique lift with \( \lift{\secOnQuot^\prime}(\mu) = x_0\).
	Then
	\[
	p \circ \lift{\secOnQuot^\prime} \circ \transl_\paramA = \secOnQuot^\prime \circ \transl_\paramA = \secOnQuot.
	\]
	This shows that \( \lift{\secOnQuot^\prime} \circ \transl_\paramA \) is a lift of \( \secOnQuot \) with \( \lift{\secOnQuot^\prime} \circ \transl_\paramA (0) = x_0 \), hence \( \sigma_\mu(\lift{\secOnQuot^\prime}) = \lift{\secOnQuot} \).
\end{proof}

\begin{rem}
	From the \thref{lem:preimageOfSectorInBuild,lem:preimageUnderLift} we know that for \( \paramA \in \domCoweights \) and a sector \( \secOnQuot \in \secQuot \) we have:
	\begin{enumerate}[(i)]
		\item
			The cardinality \( \# \sigma_\paramA^{-1}(\secOnQuot) \) does not depend on \( \secOnQuot \).
		\item
			The cardinality \( \# \sigma_\paramA^{-1}(\secOnQuot) = q_{t_{-\paramA}} = q_{t_\paramA} \) can be computed explicitly (see \cite[Appendix B.1]{parkinsonDiss}).
	\end{enumerate}
\end{rem}

\begin{prop}\label{prop:basePointProp}
	Let \( \secOnQuot_1, \secOnQuot_2 \in \secQuot \) be two sectors with \( \secOnQuot_1(0) = \secOnQuot_2(0) \) and \( \secOnQuot_1(C_0) = \secOnQuot_2(C_0) \).
	For every \( \mu \in \domCoweights \) there exist pairs of sectors \( T \subseteq \sigma_\paramA^{-1}(\secOnQuot_1) \times \sigma_\paramA^{-1}(\secOnQuot_2) \) such that
	\begin{enumerate}[(i)]
		\item
			for all pairs \( (\secOnQuot_1^\prime, \secOnQuot_2^\prime) \in T \) it holds \( \paramA \in \mathcal E_0(\secOnQuot_1^\prime, \secOnQuot_2^\prime) \) (and hence \( 0 \in \mathcal E_0(\secOnQuot_1^\prime, \secOnQuot_2^\prime) \) as the set is non-empty).
		\item
			We have
			\[
			\sigma_\paramA^{-1}(\secOnQuot_1) = \bigcup\limits_{(\secOnQuot_1^\prime, \secOnQuot_2^\prime) \in T} \Set{\secOnQuot_1^\prime} \quad \text{and} \quad \sigma_\paramA^{-1}(\secOnQuot_2) = \bigcup\limits_{(\secOnQuot_1^\prime, \secOnQuot_2^\prime) \in T} \Set{\secOnQuot_2^\prime}
			\]
		\item
			as well as \( \abs{T} = M_\paramA \).
	\end{enumerate}
\end{prop}

\begin{proof}
	We take the maps and lift them to the building.
	Then the assertion follows from \thref{lem:preimageUnderLift}.
\end{proof}

\section{Transfer Operator}\label{sec:Transfer Operator}
In \thref{sec:dynaSetting} we introduced the dynamics of the monoid \( \domCoweights \) acting by the shift operators \( (\sigma_\paramA)_{\paramA \in \domCoweights} \) on the space of sectors \( \secQuot \) on a local building \( \quot \).
In this section we define and analyze transfer operators on \( \chambSys \).

\begin{defi}[Transfer Operator on \( \chambSys \)]\label{defi:transferOpOnCs}
	For \( \paramA \in \domCoweights \) and a local building \( \chambSys \), the \emph{transfer operator on \( \chambSys \)} associated with $\mu$ is the operator \( \tOp_\paramA \colon \Cnum^{\secQuot} \to \Cnum^{\secQuot} \) defined by
	\[
	(\tOp_\mu \varphi)(\secOnQuot) \coloneqq \frac{1}{M_\paramA} \sum\limits_{\sigma_\paramA(\secOnQuot^\prime) = \secOnQuot} \varphi(\secOnQuot^\prime),
	\]
	where \( M_\paramA = \#\sigma^{-1}_\mu(s)\) is the value defined in \thref{equ:preimageNumber}, so that 
	\begin{equation}\label{eq:one_in_spectrum}
		\tOp_\paramA 1 = 1.
	\end{equation}
\end{defi}

As the \( M_\paramA \) are multiplicative (see \cite[p.~44]{parkinsonDiss}) we obtain for all \( \paramA_1, \paramA_2 \in \domCoweights \) and \( s \in \secQuot \) that
\begin{equation}\label{equ:commuRelTOp}
	(\tOp_{\paramA_1} \circ \tOp_{\paramA_2} \varphi)(\secOnQuot) = (\tOp_{\paramA_1 + \paramA_2} \varphi)(\secOnQuot).
\end{equation}

\begin{rem}
	Denote by \( \varpi_1, \ldots, \varpi_n \in \domCoweights \) the fundamental coweights, which form a basis for \( \coweightLat \) and more importantly (freely) generate \( \domCoweights \) as a monoid.
	Hence the operators \( \tOp_{\varpi_i} \) generate the monoid of transfer operators on \( \quot \).
\end{rem}

\begin{rem}\label{rem:llp20}
	As mentioned in the introduction, Lubetzky, Lubotzky, and Parzanchevski \cite{LLP20} study a higher-rank generalization of non-backtracking random walks using the formalism of \emph{\( j \)-dimensional geodesic flows} and \emph{\( k \)-branching operators}.
	While their dynamical systems seem, at first sight, quite different, it is possible (with a little effort) to establish a precise relation to the transfer operators studied in this article.
	For example, in the \( \tilde A_n \) case, it is possible to explicitly conjugate the \( n \)-th power of the \( n \)-dimensional geodesic flow branching operator \cite[Definition~5.1]{LLP20} to our transfer operator in the direction of the fundamental coweight \( \varpi_1 \), namely \( \tOp_{\varpi_1} \) acting on the function space \( F_1 \).
	While we do not pursue this connection further here, we believe it would be interesting to understand whether both approaches can be placed within a common unifying framework.
	Indeed, while \cite{LLP20} also allows one to study geodesic flows on lower-dimensional simplices, their approach appears to miss the possibility of a higher-rank action of the full coweight lattice \( \domCoweights \), instead restricting to the rank-one submonoid \( \Nnum_0 \varpi_1 \).
\end{rem}

\subsection{Function Spaces}
We want to study spectral properties of these operators on suitable function subspaces of \( \Maps(\secQuot, \Cnum) \coloneqq \Cnum^{\secQuot} = \Set{\varphi \colon \secQuot \to \Cnum} \). 

\begin{defi}[Function Spaces]
	\mbox{}
	\begin{enumerate}[(a)]
		\item
			By \( \mathcal B(\secQuot, \Cnum) \subseteq \Maps(\secQuot, \Cnum) \) we denote the set of \emph{bounded functions}.
		\item
			For \( n \in \Nnum_0 \) we define the space \( F_n \) \emph{of locally constant functions of size \( n \)} by
			\[
			F_n \coloneqq \Set*{\varphi \in \mathcal B(\secQuot, \Cnum) \given \forall \secOnQuot \in \secQuot :\ \restr{\varphi}{\ball_n(\secOnQuot)} = \text{const}}.
			\]
		\item
			A function \( \varphi \colon \secQuot \to \Cnum \) is called \emph{Lipschitz} on the metric space \( (\secQuot, \secMetricFunc) \) if 
			\[
			\abs{\varphi}_\vartheta \coloneqq \sup_{\secOnQuot_1 \neq \secOnQuot_2} \frac{\abs{\varphi(\secOnQuot_1) - \varphi(\secOnQuot_2)}}{\secMetric{\secOnQuot_1, \secOnQuot_2}} < \infty.
			\]
			We denote the space of Lipschitz functions by \( \lipschitzFunc \).
	\end{enumerate}
\end{defi}

\begin{rem}\label{rem:fct spaces}
	\mbox{}
	\begin{enumerate}[(i)]
		\item
			In view of \thref{defi:ballSizeN} a function \( \varphi \in \mathcal B(\secQuot, \Cnum) \) is in \( F_n \) if for all \( \secOnQuot_1, \secOnQuot_2 \in \secQuot \) with \( \secMetric{\secOnQuot_1, \secOnQuot_2} \leq \vartheta^n \) we have \( \varphi(\secOnQuot_1) = \varphi(\secOnQuot_2) \).
		\item
			Since the metric \( \secMetricFunc \) is bounded by \( 1 \), every Lipschitz function is bounded by definition.
		\item
			The spaces \( F_n \) are independent of the parameter \( \vartheta \in (0,1) \) of the metric.
		\item
			For the spaces \( F_n \) we have the inclusions
			\[
			F_0 \subseteq F_1 \subseteq F_2 \subseteq \ldots
			\]
		\item
			\( F_0 \) is the space of constant functions.
			In fact, by definition we have \( \ball_0(\secOnQuot) = \secQuot \) for every \( \secOnQuot \in \secQuot \).
	\end{enumerate}
\end{rem}

The supremum norm \( \norm{\cdot}_\infty \) turns the space of bounded functions \( \mathcal B(\secQuot, \Cnum) \) into a Banach space.
Furthermore, it is a standard fact (see e.g.~\cite[Proposition 5.4]{bux2023spectral}), that also \( \lipschitzFunc \) is a Banach space with norm
\[
\norm{\varphi}_\vartheta \coloneqq \abs{\varphi}_\vartheta + \norm{\varphi}_\infty.
\]

\begin{lem}[{{\cite[Lemma 4.1]{bux2023spectral}}}]\label{lem:fmClosedInBounded}
	The space \( F_n \subseteq \mathcal B(\secQuot, \Cnum) \) is a closed subspace.
\end{lem}

\begin{proof}
	Let \( (\varphi_i) \subseteq F_n \) be a Cauchy sequence and let \( \varphi_\infty \in \mathcal B(\secQuot, \Cnum) \) be its limit.
	We need to show that \( \varphi_\infty \in F_n \), i.e., \( \varphi_\infty(\secOnQuot) = \varphi_\infty(\secOnQuot^\prime) \) whenever \( \secOnQuot, \secOnQuot^\prime \in \secQuot\) with \( \secMetric{\secOnQuot, \secOnQuot^\prime} < \vartheta^n \).
	
	By assumption for these sectors we have \( \varphi_i(\secOnQuot) = \varphi_i(\secOnQuot^\prime) \).
	Since the limit with respect to \( \norm{\cdot}_\infty \) is taken pointwise this holds also for \( \varphi_\infty \).
\end{proof}

\begin{lem}[{{\cite[Lemma 5.2]{bux2023spectral}}}]\label{lem:BHW5-2}
	For any \( 0 < \vartheta < 1 \), the space \( F_n \) is a closed subspace of \( \lipschitzFunc \).
\end{lem}

\begin{proof}
	Let \( \varphi \in F_n \) be a function and \( \secOnQuot, \secOnQuot^\prime \in \secQuot \) two admissible sectors.
	If \( \secMetric{\secOnQuot, \secOnQuot^\prime} < \vartheta^n \) then \( \abs{\varphi(\secOnQuot) - \varphi(\secOnQuot^\prime)} = 0 \).
	Else we have
	\[
	\abs{\varphi(\secOnQuot) - \varphi(\secOnQuot^\prime)} \leq 2 \norm{\varphi}_\infty \leq 2 \norm{\varphi}_\infty \cdot \vartheta^{-n} \secMetric{\secOnQuot,\secOnQuot^\prime}.
	\]
	Furthermore, by definition of \( \norm{\cdot}_\vartheta \) every \( \norm{\cdot}_\vartheta \)-Cauchy sequence in \( F_n \) is a \( \norm{\cdot}_\infty \)-Cauchy sequence.
	In \thref{lem:fmClosedInBounded} we have shown that \( F_n \subseteq \mathcal B(\secQuot, \Cnum) \) is a closed subspace.
	Therefore the \( \norm{\cdot}_\infty \)-limit lies in \( F_n \) and agrees with the \( \norm{\cdot}_\vartheta \)-limit in \( \lipschitzFunc \).
\end{proof}

\begin{rem}
	By \thref{rem:fct spaces} and \thref{lem:BHW5-2} we have the following inclusions between the function spaces
	\[
	F_0 \subseteq F_1 \subseteq \ldots \subseteq \lipschitzFunc \subseteq \lipschitzFunc[\vartheta^\prime] \subseteq \mathcal B(\secQuot, \Cnum) \subseteq \Maps(\secQuot, \Cnum)
	\]
	for all \( 0 < \vartheta \leq \vartheta^\prime < 1 \).
\end{rem}

\begin{prop}\label{prop:fnFiniteSpaces}
	If \( \chambSys \) is compact, the spaces \( F_n \) are finite-dimensional vector spaces.
\end{prop}

\begin{proof}
	The dimension of \( F_n \) is given by the number of different balls of size \( n \).
	From \thref{lem:finitelyBallsCover} we know that for \( n \in \Nnum_0 \) there exists only a finite number of different balls \( \ball_n(\secOnQuot) \).
\end{proof}

\subsection{Key-inequality and Approximation}
We now want to study in more detail how the transfer operators behave on these different function spaces.

\begin{prop}\label{prop:supNormEstQuot}
	For \( \paramA \in \domCoweights \) and \( \varphi \in \mathcal B(\secQuot, \Cnum) \) we have
	\[
	\norm{\tOp_\paramA \varphi}_\infty \leq \norm{\varphi}_\infty.
	\]
	That means, \( \tOp_\paramA \) leaves \( \mathcal B(\secQuot, \Cnum) \) invariant.
	Morever, the operator norm with respect to the supremum norm is bounded by \( 1 \), i.e., \( \norm{\tOp_\paramA}_{\mathrm{op}} \leq 1 \).
\end{prop}

\begin{proof}
	For \( \secOnQuot \in \secQuot \) we have
	\begin{align*}
		\abs{\tOp_\paramA \varphi(\secOnQuot)}
&= \abs*{\frac{1}{M_\paramA} \sum\limits_{\sigma_\paramA(\secOnQuot^\prime) = \secOnQuot} \varphi(\secOnQuot^\prime)} 
\leq \frac{1}{M_\paramA} \sum\limits_{\sigma_\paramA(\secOnQuot^\prime) = \secOnQuot} \abs{\varphi(\secOnQuot^\prime)} 
\leq \frac{1}{M_\paramA} \sum\limits_{\sigma_\paramA(\secOnQuot^\prime) = \secOnQuot} \norm{\varphi}_\infty 
= \norm{\varphi}_\infty,
	\end{align*}
which implies \( \norm{\tOp_\paramA \varphi}_\infty \leq \norm{\varphi}_\infty \).
\end{proof}

The central ingredient for our spectral theory will be the following key-inequality.
\begin{thm}\label{thm:keyInequalOnQuot}
	Let \( \paramA \in \innerCoweights \).
	There exists a constant \( C \geq 0 \), such that for every \( \varphi \in \lipschitzFunc \) we have
	\[
	\abs{\tOp_\paramA \varphi}_\vartheta \leq \vartheta \abs{\varphi}_\vartheta + C \norm{\varphi}_\infty.
	\]
\end{thm}

\begin{proof}
	The proof follows the steps from \cite[p.~33, Lemma 1.2]{baladi00}.
	
	For \( \secOnQuot_1, \secOnQuot_2 \in \secQuot \) with \( \secMetric{\secOnQuot_1, \secOnQuot_2} < \vartheta \) we get
	\begin{align*}
		\abs{\tOp_\paramA \varphi(\secOnQuot_1) - \tOp_\paramA \varphi(\secOnQuot_2)} 
		&= \frac{1}{M_\paramA} \abs*{\sum\limits_{\sigma_\paramA(\secOnQuot_1^\prime) = \secOnQuot_1} \varphi(\secOnQuot_1^\prime) - \sum\limits_{\sigma_\paramA(\secOnQuot_2^\prime) = \secOnQuot_2} \varphi(\secOnQuot_2^\prime)} \\
		\shortintertext{as \( \secMetric{\secOnQuot_1, \secOnQuot_2} < \vartheta \) we can apply \thref{prop:basePointProp} and find pairs of preimages \(\secOnQuot_1', \secOnQuot_2'\), with \( \mu \in \mathcal E_0(\secOnQuot_1',\secOnQuot_2') \)}
		&= \frac{1}{M_\paramA} \abs*{\sum\limits_{(\secOnQuot_1^\prime, \secOnQuot_2^\prime) \in P} \varphi(\secOnQuot_1^\prime) - \varphi(\secOnQuot_2^\prime)} \\
		&\leq \frac{1}{M_\paramA} \sum \abs*{\varphi(\secOnQuot_1^\prime) - \varphi(\secOnQuot_2^\prime)} \\
		&\leq \frac{1}{M_\paramA} \sum \abs{\varphi}_\vartheta \cdot \secMetric{\secOnQuot_1^\prime, \secOnQuot_2^\prime} \\
		\shortintertext{and by \thref{prop:shiftKeyInequality}}
		&\leq \frac{1}{M_\paramA} \sum \abs{\varphi}_\vartheta \cdot \vartheta \secMetric{\sigma_\paramA(\secOnQuot_1^\prime), \sigma_\paramA(\secOnQuot_2^\prime)} \\
		&= \frac{1}{M_\paramA} \sum \abs{\varphi}_\vartheta \cdot \vartheta d_\vartheta(\secOnQuot_1, \secOnQuot_2) \\
		&= \frac{1}{M_\paramA} M_\paramA \cdot \abs{\varphi}_\vartheta \cdot \vartheta d_\vartheta(\secOnQuot_1, \secOnQuot_2) \\
		&= \vartheta \abs{\varphi}_\vartheta d_\vartheta(\secOnQuot_1, \secOnQuot_2).
	\end{align*}
	If \( \secMetric{\secOnQuot_1, \secOnQuot_2} \geq \vartheta \), using \thref{prop:supNormEstQuot}, we find
	\[
	\frac{\abs{\tOp_\paramA \varphi(\secOnQuot_1) - \tOp_\paramA \varphi(\secOnQuot_2)}}{\secMetric{\secOnQuot_1, \secOnQuot_2}} \leq \vartheta^{-1} \abs{\tOp_\paramA \varphi(\secOnQuot_1) - \tOp_\paramA \varphi(\secOnQuot_2)} \leq 2\vartheta^{-1} \norm{\tOp_\paramA \varphi}_\infty \leq 2 \vartheta^{-1} \norm{\varphi}_\infty.
	\]
	Setting \( C \coloneqq 2 \vartheta^{-1} \) yields the desired inequality.
\end{proof}

\begin{cor}\label{cor:keyInequalNQuot}
	Let \( \paramA \in \innerCoweights \).
	There exists a constant \( C \geq 0 \), such that for every \( \varphi \in \lipschitzFunc \) and all \( \ell \geq 0 \) we have
	\[
	\abs{\tOp_\paramA^\ell \varphi}_\vartheta \leq \vartheta^\ell \abs{\varphi}_\vartheta + C \norm{\varphi}_\infty.
	\]
\end{cor}

\begin{proof}
	This follows from \thref{thm:keyInequalOnQuot} via induction as follows:
	\begin{align*}
		\abs{\tOp_\paramA^\ell \varphi}_\vartheta 
		&\leq \vartheta \abs{\tOp_{\paramA}^{\ell-1} \varphi}_\vartheta + \widetilde{C} \norm{\varphi}_\infty 
		\leq \vartheta^\ell \abs{\varphi}_\vartheta + \sum\limits_{j=0}^{\ell-1} \vartheta^j \widetilde{C} \norm{\varphi}_\infty \\
		&= \vartheta^\ell \abs{\varphi}_\vartheta + \widetilde{C} \frac{1 - \vartheta^\ell}{1-\vartheta} \norm{\varphi}_\infty 
		\leq \vartheta^\ell \abs{\varphi}_\vartheta + \frac{\widetilde{C}}{1-\vartheta} \norm{\varphi}_\infty.
	\end{align*}
	Now we set \( C = \frac{\widetilde{C}}{1 - \vartheta} \).
\end{proof}

\begin{cor}\label{cor:tOpEstimateQuot}
	For every \( \paramA \in \domCoweights \) there exists a constant \( C \geq 0 \), such that for all \( \varphi \in \lipschitzFunc \) we have
	\[
	\abs{\tOp_\paramA \varphi}_\vartheta \leq \abs{\varphi}_\vartheta + C \norm{\varphi}_\infty.
	\]
\end{cor}

\begin{proof}
	Replacing \thref{prop:shiftKeyInequality} by \thref{cor:expandingShift} we can use the same proof as for the key-inequality \thref{thm:keyInequalOnQuot}.
\end{proof}

\begin{cor}
	Let \( \paramA \in \domCoweights \).
	The transfer operator \( \tOp_\paramA \) leaves \( \lipschitzFunc \) invariant and is a bounded linear operator with respect to the norm \( \norm{\cdot}_\vartheta \) for every \( \vartheta \in {}]0,1[{} \).
\end{cor}

\begin{proof}
	Let \( \varphi \in \lipschitzFunc \).
	By \thref{prop:supNormEstQuot} and \thref{cor:tOpEstimateQuot} we find a constant \( C>0 \) such that
	\begin{align*}
		\norm{\tOp_\paramA \varphi}_\vartheta &= 
		\abs{\tOp_\paramA \varphi}_\vartheta + \norm{\tOp_\paramA \varphi}_\infty \leq \abs{\tOp_\paramA \varphi}_\vartheta + \norm{\varphi}_\infty 
\leq \abs{\varphi}_\vartheta + (1+C) \norm{\varphi}_\infty \\
		&\leq (1+C) (\abs{\varphi}_\vartheta + \norm{\varphi}_\infty) 
		= (1+C) \norm{\varphi}_\vartheta. \qedhere
	\end{align*}
\end{proof}

\begin{cor}\label{cor:keyInqualNormQuot}
	Let \( \paramA \in \innerCoweights \).
	There exists a constant \( C \geq 0 \), such that for every \( \varphi \in \lipschitzFunc \) and all \( \ell \geq 0 \) we have
	\[
	\norm{\tOp^\ell_\paramA}_\vartheta \leq \vartheta^\ell \norm{\varphi}_\vartheta + C \norm{\varphi}_\infty.
	\]
\end{cor}

\begin{proof}
	This is a direct consequence of the key-inequality (\thref{cor:keyInequalNQuot}) and \thref{prop:supNormEstQuot}.
	\begin{align*}
		\norm{\tOp_\paramA^\ell \varphi}_\vartheta &= \abs{\tOp^\ell_\paramA \varphi}_\vartheta + \norm{\tOp^\ell_\paramA \varphi}_\infty 
		\leq \vartheta^\ell \abs{\varphi}_\vartheta + (\widetilde{C}+1) \norm{\varphi}_\infty \\
		&\leq \vartheta^\ell (\abs{\varphi}_\vartheta + \norm{\varphi}_\infty) + C \norm{\varphi}_\infty 
		= \vartheta^\ell \norm{\varphi}_\vartheta + C \norm{\varphi}_\infty. \qedhere
	\end{align*}
\end{proof}

\subsubsection{Connection to Function Spaces}

\begin{prop}\label{prop:fnInvariantQuot}
	Let \( \paramA \in \domCoweights \).
	\begin{enumerate}[(i)]
		\item\label{item:fnTOpinv}
			The transfer operator \( \tOp_\paramA \) leaves \( F_n \) invariant.
		\item\label{item:innerContractsTOp}
			If \( \paramA \in \innerCoweights \) is a strongly dominant coweight we have \( \tOp_\paramA \colon F_n \to F_{n-1} \) for all \( n > 1 \).
	\end{enumerate}
\end{prop}

\begin{proof}
	Let \( \paramA \in \domCoweights \).
	\begin{enumerate}[(i)]
		\item
			For \( n=0 \) this is \eqref{eq:one_in_spectrum}.
			
			For \( n \geq 1 \) we use \thref{prop:basePointProp} as every two admissible sectors in a ball of size at least 1 satisfy the conditions of that proposition.
			
			Here are the details.
			Let \( \secOnQuot \in \secQuot \) and \( \secOnQuot^\prime \in \ball_n(\secOnQuot) \).
			Note that \( \secQuot \) is an ultrametric space and hence any point is the center of a ball.
			This allows us to take the pairs described in \thref{prop:basePointProp}.
			We denote them by \( (t,t^\prime) \) with \( t \in \sigma_\paramA^{-1}(\secOnQuot) \) and \( t^\prime \in \sigma_\paramA^{-1}(\secOnQuot^\prime) \).
			
			For these pairs we can apply \thref{cor:expandingShift} and find
			\[
			\secMetric{t,t^\prime} \leq \secMetric{\sigma_\mu(t), \sigma_\mu(t^\prime)} = \secMetric{\secOnQuot, \secOnQuot^\prime} \leq \vartheta^n.
			\]
			Now we can compute
			\begin{align*}
				M_\paramA \left(\tOp_\paramA f(\secOnQuot) - \tOp_\paramA f(\secOnQuot^\prime)\right) &= \sum\limits_{\sigma_\paramA(t) = \secOnQuot} f(t) - \sum\limits_{\sigma_\paramA(t^\prime) = \secOnQuot^\prime} f(t^\prime) 
				= \sum\limits_{(t,t^\prime)} f(t) - f(t^\prime) 
				= 0.
			\end{align*}
		\item
			To show that \( \tOp(F_n) \subseteq F_{n-1} \) we apply \thref{prop:shiftKeyInequality}.
			This allows us to choose \( \secOnQuot^\prime \in \ball_{n-1}(\secOnQuot) \) and with the same argument as above find
			\[
			\secMetric{t,t^\prime} \leq \vartheta \secMetric{\secOnQuot, \secOnQuot^\prime} \leq \vartheta^n.
			\]
			Hence \( f(t) - f(t^\prime) = 0 \) as before. \qedhere
	\end{enumerate}
\end{proof}

\subsubsection{Approximation by Projection}

\begin{defi}[Projection operator]
	Given \( n \in \Nnum_0 \) we choose representatives \( \secOnQuot_\iota \in \secQuot \) such that \( \secQuot = \bigsqcup_{\iota \in I} \ball_n(\secOnQuot_\iota) \).
	This allows us to define the projection operator
	\[
	\Pi_n \colon \lipschitzFunc \to F_n,
	\]
	by the requirement that \( \Pi_n \varphi \) is constant on every set \( \ball_n(\secOnQuot) \) and \( \varphi(\secOnQuot_\iota) = \Pi_n \varphi(\secOnQuot_\iota) \).
\end{defi}

This means we can evaluate \( \Pi_n \varphi(\secOnQuot) \) by choosing the unique sector \( \secOnQuot^\prime \in \Set{\secOnQuot_\iota}_{\iota \in I} \) with \( \secOnQuot_\iota \in \ball_n(\secOnQuot) \) and computing \( \Pi_n \varphi(\secOnQuot) = \varphi(\secOnQuot_\iota) \).
This gives us the following estimates.

\begin{lem}\label{lem:InequalProjOnQuot}
	For \( \varphi \in \lipschitzFunc \) we have
	\[
	\norm{\varphi - \Pi_n \varphi}_\infty \leq \vartheta^n \abs{\varphi}_\vartheta.
	\]
	On the other hand,
	\[
	\abs{\Pi_n \varphi}_\vartheta \leq \abs{\varphi}_\vartheta.
	\]
\end{lem}

\begin{proof}
	For the first inequality we take \( \secOnQuot \in \secQuot \) and a representative \( \secOnQuot^\prime \in \secQuot \) such that \( \Pi_n \varphi(\secOnQuot) = \varphi(\secOnQuot^\prime) \).
	By definition these sectors suffice \( \secMetric{\secOnQuot, \secOnQuot^\prime} \leq \vartheta^n \).
	By the definition of the optimal Lipschitz constant \( \abs{\varphi}_\vartheta \) we obtain
	\begin{align*}
		\norm{(\varphi - \Pi_n \varphi)(\secOnQuot)}_\infty 
		&= \norm{\varphi(\secOnQuot) - \varphi(\secOnQuot^\prime)}_\infty
		\leq \secMetric{\secOnQuot, \secOnQuot^\prime} \abs{\varphi}_\vartheta 
		\leq \vartheta^n \abs{\varphi}_\vartheta.
	\end{align*}
	For the second inequality we take \( \secOnQuot_1, \secOnQuot_2 \in \secQuot \) and its corresponding representatives \( \secOnQuot_1^\prime, \secOnQuot_2^\prime \).
	We obtain
	\begin{equation}\label{equ:projEstimate}
		\abs{\Pi_n \varphi(\secOnQuot_1) - \Pi_n \varphi(\secOnQuot_2)} = \abs{\varphi(\secOnQuot_1^\prime) - \varphi(\secOnQuot_2^\prime)} \leq \abs{\varphi}_\vartheta
		\secMetric{\secOnQuot_1^\prime, \secOnQuot_2^\prime}.
	\end{equation}
	To get the inequality we have to distinguish two cases:
	\begin{proofcase}
		\item[1. Case]
			\( \secMetric{\secOnQuot_1, \secOnQuot_2} \leq \vartheta^n \).
			In this case we have \( \ball_n(\secOnQuot_1) = \ball_n(\secOnQuot_2) \) and therefore \( \secOnQuot_1^\prime = \secOnQuot_2^\prime \).
			Hence \( \Pi_n \varphi(\secOnQuot_1) = \Pi_n \varphi(\secOnQuot_2) \) and the left hand side of the equation is zero.
			Therefore we get
			\[
			0 = \abs{\Pi_n \varphi(\secOnQuot_1) - \Pi_n \varphi(\secOnQuot_2)} \leq \abs{\varphi}_\vartheta \secMetric{\secOnQuot_1, \secOnQuot_2}.
			\]
		\item[2. Case]
			If \( \secMetric{\secOnQuot_1, \secOnQuot_2} > \vartheta^n \) we have \( \ball_n(\secOnQuot_1) \neq \ball_n(\secOnQuot_2) \) and by \thref{rem:propOfUltrametric}\ref{item:diffBallsSameDist} we have \( \secMetric{\secOnQuot_1^\prime, \secOnQuot_2^\prime} = \secMetric{\secOnQuot_1, \secOnQuot_2} \) and in \Cref{equ:projEstimate} we get
			\[
			\abs{\Pi_n \varphi(\secOnQuot_1) - \Pi_n \varphi(\secOnQuot_2)} \leq \abs{\varphi}_\vartheta \cdot \secMetric{\secOnQuot_1^\prime, \secOnQuot_2^\prime} = \abs{\varphi}_\vartheta \secMetric{\secOnQuot_1, \secOnQuot_2}.
			\]
	\end{proofcase}
	In both cases we find that the optimal Lipschitz constant satisfies \( \abs{\Pi_n \varphi}_\vartheta \leq \abs{\varphi}_\vartheta \).
\end{proof}

\begin{cor}\label{cor:approxByProjOnQuot}
	For \( \paramA \in \innerCoweights \), \( m \in \Nnum \) and \( \varphi \in \lipschitzFunc \) we have
	\begin{enumerate}[(i)]
		\item\label{item:estimOneTOp}
			\( \norm*{\tOp_\paramA \varphi - \tOp_\paramA \Pi_m \varphi}_\infty \leq \vartheta^m \abs{\varphi}_\vartheta \leq \vartheta^m \norm{\varphi}_\vartheta \),
		\item\label{item:estimTwoTOp}
			\( \norm*{\tOp_\paramA^m \varphi - \tOp_\paramA^m \Pi_m \varphi}_\infty \leq \vartheta^m \abs{\varphi}_\vartheta \leq \vartheta^m \norm{\varphi}_\vartheta \),
	\end{enumerate}
\end{cor}

\begin{proof}
	\mbox{}
	\begin{enumerate}[(i)]
		\item
			With \thref{prop:supNormEstQuot} and \thref{lem:InequalProjOnQuot} we get
			\begin{align*}
				\norm*{\tOp_\paramA \varphi - \tOp_\paramA \Pi_m \varphi}_\infty 
				&= \norm{\tOp_\paramA(\varphi - \Pi_m \varphi)}_\infty 
				\leq \norm{\varphi - \Pi_m \varphi}_\infty 
				\leq \vartheta^m \abs{\varphi}_\vartheta 
				\leq \vartheta^m \norm{\varphi}_\vartheta.
			\end{align*}
		\item
			Same argument as above but using \thref{prop:supNormEstQuot} repeatedly. \qedhere
	\end{enumerate}
\end{proof}

\subsection{Spectral Theory for Transfer Operators}
Let us recall some standard terminology.

\begin{defi}[Fredholm Operator]\label{defi:fredholmIndexZero}
	A bounded operator \( T \) on a Banach space \( \banachSpace \) is called \emph{Fredholm of index zero} if \( \Image T \subset \banachSpace \) is closed and additionally \( \dim \Ker T = \dim\Coker T<\infty \).
\end{defi}
While the \emph{spectrum} of \( T \) is defined as
\[
\sigma(T) \coloneqq \Set{\lambda \in \Cnum \given \lambda \id_{\banachSpace} -T \text{ is not invertible}},
\]
we define the \emph{essential spectrum} of \( T \) as 
\begin{equation}\label{equ:essentialSpec}
	\sigma_{\mathrm{ess}}(T) \coloneqq \Set{\lambda \in \Cnum \given \lambda \id_{\banachSpace} -T \text{ is not Fredholm of index zero}}.
\end{equation}
The \emph{spectral radius} and \emph{essential spectral radius} of \( T \) are then defined as
\[
\specRad(T) \coloneqq \inf\Set{r>0 \given \sigma(T) \subset B_r(0)} \quad \text{and} \quad \essSpecRad(T) \coloneqq \inf\Set{r>0 \given \sigma_{\mathrm{ess}}(T) \subset B_r(0)}.
\]
\begin{defi}[Quasi-compact Operator]
	We call a bounded operator $T$ on a Banach space \emph{quasi-compact} if \( \essSpecRad(T) < r(T) \).
\end{defi}

Note that by the analytic Fredholm theorem (see \cite[Thm.~D.4]{Zworski2012}), all \( \lambda\in \sigma(T) \setminus B_{\essSpecRad(T)}(0) \) are isolated eigenvalues of finite multiplicity.

We will use the following criterion for quasi-compactness.
\begin{prop}[{{\cite[Corollaire 1]{hennion93}}}]\label{prop:hennionResult}
	Let \( T \) be a bounded operator on a Banach space \( (\banachSpace, \norm{\cdot}) \).
	If there exists another norm \( \abs{\cdot} \) on \( \banachSpace \) such that \( T \) has the following two properties
	\begin{enumerate}[(i)]
		\item
			\( T \) is a compact operator from \( (\banachSpace, \norm{\cdot}) \) to \( (\banachSpace, \abs{\cdot}) \)
		\item
			There are sequences \( \Set{r_n} \) and \( \Set{\rho_n} \) of positive numbers, with
			\[
			r = \liminf\limits_{n \to \infty} \sqrt[n]{r_n} < r(T),
			\]
			and 
			\[
			\forall n \geq 1:\quad \norm{T^n \varphi} \leq r_n \norm{\varphi} + \rho_n \abs{\varphi}.
			\]
	\end{enumerate}
	Then the essential spectral radius of \( T \) on \((\mathcal B, \norm{\cdot}) \) satisfies \( \essSpecRad(T) \leq r < r(T) \).
	In particular, \( T \) is quasi-compact.
\end{prop}

\begin{thm}\label{thm:transferOpInteriorQuasiComp}
	For \( \paramA \in \innerCoweights \) the transfer operator \( \tOp_\paramA \colon \lipschitzFunc \to \lipschitzFunc \) is a quasi-compact operator with essential spectral radius \( \essSpecRad(\tOp_\paramA)\leq \vartheta \).
\end{thm}

\begin{proof}
	We want to apply \thref{prop:hennionResult}.
	First we check that
	\[
	\tOp_\paramA \colon (\lipschitzFunc, \norm{\cdot}_\vartheta) \to (\lipschitzFunc, \norm{\cdot}_\infty)
	\]
	is a compact operator.
	To do so we show that we can approximate \( \tOp_\paramA \) by a sequence of operators of finite rank.
	We define
	\[
	A_\paramA^n \coloneqq \tOp_\paramA \Pi_n \colon \lipschitzFunc \to F_n.
	\]
	Here we use \thref{prop:fnInvariantQuot}\ref{item:fnTOpinv} which says that \( \tOp_\paramA \) leaves \( F_n \)-invariant.
	
	Since \( F_n \) is a finite-dimensional vector space (\thref{prop:fnFiniteSpaces}) the operator \( A_\paramA^n \) has finite rank.
	As we have seen in \thref{cor:approxByProjOnQuot}\ref{item:estimOneTOp}
	\[
	\norm{\tOp_\paramA \varphi - A_\paramA^n \varphi}_\infty \leq \vartheta^n \norm{\varphi}_\vartheta \quad \text{for all } \varphi \in \lipschitzFunc.
	\]
	Thus we have \( \lim_{n \to \infty} A_\paramA^n = \tOp_\paramA \) and \( \tOp_\paramA \) is a compact operator.
	
	Next we need to find sequences \( \Set{r_n} \) and \( \Set{\rho_n} \) of positive numbers, with
	\[
	r = \liminf\limits_{n \to \infty} (r_n)^{\sfrac{1}{n}} < r(\tOp_\paramA)
	\]
	and for \( n \geq 1 \) an estimate
	\[
	\norm{\tOp_\paramA^n \varphi}_\vartheta \leq r_n \norm{\varphi}_\vartheta + \rho_n \norm{\varphi}_\infty.
	\]
	We can choose \( r_n = \vartheta^n \) and \( \rho_n = C \) (a constant sequence).
	In fact, by \thref{eq:one_in_spectrum} we have \( 1 \in \sigma(\tOp_\paramA)\) so that \( r = \vartheta < 1 \leq r(\tOp_\paramA) \).
	The second inequality is simply the key-inequality from \thref{cor:keyInqualNormQuot}.
	Now \thref{prop:hennionResult} implies \( \essSpecRad(\tOp_\paramA) \leq r < r(\tOp_\paramA) \).
	Therefore \( \tOp_\paramA \) is quasi-compact.
\end{proof}

The following proposition is closely related to \cite[Proposition 7.3]{bux2023spectral}.

\begin{prop}\label{prop:GenEigenspaceInFOne}
	Let \( \vartheta \in {}]0,1[{} \), \( \paramA \in P_{++}^\vee \) and \( z \in \Cnum \) with \( \abs{z} > \vartheta \).
	Then the generalized eigenspace \( \Hau(\tOp_\paramA, z) \coloneqq \Set{\varphi \in \lipschitzFunc \given \exists n\in \Nnum: \ (\tOp_\paramA-z\id_{\lipschitzFunc})^n\varphi} \) on \( \lipschitzFunc \) is a subspace of \( F_1 \).
\end{prop}

\begin{proof}
	Choose \( \varepsilon > 0 \) such that \( \abs{z} > \vartheta + \varepsilon \).
	First we show that the eigenspace \( \Eig(\tOp_\paramA, z) \) of \(\tOp_\paramA\) for the eigenvalue $z$ is contained in $F_1$.
	Then we proceed by induction.
	\begin{itemize}
		\item
			Consider \( \varphi \in \Eig(\tOp_\paramA, z) \) and set \( \varphi_n \coloneqq \Pi_n \varphi \) and \( r_n \coloneqq \varphi - \varphi_n \).
			As \( \tOp_\paramA \varphi = z \varphi \), we find
			\begin{align*}
				\varphi = z^{-n} \tOp_\paramA^n \varphi = z^{-n} \tOp_\paramA^n \varphi_n + z^{-n} \tOp_\paramA^n r_n.
			\end{align*}
			Note that by \thref{prop:fnInvariantQuot}\ref{item:innerContractsTOp} \( z^{-n} \tOp_\paramA^n \varphi_n \in F_1 \).
			From \thref{cor:approxByProjOnQuot}\ref{item:estimTwoTOp} we know that
			\[
			\norm{z^{-n} \tOp_\paramA^n r_n}_\infty \leq \abs{z^{-n}} \norm{\tOp_\paramA^n r_n}_\infty \leq \frac{\vartheta^n \abs{\varphi}_\vartheta}{(\vartheta + \varepsilon)^n} = \left( \frac{\vartheta}{\vartheta + \varepsilon}\right)^n \abs{\varphi}_\vartheta.
			\]
			For \( n \to \infty \) we see that \( \varphi \) is the \( \norm{\cdot}_\infty \)-limit of the sequence \( (z^{-n} \tOp_\paramA^n \varphi_n)_n \) which stays in the closed subspace \( F_1 \).
			Hence \( \varphi \in F_1 \).
		\item
			We can use the inclusion \( \Eig(\tOp_\paramA, z) \subseteq F_1 \) as the starting point of an induction.
			
			Let \( \varphi \in \Ker(\tOp_\paramA -z)^n \), then there exists some \( \widetilde{\varphi} \in \Ker(\tOp_\paramA - z)^{n-1} \) such that
			\[
			\tOp_\paramA \varphi = \widetilde{\varphi} + z \varphi.
			\]

			Iterating this equation and using that \( \Ker(\tOp_\paramA -z)^{n-1} \) is \( \tOp_\paramA \)-invariant, we obtain elements \( \widetilde{\varphi}_m \in \Ker(\tOp_\paramA - z)^{n-1} \) such that
			\[
			\tOp_\paramA^m \varphi = \widetilde{\varphi}_m + z^m \varphi.
			\]
			By induction hypothesis \( \Ker(\tOp_\paramA - z)^{n-1} \subseteq F_1 \).
			Now we decompose \( \varphi \) as in the case for the eigenspace.
			We set \( \varphi_m \coloneqq \Pi_m \varphi \), \( r_m \coloneqq \varphi - \varphi_m \) and \( \varphi^\prime_m = - z^{-m} \widetilde{\varphi}_m \).
			Now we can write
			\[
			\varphi = \varphi^\prime_m + z^{-m} \tOp_\paramA^m \varphi = \varphi^\prime_m + z^{-m} \tOp_\paramA^m \varphi_m + z^{-m} \tOp_\paramA^m r_m.
			\]
			With the same argument as above we find \( z^{-m} \tOp_\paramA^m r_m \to 0 \) and therefore
			\[
			\varphi^\prime_m + z^{-m} \tOp_\paramA^m \varphi_m \to \varphi
			\]
			in \( F_1 \).
			Since \( F_1 \) is a finite-dimensional closed subspace we have \( \varphi \in F_1 \).
	\end{itemize}
\end{proof}

As one can see in the proof it is important that we have \( \tOp_\paramA(F_n) \subseteq F_{n-1} \) (we get this via \thref{prop:fnInvariantQuot}\ref{item:innerContractsTOp}).
Here we need the assumption \( \paramA \in P_{++}^\vee \).

\begin{cor}\label{cor:eigenvalIndOfTheta}
	Let \( \paramA \in P_{++}^\vee \) and \( \vartheta, \vartheta^\prime \in {}]0,1[{} \) with \( \vartheta < \vartheta^\prime \).
	If \( z \in \Cnum \) is an eigenvalue for \( \tOp_\paramA \colon \mathcal F_{\vartheta^\prime} \to \mathcal F_{\vartheta^\prime} \) with \( \abs{z} > \vartheta^\prime \) then \( z \in \Cnum \) is an eigenvalue for \( \tOp_\paramA \colon \mathcal F_{\vartheta} \to \mathcal F_{\vartheta} \) as well.
\end{cor}

\begin{lem}
	Let \( \paramA, \paramB \in P_+^\vee \).
	The operators \( \tOp_{\paramA} \colon \lipschitzFunc \to \lipschitzFunc \) and \( \tOp_{\paramB} \colon \lipschitzFunc \to \lipschitzFunc \) have a common eigenvector in \( F_1 \).
	Hence the eigenvalue spectrum is non-empty, i.e., \( \sigma_{\mathrm{p}}(\tOp_\paramB) \neq \emptyset \).
\end{lem}

\begin{proof}
	From \thref{prop:fnInvariantQuot}\ref{item:fnTOpinv} we know that \( F_1 \) is an invariant subspace of \( \tOp_\paramA \) and \( \tOp_\paramB \).
	Therefore, \( \restr{\tOp_\paramA}{F_1} \) and \( \restr{\tOp_\paramB}{F_1} \) are two commuting operators on a finite dimensional complex vector space.
	Hence they have a common eigenvector.
\end{proof}

\subsection{Spectral Theory of Quasi-Compact Operators}\label{subsec:qu-cpt op}

In this subsection we recall some facts from operator theory that will be used in our analysis of joint spectra.
We start with the construction of Riesz projections as explained in \cite[Section 6.1]{hislopSigal96}.

Let \( \operator \in \mathcal L(\banachSpace) \) be a bounded operator on the Banach space \( \banachSpace \) and \( z_0 \in \sigma(\operator) \) an isolated point of the spectrum, and let \( \Gamma_{z_0} \) be a simple closed contour around \( z_0 \) such that the closure of the region bounded by \( \Gamma_{z_0} \) and containing \( z_0 \) intersects \( \sigma(\operator ) \) only at \( z_0 \).
We refer to such a contour as \emph{admissible} for \( z_0 \) and \( \operator \).
Then the contour integral
\[
P_{z_0} \coloneqq (2 \pi i)^{-1} \oint\limits_{\Gamma_{z_0}} R_\operator (z) \, \diff z
\]
where \( R_\operator \) is the resolvent of \( \operator \), is called the \emph{Riesz integral} for \( \operator \) and \( z_0 \).
One can show, that due to the analyticity of \( R_\operator (z) \) in \( z \) the integral exists and defines a bounded linear transformation on \( \banachSpace \).

\begin{lem}
	\( P_{z_0} \) is independent of the contour \( \Gamma_{z_0} \) provided that the contour is admissible for \( z_0 \) and \( \operator \), that is, it lies in the resolvent set \( \rho(\operator ) \) and contains no other part of \( \sigma(\operator ) \) besides \( z_0 \).
\end{lem}

\begin{prop}\label{prop:rieszProjEigspaceInImage-Text}
	Let \( P_{z_0} \) be the Riesz integral for \( \operator \) and \( z_0 \).
	\begin{enumerate}[(i)]
		\item
			\( P_{z_0} \) is a projection.
		\item
			\( \Image P_{z_0} \supset \Ker(\operator - z_0) \).
	\end{enumerate}
\end{prop}

From now on we assume that \( \operator \in \mathcal L(\banachSpace) \) is quasi-compact. We are interested in its \emph{discrete spectrum}
\[
\sigma_{\mathrm{disc}}(\operator) \coloneqq \Set{\lambda \in \sigma(\operator) \given \lambda \text{ is an isolated eigenvalue of finite multiplicity}}.
\]
We call the elements of \( \sigma_{\mathrm{disc}}(\operator) \) \emph{normal eigenvalues}.

For the following we refer to \cite[Chap.~XIV]{hennionHerve01}.
Given \( \varepsilon > 0 \), every \( \lambda \in \sigma(\operator) \) with \( \abs{\lambda} \geq \essSpecRad(\operator) + \varepsilon \) is a normal eigenvalue.
Fix a number \( \essSpecRad(\operator) < r < \infty \). Then, the annulus 
\[
\annulus{r}{\essSpecRad(\operator) + \varepsilon}(0) \coloneqq \Set{z \in \Cnum \given \essSpecRad(\operator) + \varepsilon \leq \abs{z} \leq r}
\] 
is compact and hence \( \sigma(\operator) \cap \annulus{r}{\essSpecRad(\operator) + \varepsilon}(0) \) has only finitely many elements.

For a given \( \varepsilon > 0 \) we can decompose \( \banachSpace \) as
\[
\banachSpace = F(\essSpecRad(\operator) +\varepsilon) + R(\essSpecRad(\operator) + \varepsilon).
\]
The construction of this decomposition is based on Riesz projectors.
We know that every \( \lambda \in \sigma(\operator) \) with \( \lambda > \essSpecRad(\operator) \) is automatically a normal eigenvalue, i.e., \( \lambda \in \sigma_{\mathrm{disc}}(\operator) \).
For every normal eigenvalue \( z_1,\ldots,z_n \in \annulus{r}{\essSpecRad(\operator)+\varepsilon}(0) \) we have a Riesz projector \(P_{z_i}\).
It satisfies 
\begin{align*}
	\Eig(\operator, z_i) &\coloneqq \Set*{x \in \banachSpace \given \operator x=z_i x} \subseteq \Image P_{z_i},\\
	\Hau(\operator, z_i) &\coloneqq \Set*{x \in \banachSpace \given \exists n \in\Nnum:\ (\operator -z_i \id_{\banachSpace})^nx= 0} = \Image P_{z_i}. 
\end{align*} 
Setting \( \sigma_1 \coloneqq \Set{z_1,\ldots,z_n} = \sigma_{\mathrm{disc}}(\operator) \cap \annulus{r}{\essSpecRad(\operator)+\varepsilon}(0) \) we have
\[
P_{\sigma_1} \coloneqq \oint\limits_{\Gamma} R_{\operator}(z) \; \diff z = \sum\limits_{i=1}^n P_{z_i},
\]
for some suitable contour \( \Gamma \) around \( \sigma_1 \).
This gives us the decomposition with
\[
F(\essSpecRad(\operator)+\varepsilon) \coloneqq \Image P_{\sigma_1} \quad \text{and} \quad R(\essSpecRad(\operator) + \varepsilon) \coloneqq \Ker P_{\sigma_1}.
\]
Moreover, we have
\begin{enumerate}[(i)]
	\item
		\( \dim F(\essSpecRad(\operator)+\varepsilon) < \infty \) and
	\item
		let \( z_1,\ldots,z_n \) be the normal eigenvalues in \( \annulus{r}{\essSpecRad(\operator)+\varepsilon}(0) \).
		Then
		\[
		\bigoplus_{i=1}^n \Eig(\operator, z_i) \subseteq F(\essSpecRad(\operator) + \varepsilon) = \bigoplus\limits_{i=1}^n \Hau(\operator, z_i).
		\]
\end{enumerate}

In the following remark we continue to assume that \( \banachSpace \) is a complex Banach space and \( \operator \colon \banachSpace \to \banachSpace \) is a quasi-compact operator.

\begin{rem}\label{rem:propOfFredholmAndQuasiCompOp}
	Let \( \finiteSpace \) be a finite dimensional \( \Cnum \)-vector space and \( \lambda \in \sigma(\operator) \) with \( \abs{\lambda} > \essSpecRad(\operator) \) so that \(\operator - \lambda \id_{\banachSpace}\) is Fredholm of index zero.
	Note that \( \id_{\finiteSpace} \otimes \operator\) is quasi-compact as it consists of finitely many blocks identical to \(\operator\). 
	An operator \( F \) of the form
	\[
	F \coloneqq \id_{\finiteSpace} \otimes (\operator - \lambda \id_{\banachSpace}) \colon \finiteSpace \otimes \banachSpace \to \finiteSpace \otimes \banachSpace
	\]
	has the following properties:
	\begin{enumerate}[(i)]
		\item\label{item:tensorProperty}
			\( F \) is Fredholm of index \( 0 \).
		\item\label{item:zeroNormalEV}
			\( 0 \) is a normal eigenvalue of \( F \) since \( \lambda \) is a normal eigenvalue of \( \operator \) and \( \id_{\finiteSpace} \) has only the eigenvalue \( 1 \) with finite algebraic multiplicity.
		\item\label{item:rieszProjFiniteIm}
			The Riesz projector \( \widetilde{P}_0 \) of \( F \) for the eigenvalue \( 0 \) has finite dimensional image.
		\item\label{item:rieszProjTensor}
			The Riesz projector \( \widetilde{P}_0 \) is of the form \( \widetilde{P}_0 = \id_{\finiteSpace} \otimes P_0 \), where \( P_0 \) as before is the Riesz projector of \( \operator - \lambda \id_{\banachSpace} \) for the eigenvalue \( 0 \).
		\item\label{item:splittingAndIsomOfGenEigen}
			Let \( j \in \Nnum_0 \) be the smallest number such that \( \Ker F^j = \Ker F^{j+n} \) for all \( n \in \Nnum_0 \).
			Then \( \Ker F^j = \Image \widetilde{P}_0 = \finiteSpace \otimes \Image P_0 \) and we have
			\[
			\finiteSpace \otimes \banachSpace = \Ker F^j \oplus \Image F^j = (\finiteSpace \otimes \Ker(\operator - \lambda \id_{\banachSpace})^j) \oplus (\finiteSpace \otimes \Image(\operator - \lambda \id_{\banachSpace})^j)
			\]
			where all spaces are \( F^j \)-invariant.
			Moreover, \( F \colon \Image F^j \to \Image F^j \) is an isomorphism.
		\item\label{item:rieszProjOnZeroDoesNotChange}
			The Riesz projector of \( F \) on the eigenvalue \( 0 \) is the same as the Riesz projector for \( F^k \) on the eigenvalue \( 0 \) for all \( k \in \Nnum \).
	\end{enumerate}
\end{rem}

\section{Joint Spectra}\label{sec:SpecTheory}

\subsection{Algebraic Preliminaries}\label{subsec:alg-prelim}

Let \( \complVs \) be a complex vector space and \( \operator_1,\ldots,\operator_n \in \End(\complVs)\) be a family of commuting linear endomorphisms of \( \complVs \).
We denote by 
\[ \mathcal A \coloneqq \spn{\operator_1,\ldots,\operator_n}_{\Cnum\text{-alg}} \subseteq \End(\complVs)\] 
the unital \( \Cnum \)-algebra generated by the \( \operator_1,\ldots,\operator_n \).
Then we have an injective map \( \Cnum \to \mathcal A \), identifying the unit element \( 1_\Cnum \) of \( \Cnum \) with the unit \( 1_{\mathcal A} \),
and a surjective unital \( \Cnum \)-algebra homomorphism
\[
\nPolyAlg \twoheadrightarrow \mathcal A,\qquad X_i\to \operator_i.
\]
We write \( \polyAlg \coloneqq \nPolyAlg \) whenever the number $n$ of indeterminates is clear from the context.

We want to introduce appropriate joint spectra of the operators in \( \mathcal A \) using techniques from homological algebra, for which we use \cite{weibel95} as a general reference. 

As a first step in the construction we take a look at the characters \( \chi \in \Hom_{\Cnum\text{-}\mathrm{alg.}}(\polyAlg, \Cnum) \). Each character defines a \( \polyAlg \)-bimodule structure on \( \Cnum \) given by
\[
\polyAlg \times \Cnum \times \polyAlg \to \Cnum, \; (P,z,Q) \coloneqq P\cdot z\cdot Q \coloneqq \chi(P)z\chi(Q).
\]
Since \( \Cnum \) is commutative and \( \Ker \chi \coloneqq \chi^{-1}(0)\) is a two-sided ideal in \( \polyAlg \) annihilating \( \Cnum \), this \( \polyAlg \)-bimodule structure on \( \Cnum \) factors to a \( \polyAlg/\Ker\chi \)-bimodule structure. If we identify the algebra \( \polyAlg/\Ker\chi \) with \( \Cnum \) via the canonical isomorphism $P+\Ker\chi\mapsto \chi(P)$ the \( \polyAlg/\Ker\chi \)-bimodule structure on \( \Cnum \) reduces to the \( \Cnum \)-bimodule structure of the algebra \( \Cnum \).

\begin{defi}\label{def:Cchi}
	For a given character \( \chi \in \polyAlg \) we denote \( \Cnum \), when equipped with the \( \polyAlg\)-bimodule structure, by \( \Cnum_\chi \).
\end{defi}

Note that by Hilbert's Nullstellensatz every character of \( \polyAlg \) is of the form 
\( \chi_a(P) \coloneqq P(a) \) for some \( a= (a_1,\ldots,a_n) \in \Cnum^n \).
We have 
\[
\Ker \chi_a = (X_1-a_1,\ldots,X_n - a_n),
\]
the ideal in \( \polyAlg \) generated by the monomials \( X_i-a_i \).

The second step in the construction of joint spectra of the operators in \( \mathcal A \) is to use \( \mathcal A \subseteq \End(\complVs) \) and the surjective algebra homomorphism \( \nPolyAlg \twoheadrightarrow \mathcal A \) to introduce a \( \polyAlg\)-bimodule structure also on \( \complVs \).
Note that, as \( \polyAlg\) is the free abelian algebra with generators \( X_1 \ldots,X_n \), the combined information \( \mathcal A \subseteq \End(\complVs) \) and \( \nPolyAlg \twoheadrightarrow \mathcal A \) is equivalent to the information of a set of generators \( \operatorFam = (\operator_1,\ldots,\operator_n) \) for $\mathcal A$.

\begin{defi}\label{def:VT}
	We define a \( \polyAlg \)-bimodule structure on \( \complVs \) via 
	\[
	\polyAlg \times \complVs \times \polyAlg \to \complVs, \; \genVar_i.x.\genVar_j \coloneqq \operator_j\operator_i x.
	\]
	We denote \( \complVs \), when equipped with this \( \polyAlg \)-bimodule structure, by \( \complVs_{\operatorFam} \).
\end{defi}

We now define the joint eigenvalue spectrum of the commuting operators $\operator_1,\ldots,\operator_n$:

\begin{defi}[Joint Eigenvalues]\label{defi:jointEV}
	Let \( \operatorFam = (\operator_1,\ldots,\operator_n) \) and \( \complVs_T \) be as in \thref{def:VT}.
	\begin{enumerate}[(a)]
		\item
			We say \( \lambda \in \Cnum^n \) is a \emph{joint eigenvalue} of \( \operatorFam \) if
			\[
			\Eig(\operatorFam, \lambda) \coloneqq \bigcap\limits_{i=1}^n \Ker(\operator_i - \lambda_i) \neq \Set{0}.
			\]
			If \( \Eig(\operatorFam,\lambda) \neq \Set{0} \) we call \( \Eig(\operatorFam,\lambda) \) the \emph{joint eigenspace} and an element \( v \in \Eig(\operatorFam,\lambda) \setminus \Set{0} \) a \emph{joint eigenvector}.
		\item
			The \emph{joint eigenvalue spectrum} \( \sigma_{\mathrm{p}}(\operatorFam) \) of \( \operatorFam \) is the set of all joint eigenvalues, i.e.,
			\[
			\sigma_{\mathrm{p}}(\operatorFam) \coloneqq \Set{\lambda \in \Cnum^n \given \Eig(\operatorFam, \lambda) \neq \Set{0}}.
			\]
	\end{enumerate}
\end{defi}

\begin{rem}\label{rem:jointEV-top} Every joint eigenvalue is a character \( \chi \in \Hom_{\Cnum\text{-}\mathrm{alg.}}(\mathcal A, \Cnum) \) and
			\[
			\bigcap\limits_{i = 1}^n \Eig(\operator_{i}, \chi(\genVar_{i})) \subseteq \Eig(B, \chi(B)) \quad \text{for all } B \in \mathcal A.
			\]
			That shows that \( \sigma_{\mathrm{p}}(\operatorFam) \) is independent of the generators \( \operatorFam = (\operator_{1},\ldots, \operator_{n}) \) and we may write \( \sigma_{\mathrm{p}}(\mathcal A) \).
\end{rem}

In general joint eigenvalue spectra on infinite dimensional vector spaces do not have nice properties. Therefore alternative joint spectra have been introduced in the literature. The following definition of the Taylor and the Co-Taylor spectrum of \( \operatorFam = (\operator_1,\ldots,\operator_n) \in \End(\complVs)^n \) in terms of homological algebra (see \cite[Chap.~2]{weibel95} for the definitions of the functors \( \Tor \) and \( \Ext \)) has its origin in \cite{taylor70,taylor72b,taylor72a}.

\begin{defi}[Regular Characters and Taylor Spectra]\label{defi:regularChar}
	Let \( \complVs_T \) be the \( \polyAlg \)-bimodule given in \thref{def:VT} and \( \chi \in \Hom_{\Cnum\text{-}\mathrm{alg.}}(\polyAlg, \Cnum) \).
	Recall the \( \polyAlg\)-bimodule \( \Cnum_\chi \) from \thref{def:Cchi}.
	\begin{enumerate}[(a)]
		\item We say \( \chi \) is a
			\begin{enumerate}[(i)]
				\item
					 \emph{regular character} if
					\(
					\Tor_i^{\polyAlg}(\Cnum_\chi, \complVs_\operatorFam) = 0\) for  all \(i \in \Znum	\),
				\item
					\emph{coregular character} if
					\(
					\Ext^i_{\polyAlg}(\Cnum_\chi, \complVs_{\operatorFam}) = 0\) for all \( i \in \Znum\).
			\end{enumerate}
		\item 	We say \( \chi \) is in the
		\begin{enumerate}[(i)]
				\item
					 \emph{Taylor spectrum} \( \sigma_{\mathrm{T}}(\operatorFam) \) of \( \operatorFam \) if it is not a regular character,
				\item
					 \emph{Co-Taylor spectrum} \( \sigma_{\mathrm{co}\text{-}\mathrm{T}}(\operatorFam) \) of \( \operatorFam \) if it is not a coregular character.
			\end{enumerate}
	\end{enumerate}	
\end{defi}

The definition of the Taylor spectra via homological functors may seem unnecessarily abstract compared to the definition in the original work of Taylor which uses modified Koszul complexes.
But just as the ordinary Koszul chain complex computes a \( \Tor \)-functor, the \( \Tor \)- and \( \Ext \)-functors used in \thref{defi:regularChar} can be computed by modified \emph{Koszul--Taylor chain} and \emph{cochain complexes} as described by Taylor.

As a reminder the Koszul--Taylor chain complex is given by the following sequence (see \cite[Sec.~4.5]{weibel95} where also the notation is explained)
\begin{equation}\label{fig:koszulChainComplex}
	\begin{tikzcd}[font=\normalsize]
		0 & \complVs_T \arrow[l] & \Cnum^n \otimes \complVs_T \arrow[l, "\tilde\delta_{\operatorFam}^1"'] & \Lambda^2(\Cnum^n) \otimes \complVs_T \arrow[l, "\tilde\delta_{\operatorFam}^2"'] & \ldots \arrow[l, "\tilde\delta_{\operatorFam}^3"'] & \Lambda^n(\Cnum^n) \otimes \complVs_T \arrow[l, "\tilde\delta_{\operatorFam}^{n}"'] & 0 \arrow[l]
	\end{tikzcd}
\end{equation}
\noindent
where the differentials \( \tilde\delta_{\operatorFam}^p \colon \Lambda^p(\Cnum^n) \otimes_\Cnum \complVs_T \to \Lambda^{p-1}(\Cnum^n) \otimes_\Cnum \complVs_T \) are given by
\begin{equation}\label{equ:diffOfKoszulChain}
	\tilde\delta_{\operatorFam}^p(v \otimes x) \coloneqq \sum\limits_{i=1}^n e_i \lrcorner v \otimes \operator_i x.
\end{equation}
The Koszul--Taylor cochain complex is given by the sequence
\begin{equation}\label{fig:koszulCochainComplexMotivation}
\begin{tikzcd}[font=\normalsize]
	0 \arrow[r] & \complVs_T \arrow[r, "\delta_{\operatorFam}^0"] & \Cnum^n \otimes \complVs_T \arrow[r, "\delta_{\operatorFam}^1"] & \Lambda^2(\Cnum^n) \otimes \complVs_T \arrow[r, "\delta_{\operatorFam}^2"] & \ldots \arrow[r, "\delta_{\operatorFam}^{n-1}"] & \Lambda^n(\Cnum^n) \otimes \complVs_T \arrow[r] & 0
\end{tikzcd}
\end{equation}
\noindent
with the differentials \( \delta_{\operatorFam}^p \colon \Lambda^p(\Cnum^n) \otimes_\Cnum \complVs_T \to \Lambda^{p+1}(\Cnum^n) \otimes_\Cnum \complVs_T \) given by
\begin{equation}\label{equ:diffOfKoszulCochain}
	\delta_{\operatorFam}^p(v \otimes x) \coloneqq \sum\limits_{i=1}^n e_i \wedge v \otimes \operator_i x.
\end{equation}

In the sequel we explain in more detail how the Koszul--Taylor chain complex corresponds to our definition via the \( \Tor \)-functor.
Afterwards we look into the identification of the definition of the Taylor spectrum via the Koszul cochain complex and the \( \Ext \)-functor.

In the following we discuss the details on this identifications.
First of all we define a ring-theoretic Koszul chain complex in a slightly more general way than needed. For more information see \cite[Section 4.5]{weibel95}.

\begin{defi}[Koszul Complex]\label{defi:koszulKomplex}
	Let \( R \) be a ring and \( \mathbf{x} = (x_1,\ldots,x_n) \) a finite sequence of central elements in \( R \).
	\begin{enumerate}[(i)]
		\item
			For a single central element \( x \in R \) we define the chain complex of \( R \)-bimodules.
			\[
			K(x): \quad 0 \to R \overset{x}{\to} R \to 0
			\]
			concentrated in degrees 1 and 0.
		\item
			The \emph{Koszul complex} \( K(\mathbf{x}) \) is the total tensor product complex
			\[
			K(x_1) \otimes_R K(x_2) \otimes_R \ldots \otimes_R K(x_n).
			\]
		\item
			For an \( R \)-module \( M \) we define
			\begin{align*}
				H_q(\mathbf{x}, M) &\coloneqq H_q(K(\mathbf{x}) \otimes_R M) \\
				H^q(\mathbf{x}, M) &\coloneqq H^q(\Hom_R(K(\mathbf{x}), M)).
			\end{align*}
	\end{enumerate}
\end{defi}

\begin{exmp}\label{exmp:settingKoszul}
	In our setting the ring \( R \) is the polynomial ring \( \polyAlg \) and the regular sequence \( \mathbf{x} = (\genVar_1 - \chi(\genVar_1), \ldots, \genVar_n - \chi(\genVar_n)) \) is a minimal generating set for the kernel \( \Ker \chi \) of a character \( \chi \in \Hom_{\Cnum\text{-}\mathrm{alg.}}(\polyAlg, \Cnum) \).
	The \( R \)-module \( M \) is the \( \polyAlg \)-bimodule  \( \complVs_T \).
\end{exmp}

\begin{rem}[Properties of the Koszul complex]\label{rem:onKoszulComplex}
	\mbox{}
	\begin{enumerate}[(i)]
		\item
			In \cite[Notation 4.5.1]{weibel95} we find:
			
			The degree \( p \) part of \( K(\mathbf{x}) \) is a free \( R \)-module generated by the symbols
			\[
			e_{i_1} \wedge \ldots \wedge e_{i_p} \coloneqq 1 \otimes \ldots \otimes 1 \otimes e_{x_{i_1}} \otimes \ldots \otimes e_{x_{i_p}} \otimes \ldots \otimes 1 \quad i_1 < \ldots < i_p.
			\]
			In particular \( K_p(\mathbf{x}) \) is isomorphic to the \( p \)-th exterior product \( \Lambda^p R^n \).
			The derivative \( K_p(\mathbf{x}) \to K_{p-1}(\mathbf{x}) \) is given by 
			\[
			e_{i_1} \wedge \ldots \wedge e_{i_p} \mapsto \sum\limits_{k=1}^p (-1)^k x_{i_k} e_{i_1} \wedge \ldots \wedge \widehat{e}_{i_k} \wedge \ldots \wedge e_{i_p}.
			\]
		\item\label{item:koszulComplexHomolGr}
			In \cite[Exercise 4.5.2]{weibel95} we find:
			\begin{enumerate}[(a)]
				\item\label{item:zeroTor}
					\( H_0(\mathbf{x}, M) = M/(x_1,\ldots,x_n)M \)
				\item\label{item:zeroExt}
					\( H^0(\mathbf{x}, M) = \Hom(R/\mathbf{x} R, M) = \Set{m \in M \given x_i m = 0 \quad \text{for all } i} \)
				\item\label{item:isomExtTor}
					There are isomorphisms
					\begin{equation}\label{equ:homAndCohomIsom}
						H_p(\mathbf{x}, M) \cong H^{n-p}(\mathbf{x}, M) \quad \text{for all } p.
					\end{equation}
			\end{enumerate}
	\end{enumerate}
\end{rem}

The important aspect of this discussion is the following \cite[Corollary 4.5.5]{weibel95}. 

\begin{cor}[Koszul resolution]\label{cor:koszulReso}
	If \( \mathbf{x} \) is a regular sequence in \( R \), then \( K(\mathbf{x}) \) is a free resolution of \( R/I \), where \( I \coloneqq (x_1,\ldots,x_n) R \).
	In particular, the following sequence is exact:
	\[
	0 \to \Lambda^n(R^n) \to \ldots \to \Lambda^2(R^n) \to R^n \overset{x}{\to} R \to R/I \to 0.
	\]
	In this case we have
	\begin{align*}
		\Tor_p^R(R/I, M) = H_p(\mathbf{x}, M); \\
		\Ext^p_R(R/I, M) = H^p(\mathbf{x}, M)
	\end{align*}
	for every left \( R \)-module \( M \).
\end{cor}

To compute \( \Tor_i^{\polyAlg}(\Cnum_\chi, \complVs_T) \) using \thref{cor:koszulReso} we tensor the resolution 
\[
0 \to \Lambda^n \polyAlg^n \to \ldots \to \Lambda^2 \polyAlg^n \to \polyAlg^n \to \polyAlg \to \Cnum_\chi \to 0.
\]
with our module \( \complVs_{\operatorFam} \) and compute the homology of the resulting complex
\[
0 \to \Lambda^n \polyAlg^n \otimes_{\polyAlg} \complVs_{\operatorFam} \to \ldots \to \polyAlg^n \otimes_{\polyAlg} \complVs_{\operatorFam} \to \polyAlg \otimes_{\polyAlg} \complVs_{\operatorFam} \to 0.
\]
Note that if \( \chi \) is the trivial character this is precisely the sequence from \Cref{fig:koszulChainComplex} with \( \Cnum \) replaced by \( \polyAlg \).
Else we need to replace the operator family \( \operatorFam = (\operator_1,\ldots,\operator_n) \) by \( \operatorFam - \chi \coloneqq (\operator_1 - \chi(\genVar_1),\ldots,\operator_n - \chi(\genVar_n)) \).
This gives rise to an alternative computation of \( \Tor_i^{\polyAlg}(\Cnum_\chi, \complVs_T) \).

\begin{rem}
	For any complex vector space \( W \) we can define a \( \polyAlg \)-module structure by setting \( W \otimes_\Cnum \polyAlg \) and acting with \( \polyAlg \) on the right factor.
	This procedure is called \emph{extension of scalars}.
	Another important observation is that we have isomorphism \( \Lambda^p \Cnum^n \cong \Cnum^{\binom{n}{p}} \) as \( \Cnum \)-modules and \( \Lambda^p \polyAlg^n \cong \polyAlg^{\binom{n}{p}} \) as \( \polyAlg \)-modules.
	Combining these two observations gives us the following isomorphisms of \( \polyAlg \)-modules
	\[
	\Lambda^p \Cnum^n \otimes_{\Cnum} \polyAlg \cong \Cnum^{\binom{n}{p}} \otimes_{\Cnum} \polyAlg \cong \polyAlg^{\binom{n}{p}} \cong \Lambda^p \polyAlg^n.
	\]
	This shows that \( \Lambda^p \Cnum^n \otimes_{\Cnum} \polyAlg \) is also a projective \(\polyAlg \)-module. As the tensor product is associative we obtain
	\begin{align}\label{equ:isomToLitKoszul}
		\Lambda^p \polyAlg^n	\otimes_{\polyAlg} \complVs_{\operatorFam} 
		&\cong	\left( \Lambda^p \Cnum^n \otimes_\Cnum \polyAlg\right) \otimes_{\polyAlg} \complVs_{\operatorFam} \nonumber\\
		&\cong \Lambda^p \Cnum^n \otimes_\Cnum \left( \polyAlg \otimes_{\polyAlg} \complVs_{\operatorFam} \right) \\
		&\cong \Lambda^p \Cnum^n \otimes_\Cnum \complVs_{\operatorFam}, 
	\end{align}
	where \( \polyAlg \) acts on the right component of \( \Lambda^p \Cnum^n \otimes_\Cnum \complVs_{\operatorFam} \).
	The differentials 
	\[
	\widetilde\delta^p \colon \Lambda^p \polyAlg^n \otimes_{\polyAlg} \complVs_{\operatorFam} \to \Lambda^{p-1} \polyAlg^n \otimes_{\polyAlg} \complVs_{\operatorFam}
	\] 
	are given by
	\begin{align*}
		\widetilde\delta^p(e_{i_1} \wedge \ldots \wedge e_{i_p} \otimes x) &= \sum\limits_{k=1}^p (-1)^k \genVar_{i_k} e_{i_1} \wedge \ldots \wedge \widehat{e}_{i_k} \wedge \ldots \wedge e_{i_p} \otimes x \\
		&= \sum\limits_{k=1}^p (-1)^k e_{i_1} \wedge \ldots \wedge \widehat{e}_{i_k} \wedge \ldots \wedge e_{i_p} \otimes \genVar_{i_k}.x \\
		&= \sum\limits_{k=1}^p (-1)^k e_{i_1} \wedge \ldots \wedge \widehat{e}_{i_k} \wedge \ldots \wedge e_{i_p} \otimes \operator_{i_k} x.
	\end{align*}
	These transform into the differentials \( \tilde{\delta}_{\operatorFam}^p \) from \eqref{equ:diffOfKoszulChain} under the isomorphisms from \eqref{equ:isomToLitKoszul}.
	Thus \eqref{equ:diffOfKoszulChain} also computes \( \Tor_i^{\polyAlg}(\Cnum_\chi, \complVs_T) \) for the trivial character \( \chi \).
\end{rem}

We collect the results from the previous remark in the following proposition.

\begin{prop}
	Let \( (\Lambda^\bullet \Cnum \otimes \complVs_{\operatorFam}, \widetilde{\delta}_{\operatorFam- \chi}^p) \) be the Koszul--Taylor cochain complex from \Cref{fig:koszulChainComplex}.
	Denote by \( H^p(\widetilde{\delta}_{\operatorFam-\chi}^\bullet) \) the \( p \)-th cohomology group.
	Then we have an isomorphism of \( \polyAlg \)-modules
	\[
	H^p(\widetilde{\delta}_{\operatorFam-\chi}^{\bullet}) \cong \Tor_p^{\polyAlg}(\Cnum_\chi, \complVs_{\operatorFam}) \quad \text{for all } p \in \Znum.
	\]
\end{prop}

\begin{rem}
	Suppose that \( \chi \in \sigma_{\mathrm{p}}(\operatorFam) \) is a joint eigenvalue.
	Then \thref{rem:onKoszulComplex}\ref{item:koszulComplexHomolGr}\ref{item:zeroExt} implies that
	\begin{align*}
		H^0((\operatorFam - \chi), \complVs_{\operatorFam}) &= \Set{x \in \complVs \given (\genVar_i - \chi(\genVar_i)).x = 0 \quad \text{for all } i} \\
		&= \Set{x \in \complVs \given (\operator_i - \chi(\genVar_i))x = 0 \quad \text{for all } i} \\
		&= \bigcap\limits_{i=1}^n \Ker(\operator_i - \chi(\genVar_i)).
	\end{align*}
	Hence the joint eigenvalues are contained in the Co-Taylor spectrum.
	But \eqref{equ:homAndCohomIsom} in \thref{cor:koszulReso} shows that a character \( \chi \in \Hom_{\Cnum\text{-}\mathrm{alg.}}(\polyAlg, \Cnum) \) is regular if and only if it is coregular.
	Altogether, we have
	\begin{equation}\label{eq:pSubeqTEqCo-T}
		\sigma_{\mathrm{p}}(\operatorFam) \subseteq \sigma_{\mathrm{T}}(\operatorFam) = \sigma_{\mathrm{co}\text{-}\mathrm{T}}(\operatorFam).
	\end{equation}
\end{rem}	
		
\begin{lem}[{{\cite[Lemma 3.8]{bghw20}}}]\label{lem:taylorSpecOnFiniteSpace}
	Suppose that \( \complVs \) is finite dimensional.
	Then we have 
	\[
	\sigma_{\mathrm{p}}(\operatorFam)=\sigma_{\mathrm{T}}(\operatorFam).
	\]
\end{lem}

\subsection{Banach Modules}\label{subsec:Banach-modules}

In this subsection we keep the notation of \Cref{subsec:alg-prelim} but assume in addition that \( \complVs \) is a Banach space with respect to a given norm \( \norm{\cdot}_{\complVs} \) and the family \( (\operator_1,\ldots,\operator_n) \) of commuting linear endomorphisms of \( \complVs \) consists of bounded operators on \( (\complVs , \norm{\cdot}_V) \).
We use \cite{eschmeierPutinar96} as a general reference for the specifics of the homological algebra of Banach complexes.

As \( (\complVs , \norm{\cdot}_{\complVs}) \) is a Banach space we can equip the algebra \( \mathcal L(\complVs) \) of continuous endomorphisms of \( \complVs \) with the operator norm.
Let \( \overline{\mathcal A} \) be the closure of \( \mathcal A \) in \( \mathcal L(\complVs) \) and note that \( \overline{\mathcal A} \) is a Banach algebra.
	
\begin{rem}\label{rem:barA-characters}
	As \( \mathcal A \) is dense in \( \overline{\mathcal A} \) we have
	\[
	\Hom_{\Cnum\text{-}\mathrm{alg.}}(\overline{\mathcal A}, \Cnum) \hookrightarrow \Hom_{\Cnum\text{-}\mathrm{alg.}}(\mathcal A, \Cnum) \hookrightarrow \Hom_{\Cnum\text{-}\mathrm{alg.}}(\polyAlg, \Cnum).
	\]
	Moreover, we have \( \operatorFam = (\operator_1,\ldots,\operator_n) \subseteq \overline{\mathcal A} \) hence they are in a commutative Banach algebra with identity.
	By \cite[Lemma 3.2 and Theorem 3.3]{taylor70} the Taylor spectrum for a family of operators in a Banach algebra is contained in the characters of the Banach algebra, i.e.,
	\[
	\sigma_{\mathrm{T}}(\operatorFam) \subseteq \Hom_{\Cnum\text{-}\mathrm{alg.}}(\overline{\mathcal A}, \Cnum).
	\]
	Hence we know that
	\begin{enumerate}[1.]
		\item
			Every character \( \chi \in \Hom_{\Cnum\text{-}\mathrm{alg.}}(\polyAlg, \Cnum) \) that lies in the Taylor spectrum \( \sigma_{\mathrm{T}}(\operatorFam) \) is a character for the Banach algebra \( \overline{\mathcal A} \).
		\item
			If \( \operatorFam^\prime = (\operator_1^\prime, \ldots, \operator_m^\prime) \) is another commuting family generating \( \mathcal A \) (or \( \overline{\mathcal A} \)) we have \( \sigma_{\mathrm{T}}(\operatorFam^\prime) = \sigma_{\mathrm{T}}(\operatorFam) \subseteq \Hom_{\Cnum\text{-}\mathrm{alg.}}(\overline{\mathcal A}, \Cnum) \).
			That means the characters (of \( \overline{\mathcal A} \)) in the Taylor spectrum are independent of the generators (see \cite[Theorem 3.3]{taylor70} and note that \( \Delta(\overline{\mathcal A}, \complVs) \) is the same as our spectrum).
	\end{enumerate}
	By abuse of notation we plug in operators in our characters \( \chi \in \Hom_{\Cnum\text{-}\mathrm{alg.}}(\polyAlg, \Cnum) \).
	If this would cause a problem (i.e., \( \chi \notin \Hom_{\Cnum\text{-}\mathrm{alg.}}(\mathcal A, \Cnum) \)) this character is not in the spectrum.
\end{rem}

\begin{defi}[Essential Taylor Spectrum]\label{defi:essentialTaylorSpec}
We say the character \( \chi \in \Hom_{\Cnum\text{-}\mathrm{alg.}}(\polyAlg, \Cnum) \) is
	\begin{enumerate}[(a)]
		\item
			a \emph{Fredholm character} for \( \operatorFam \) if
			\begin{align*}
				&\dim_{\Cnum\text{-vs.}} \Tor_i^{\polyAlg}(\Cnum_\chi, \complVs_{\operatorFam}) < \infty &\text{for all } i \in \Znum \\ 
				&\dim_{\Cnum\text{-}\mathrm{vs.}} \Tor_i^{\polyAlg}(\Cnum_\chi, \complVs_{\operatorFam}) = 0&\text{for all but finitely many } i \in \Znum; 
			\end{align*}
		\item
			is a \emph{Co-Fredholm character} for $T$ if
			\begin{align*}
				&\dim_{\Cnum\text{-}\mathrm{vs.}} \Ext^i_{\polyAlg}(\Cnum_\chi, \complVs_{\operatorFam}) < \infty&\text{for all } i \in \Znum \\ 
				&\dim_{\Cnum\text{-vs.}} \Ext^i_{\polyAlg}(\Cnum_\chi, \complVs_{\operatorFam}) = 0&\text{for all but finitely many } i \in \Znum;
			\end{align*}
		\item
			in the \emph{essential Taylor spectrum} \( \sigma^{\mathrm{ess}}_{\mathrm{T}}(\operatorFam) \) if \( \chi \in \sigma_T(T)\) is not a Fredholm character;
		\item
			in the \emph{essential Co-Taylor spectrum} \( \sigma^{\mathrm{ess}}_{\mathrm{co-T}}(\operatorFam) \) if \( \chi \in \sigma_T(T) \) is not a Co-Fredholm character.
	\end{enumerate}
\end{defi}

\begin{rem}[Fredholm Banach Complexes]\label{rem:BanachFredholm}
	There are alternative definitions of Fredholm designed particularly for complexes of Banach spaces.
	We compare them with the notions from \thref{defi:essentialTaylorSpec}.
	\begin{enumerate}[(i)]
		\item
			A Banach space complex \( (\banachSpace^\bullet, \delta^\bullet) \) 
			\begin{center}
				\begin{tikzcd}[font=\normalsize]
					0 \arrow[r] & \banachSpace^{0} \arrow[r, "\delta^{0}"] & \banachSpace^{1} \arrow[r, "\delta^{1}"] & \ldots \arrow[r, "\delta^{n-1}"] & \banachSpace^n \arrow[r] & 0
				\end{tikzcd}
			\end{center}
			is called 
			\begin{enumerate}[(a)]
				\item
					\emph{Fredholm} if \( \delta^\bullet \) has closed range and the function
					\[
					\Znum \ni p \mapsto \dim_{\Cnum} H^p(\delta^\bullet) \in \Nnum_0 \cup \Set{\infty}
					\]
					is finite (and has finite support);
				\item
					\emph{split Fredholm in degree \( p \)} for \( 0 \leq p \leq n \) if there are operators
					\begin{center}
						\begin{tikzcd}[font=\normalsize]
							\banachSpace^{p-1} & \banachSpace^p \arrow[l, "h^p"'] & \banachSpace^{p+1} \arrow[l, "h^{p+1}"']
						\end{tikzcd}
					\end{center}
					such that \( \id_{\banachSpace^p} - (\delta h+h\delta) \in \mathcal K(\banachSpace^p, \banachSpace^p) \);
				\item
					\emph{split Fredholm} if it is split Fredholm in every degree;
				\item
					\emph{algebraically Fredholm} if \( \dim_\Cnum H^p(\delta^\bullet) < \infty \) for all \( p \in \Znum \) and \( \dim_\Cnum H^p(\delta^\bullet) = 0 \) for all but finitely many indices.
			\end{enumerate}
		\item\label{item:boundedBanachAlgFredholm}
			Suppose that \( (\banachSpace^p, \delta^p)_{p \in \Znum} \) is a Banach complex of finite length, which is algebraically Fredholm.
			As the \( \delta^p \) are all bounded, so that \( \Ker \delta^p \) is closed in \( \banachSpace^p \), this implies that \( \Image \delta^p \) is of finite codimension, hence closed in \( \Ker \delta^p \).
			This means that all \( \delta^p \) have closed range, so that \( (\banachSpace^p,\delta^p)_{p \in \Znum} \) is Fredholm.
		\item\label{item:fredholmCharHaveFredholmCompl}
			Note that the complexes \eqref{equ:diffOfKoszulChain} and \eqref{equ:diffOfKoszulCochain} are Banach complexes.
			Thus, from \ref{item:boundedBanachAlgFredholm} we obtain that for each Fredholm character \( \chi \in \Hom_{\Cnum\text{-}\mathrm{alg.}}(\polyAlg, \Cnum) \) the finite complexes computing \( \Tor_i^{\polyAlg}(\Cnum_\chi, \complVs_T) \) and \( \Ext^i_{\polyAlg}(\Cnum_\chi, \complVs_T) \) are Fredholm. 
		\item\label{item:fredholmEquiCoFredholm}
			By \thref{rem:onKoszulComplex}\ref{item:koszulComplexHomolGr}.\ref{item:isomExtTor} the two definitions of Fredholm and Co-Fredholm are equivalent.
			Consequently the essential Taylor and essential Co-Taylor spectrum coincide as well.
		\item\label{item:splitFredholm}
			According to \cite[Proposition 10.2.2]{eschmeierPutinar96} a Banach complex \( (\banachSpace^\bullet, \delta^\bullet) \) is split Fredholm if and only if there are operators \( \varepsilon^p \in \mathcal L(\banachSpace^p, \banachSpace^{p-1}) \) with the property
			\[
			\delta^{p-1} \varepsilon^p + \varepsilon^{p+1} \delta^p \in I + \mathcal K(\banachSpace^p).
			\]		
		\item\label{item:splitFredholmImpliesFredholm}
			\cite[Lemma 2.6.13]{eschmeierPutinar96} implies that every bounded split Fredholm complex is Fredholm.
		\item\label{item:essTaylorSpecIndOfChoice}
			\cite[Section IV.34]{mueller07} tells us that one can compute the essential Taylor spectrum of a family of operators by computing the Taylor spectrum of on a modified space.
			Consequently by discussion in \thref{rem:barA-characters} we find that the essential Taylor spectrum is independent of the generators of \( \mathcal A \) as well.
	\end{enumerate}
\end{rem}

\begin{prop}[{{\cite[Lemma 2.2.4]{eschmeierPutinar96}}}]\label{prop:homotopyViaDualKoszul}
	If \( B = (B_1,\ldots,B_n) \in \mathcal L(\complVs_{\operatorFam})^n \) belong to the commutant of \( \operatorFam \), i.e., \( [B_i,T_j]=0 \) for all \( i,j\), and \( C = \sum_{i=1}^n \operator_i B_i \) is a bounded operator on \( \complVs \), then 
	\[
	C \colon \Ext_{\polyAlg}^p(\Cnum, \complVs_{\operatorFam}) \to \Ext_{\polyAlg}^p(\Cnum, \complVs_{\operatorFam}), 
	\]
	induced by the \( \Ext \)-functor via the componentwise action of \( C \), is the zero operator for each \( p \).
	The same can be shown for \( \Tor^{\polyAlg}_\ast(\Cnum, V_{\operatorFam}) \).
	
	Here \( \Cnum = \Cnum_\chi \) for the trivial character \( \chi \) which is the identity on \( \Cnum \) and zero everywhere else.
\end{prop}

As it will play an important role we shortly discuss the proof of \thref{prop:homotopyViaDualKoszul} as presented in \cite{eschmeierPutinar96} for our setting.

\begin{proof}
	We use the Koszul resolution to compute the homology groups for \( \Tor \).
	Note that the differentials \( \widetilde{\delta}_{\operatorFam}^p\) from \eqref{equ:diffOfKoszulChain} are given by
	\[
	\widetilde{\delta}_{\operatorFam}^p(e_{i_1} \wedge \ldots \wedge e_{i_p} \otimes x) = \sum\limits_{k=1}^p (-1)^k e_{i_1} \wedge \ldots \wedge \widehat{e}_{i_k} \wedge \ldots \wedge e_{i_p} \otimes \operator_{i_k} x.
	\]
	We define module homomorphisms \( \alpha^p \colon \Lambda^p \Cnum^n \otimes_{\Cnum} V_{\operatorFam} \to \Lambda^{p+1} \Cnum^n \otimes_{\Cnum} V_{\operatorFam} \) by
	\[
	\alpha^p(e_{i_1} \wedge \ldots \wedge e_{i_p} \otimes x) = \sum\limits_{k=1}^n e_k \wedge e_{i_1} \wedge \ldots \wedge e_{i_p} \otimes B_k x.
	\]
	Here we want to let \( \genVar_i \) act on \( \complVs \) via \( B_i \) this coincides with the differential \( \delta_{B}^p \).
	
	One can check that
	\[
	\left(\alpha^{p-1} \widetilde{\delta}_{\operatorFam}^p + \widetilde{\delta}_{\operatorFam}^{p+1} \alpha^p \right)(e_I \otimes x) = e_I \otimes C x = (\id_{\Lambda^p \Cnum^n} \otimes C)(e_I \otimes x)
	\]
	holds for each tuple \( I \) if strictly increasing integers \( 1 \leq i_1 < \ldots < i_p \leq n \) and each \( x \in V \).
	The assertion now follows from the following calculation:
	Let \( u\in\Ker\widetilde{\delta}_{\operatorFam}^p \), then
	\[
	C u = \left(\alpha^{p-1} \widetilde{\delta}_{\operatorFam}^p + \widetilde{\delta}_{\operatorFam}^{p+1} \alpha^p \right) u = \widetilde{\delta}_{\operatorFam}^{p+1} \alpha^p u
	\]
	and consequently \( C \) is the zero operator on \( \Ker\widetilde{\delta}_{\operatorFam}^p / \Image \widetilde{\delta}_{\operatorFam}^{p+1} \).
\end{proof}

\begin{rem}\label{rem:EP-Lemma_chi}
	\mbox{}
	\begin{enumerate}[(i)]
		\item\label{item:CisChainMap}
			The proof of \thref{prop:homotopyViaDualKoszul} shows that \( \id \otimes C \) is a chain map.
		\item\label{item:CforArbiChar}
			For an arbitrary character \( \chi \) we need to replace \( C \) by \( \sum_{i=1}^n (\operator_i - \chi(\operator_i))B_i \).
			That means the operator \( C = \sum_{i=1}^n (\operator_i - \chi(\operator_i))B_i \) induces the zero operator on \( \Tor_\ast^{\polyAlg}(\Cnum_\chi, V_{\operatorFam}) \).
			To see that one takes the family \( \operatorFam - \chi \coloneqq (\operator_1 - \chi(\operator_1),\ldots,\operator_n - \chi(\operator_n)) \) and follows the steps of the proof.
	\end{enumerate}
\end{rem}

\begin{lem}\label{lem:zeroToInvertibleImpliesToOne}
	Let \( (\banachSpace^\bullet, d^\bullet) \) be a cochain complex of finite length of Banach spaces
 \begin{center}
  \begin{tikzcd}[font=\normalsize]
   0 \arrow[r] & \banachSpace^{0} \arrow[r, "d^{0}"] & \banachSpace^{1} \arrow[r, "d^{1}"] & \ldots \arrow[r, "d^{n-1}"] & \banachSpace^n \arrow[r] & 0
  \end{tikzcd}
 \end{center}
	and \( Q^p \in \mathcal L(\banachSpace^p ,\banachSpace^{p-1}) \) such that
	\[
	Q^{p+1} d^p + d^{p-1} Q^p = T^p + K^p
	\]
	with \( T^p \in \mathcal L(\banachSpace^p)\) invertible and \( K^p \in \mathcal K(\banachSpace^p) \) compact for every \( p \).
	Then \( (\banachSpace^\bullet, d^\bullet) \) is a Fredholm complex.
	
	In other words, if there exist linear maps \( (Q^\bullet) \) such that \( d Q + Q d \) is Fredholm of index \( 0 \) (see \thref{defi:fredholmIndexZero}) then \( (\banachSpace^\bullet, d^\bullet) \) is a Fredholm complex.
\end{lem}

\begin{proof}
	We show that there exists a family of linear maps \( (\widetilde{Q}^\bullet) \) such that
	\begin{equation}\label{equ:homotopyToFredholm}
		\widetilde{Q}^{p+1} d^p + d^{p-1} \widetilde{Q}^p = \id^p + \widetilde{K}^p
	\end{equation}
	for some compact operator \( \widetilde{K}^p \in \mathcal K(\banachSpace^p) \).
	As \(T^p + K^p = Q^{p+1} d^p + d^{p-1} Q^p\) one directly sees that \( (T^p + K^p) \) is a chain morphism (see \cite[p.~16f.]{weibel95}) that means
	\[
	d^p \circ (T^p+K^p) = (T^{p-1}+K^{p-1}) \circ d^p.
	\]
	Transforming this we find \( d^p \circ T^p - T^{p+1} \circ d^p \in \mathcal K(\banachSpace^p, \banachSpace^{p+1}) \).
	Since \( \mathcal L(X^{p+1}) \cdot \mathcal K(\banachSpace^p, \banachSpace^{p+1}) \subseteq \mathcal K(\banachSpace^p, \banachSpace^{p+1}) \) and \( \mathcal K(\banachSpace^p, \banachSpace^{p+1}) \cdot \mathcal L(\banachSpace^p) \) we also get
	\begin{equation}\label{equ:commuUptoCompact}
		d^p (T^p)^{-1} - (T^{p+1})^{-1} d^p \in \mathcal K(\banachSpace^p, \banachSpace^{p+1}).
	\end{equation}
	We define \( \widetilde{Q}^p \colon \banachSpace^p \to \banachSpace^{p-1} \) as follows
	\[
	\widetilde{Q}^p \coloneqq 
	\begin{cases}
		0 & \text{if } p = 0 \\
		(T^{p-1})^{-1} Q^p & \text{if } p = 1,\ldots,n+1 \\
	\end{cases}
	\]
	We check that this gives us \eqref{equ:homotopyToFredholm}.
	We compute
	\begin{align*}
		\widetilde{Q}^{p+1} d^p + d^{p-1} \widetilde{Q}^p &= (T^p)^{-1} Q^{p+1} d^p + d^{p-1} (T^{p-1})^{-1} Q^p \\
		\shortintertext{by using \eqref{equ:commuUptoCompact} we find a compact operator \( K \in \mathcal K(X^{p-1}, X^p) \) such that}
		(T^p)^{-1} Q^{p+1} d^p + d^{p-1} (T^{p-1})^{-1} Q^p &= (T^p)^{-1} Q^{p+1} d^p + \left( (T^p)^{-1} d^{p-1} + K \right) Q^p \\
		&= (T^p)^{-1} Q^{p+1} d^p + (T^p)^{-1} d^{p-1} Q^p + K Q^p \\
		&= (T^p)^{-1} \left(Q^{p+1} d^p + d^{p-1} Q^p \right) + K Q^p \\
		&= (T^p)^{-1} (T^p + K^p) + K Q^p \\
		&= \id^p + (T^p)^{-1} K^p + K Q^p.
	\end{align*}
	As the sum of two compact operators remains compact we have found the desired form of \eqref{equ:homotopyToFredholm}.
	From \thref{rem:BanachFredholm}\ref{item:splitFredholm} we know that the cochain complex \( (\banachSpace^\bullet, d^\bullet) \) is split Fredholm.
	Then \thref{rem:BanachFredholm}\ref{item:splitFredholmImpliesFredholm} shows that it is Fredholm as well.
\end{proof}

For the following lemma recall the Riesz projector from \thref{prop:rieszProjEigspaceInImage-Text}.

\begin{lem}[cf.\ {\cite[Lemma 3.9]{bghw20}}]\label{lem:projectionOnEigenspace}
	Let \( \chi \in \Hom_{\Cnum\text{-}\mathrm{alg.}}(\polyAlg, \Cnum) \).
	Assume there exist operators \( B_1,\ldots,B_n \in \mathcal A \) such that
	\[
	C - \lambda \id_{\complVs} \coloneqq \sum\limits_{i=1}^n (\operator_i - \chi(\operator_i)) B_i \colon \complVs \to \complVs
	\]
	where \( C \) is a quasi-compact operator on \( \complVs \) and \( \lambda \in \sigma(C) \) with \( \abs{\lambda} > \essSpecRad(C) \).
	Denote by \( \widetilde{P}_0 \) the Riesz projector on the eigenvalue \( 0 \) of the operator \( C - \lambda \id_{\complVs} \).
	Then we have
	\begin{align*}
		\Tor_i^{\polyAlg}(\Cnum_\chi, \complVs_{\operatorFam}) &\cong \Tor_i^{\polyAlg}(\Cnum_\chi, \Image \widetilde{P}_0), \\
		\Ext^i_{\polyAlg}(\Cnum_\chi, \complVs_{\operatorFam}) &\cong \Ext^i_{\polyAlg}(\Cnum_\chi, \Image \widetilde{P}_0).
	\end{align*}
\end{lem}

\begin{proof}
	We only prove the \( \Tor \)-case.
	
	We know that \( P_0 \coloneqq \id_{\Lambda^p \Cnum^n} \otimes \widetilde{P}_0 \) is the Riesz projector of \( F \coloneqq C - \lambda \id_{\complVs} \) on the eigenvalue \( 0 \).
	For every \( m \in \Nnum_0 \) the operator \( F^m \) commutes with \( \operator_i \) and \( B_i \) for every \( i = 1,\ldots,n \).
	
	Hence the spaces \( \Ker F^j \) and \( \Image F^j \) are \( \polyAlg \)-submodules of \( \complVs_{\operatorFam} \).
	
	Let \( j \in \Nnum_0 \) be the smallest number such that \( \Ker F^j = \Ker F^{j+k} \) for all \( k \in \Nnum_0 \).
	The decomposition \( \complVs = \Ker F^j \oplus \Image F^j \) is therefore a decomposition into \( \polyAlg \)-submodules which are also \( F \)-invariant.
	As \( \Tor \) is additive in the second component by \cite[Cor.~2.6.12]{weibel95} we know that
	\[
	\Tor_i^{\polyAlg}\left(\Cnum_\chi, \complVs_{\operatorFam} \right) \cong \Tor_i^{\polyAlg}\left(\Cnum_\chi, \Ker F^j \right) \oplus \Tor_i^{\polyAlg}\left(\Cnum_\chi, \Image F^j \right).
	\]
	We want to show that \( \Tor_i^{\polyAlg}(\Cnum_\chi, \Image F^j) = 0 \).

	We take the Koszul resolution of \( \Cnum_\chi \) as a \( \polyAlg \)-module to compute the homology groups.
	Since both \( \Ker F^j \) and \( \Image F^j \) are both \( \operatorFam \)- and \( B \)-invariant, the chain homotopy \( \id_{\Lambda^p \Cnum^n} \otimes F \) restricts to the Koszul complexes tensored with \( \Ker F^j \) and \( \Image F^j \) respectively.
	From \thref{rem:EP-Lemma_chi}\ref{item:CisChainMap} we know that \( F_p \coloneqq \id_{\Lambda^p \Cnum^n} \otimes F \) is a chain morphism and moreover by \thref{prop:homotopyViaDualKoszul} the induced morphism on the homology groups is the zero map.
	
	On the other hand, we also know that (see \thref{rem:propOfFredholmAndQuasiCompOp}\ref{item:splittingAndIsomOfGenEigen}) \( \id_{\Lambda^p \Cnum^n} \otimes F \colon \Lambda^p \Cnum^n \otimes \Image F^j \to \Lambda^p \Cnum^n \otimes \Image F^j \) is an isomorphism for every \( p \).
	As \( H_p(\id_{\Lambda^p \Cnum^n} \otimes F) \) is functorial, isomorphisms map to isomorphisms.
	Hence the zero map is an isomorphism and therefore \( \Tor_i^{\polyAlg}(\Cnum_\chi, \Image F^j) = 0 \) for all \( i \in \Znum \).	

\end{proof}

\subsection{Representations of Finitely Generated Commutative Monoids}\label{subsec:repOnFinitelyGenCommuMonoids}

We keep the hypotheses and notations from \Cref{subsec:Banach-modules}.
In addition, let \( (\semigr,\ast,e) \) be an abelian monoid generated by elements \( \varpi_1,\ldots,\varpi_n \).
For \( \semigrelemB = \varpi_1^{\ell_1} \ast \ldots \ast \varpi_n^{\ell_n}\in \semigr \) with \( \ell_j \in \Nnum_0 \) we set \( \operator_0 = \id_{\complVs} \) and 
\[
\operator_\semigrelemB \coloneqq \operator_1^{\ell_1} \circ \ldots \circ \operator_n^{\ell_n}.
\]
Then \( \semigr \to \mathcal L(\complVs),\ \semigrelemB \mapsto \operator_\semigrelemB \) is a homomorphism of monoids satisfying \( \operator_{\semigrgen_j} = \operator_j \).
This extends to an algebra morphism \( \Cnum[\semigr] \to \mathcal L(\complVs) \).
By the discussion in \thref{rem:barA-characters} we can naturally identify the Taylor spectrum as a subset \( \sigma_{\mathrm{T}}(\operatorFam) \subseteq \Hom_{\Cnum\text{-}\mathrm{alg.}}(\Cnum[\semigr], \Cnum) \).

\begin{prop}\label{prop:parametrixForSemigr}
	Given a character \( \chi \in \Hom_{\Cnum\text{-}\mathrm{alg.}}(\polyAlg, \Cnum) \) and some \( \semigrelemB \in \semigr \), there exists a family of operators \( B(\chi) = (B_1(\chi), \ldots, B_n(\chi)) \in \mathcal L(V)^n \) such that
	\begin{enumerate}[(i)]
		\item
			\( B_i(\chi) \in \mathcal A = \spr{\operator_{\semigrgen_1},\ldots, \operator_{\semigrgen_n}}_{\Cnum\text{-}\mathrm{alg.}} \), in particular \( B(\chi) \) is a commuting family and \( [\operator_{\semigrgen_i}, B_j(\chi)] = 0 \) for all \( i,j = 1,\ldots,n \) and
		\item
			the \( B_i(\chi) \) satisfy
			\[
			\sum\limits_{i=1}^n (\operator_{\semigrgen_i} - \chi(\operator_{\semigrgen_i}))B_i(\chi) = \operator_{\semigrelemB} - \chi(\operator_{\semigrelemB}).
			\]
	\end{enumerate}
\end{prop}

\begin{proof} Let \(\semigrelemB = \sum\limits_{i=1}^n \ell_i \semigrgen_i\) with
\( \ell_i \in \Nnum_0 \), so that
	\[
	\operator_\semigrelemB = \prod\limits_{i=1}^n \operator_{\semigrgen_i}^{\ell_i} \quad \text{and} \quad \chi(\operator_\semigrelemB) = \prod\limits_{i=1}^n \chi(\operator_{\semigrgen_i})^{\ell_i}.
	\]
	The polynomial \( P \coloneqq \prod_{i=1}^n X_i^{\ell_i} - \prod\limits_{i=1}^n \chi(\operator_{\semigrgen_i})^{\ell_i} \) vanishes at
	\( b \coloneqq (\chi(\operator_{\semigrgen_1}), \ldots, \chi(\operator_{\semigrgen_n})) \), so by Hilbert's Nullstellensatz \( P \) is contained in the ideal 
	\[
	(X_1 - \chi(\operator_{\semigrgen_1}),\ldots,X_n - \chi(\operator_{\semigrgen_n}))
	\]
	generated by the linear polynomials \( X_j - \chi(\operator_{\semigrgen_j}) \).
	That means there exists polynomials \( Q_1,\ldots,Q_n \in \nPolyAlg \) such that
	\[
	\sum\limits_{i=1}^n (X_i - \chi(\operator_{\semigrgen_i})) \cdot Q_i = P = \prod\limits_{i=1}^n X_i^{\ell_i} - \prod\limits_{i=1}^n \chi(\operator_{\semigrgen_i})^{\ell_i}. 
	\]
	Then we choose \( B_i(\chi) \) to be the image of the \( Q_i \) under the surjective algebra homomorphism \( \polyAlg \to \mathcal A \).
	Then the \( B_i(\chi) \) clearly have the desired properties.
\end{proof}

Note that \thref{prop:parametrixForSemigr} and its proof make no use of the Banach space structure of \( \complVs \).
The topology is, however, needed for the following definition.

\begin{defi}
	\( \semigr^\circ \coloneqq \Set{\semigrelemB \in \semigr \given  \operator_\semigrelemB \text{ is quasi-compact}} \).
\end{defi}

We will always want that \( \semigr^\circ \neq \emptyset \) so there exists at least one quasi-compact operator among the \( \operator_\semigrelemB \).

\begin{prop}\label{prop:Fredholm} 
	Given a character \( \chi \in \Hom_{\Cnum\text{-}\mathrm{alg.}}(\polyAlg, \Cnum) \) and some \( \semigrelemB \in H^\circ \).
	If \( \abs{\chi(\operator_{\semigrelemB})} > \essSpecRad(\operator_{\semigrelemB}) \), then \( \chi \) is Fredholm in the sense of \thref{defi:essentialTaylorSpec}.
\end{prop}

\begin{proof}
	The idea of the proof is to apply \thref{lem:zeroToInvertibleImpliesToOne} to the Koszul resolution.
	Let
	\[
	0 \to \Lambda^n \Cnum^n \otimes_{\Cnum} \complVs_{\operatorFam} \to \ldots \to \Cnum^n \otimes_{\Cnum} \complVs_{\operatorFam} \to \Cnum \otimes_{\Cnum} \complVs_{\operatorFam} \to 0
	\]
	be the Koszul complex/resolution for \( \operatorFam \).
	That means the differentials \( \widetilde{\delta}_{\operatorFam}^p \colon \Lambda^p \Cnum^n \otimes_{\Cnum} \complVs_{\operatorFam} \to \Lambda^p \Cnum^n \otimes_{\Cnum} \complVs_{\operatorFam}\) are given by
	\[
	\widetilde{\delta}_{\operatorFam}^p(e_{i_1} \wedge \ldots \wedge e_{i_p} \otimes x) = \sum\limits_{k=1}^p (-1)^k e_{i_1} \wedge \ldots \wedge \widehat{e}_{i_k} \wedge \ldots \wedge e_{i_p} \otimes \operator_{\varpi_{i_k}} x.
	\]
	The family \( B(\chi) = (B_1(\chi),\ldots,B_n(\chi)) \) constructed in \thref{prop:parametrixForSemigr} lies in the commutant of \( \operatorFam = (\operator_{\semigrgen_1} - \chi(\operator_{\semigrgen_1}),\ldots,\operator_{\semigrgen_n} - \chi(\operator_{\semigrgen_n})) \).
	Hence with the same argument as in the proof of \thref{prop:homotopyViaDualKoszul} we obtain the operator
	\begin{align*}
		\left(\alpha^{p-1} \widetilde{\delta}_{\operatorFam}^p + \widetilde{\delta}_{\operatorFam}^{p+1} \alpha^p \right) &= \left(\id_{\Lambda^p \Cnum^n} \otimes \sum\limits_{i=1}^n (\operator_{\semigrgen_i} - \chi(\semigrgen_i)) B_i(\chi) \right) \\
		&= \id_{\Lambda^p \Cnum^n} \otimes (\operator_\semigrelemB - \chi(\operator_\semigrelemB))
	\end{align*}
	 with \( \alpha_p \in \mathcal L(\Lambda^p \Cnum^n \otimes_{\Cnum} \complVs_{\operatorFam}, \Lambda^{p-1} \Cnum^n \otimes_{\Cnum} \complVs_{\operatorFam}) \).
	 As \( \Lambda^p \Cnum^n \) is a finite-dimensional vector space the spectrum and essential spectrum of \( \id_{\Lambda^p \Cnum^n} \otimes \operator_{\semigrelemB} \) agree with the spectrum, respectively the essential spectrum of \( \operator_{\semigrelemB} \).
	 Consequently the \( \id_{\Lambda^p \Cnum^n} \otimes \operator_{\semigrelemB} \) is quasi-compact as well. Moreover, we can write
	\[
	\id_{\Lambda^p \Cnum^n} \otimes (\operator_{\semigrelemB} - \chi(\operator_{\semigrelemB})) = (\id_{\Lambda^p \Cnum^n} \otimes \operator_{\semigrelemB}) - (\id_{\Lambda^p \Cnum^n} \otimes \id_{\banachSpace} \cdot \chi(\operator_{\semigrelemB})).
	\]
	As \( \id_{\Lambda^p \Cnum^n} \otimes \operator_\semigrelemB \) is quasi-compact and \( \abs{\chi(\operator_{\semigrelemB})} > \essSpecRad(\operator_{\semigrelemB}) \), the operator \( \id_{\Lambda^p \Cnum^n} \otimes (\operator_\semigrelemB - \chi(\operator_\semigrelemB)) \) is Fredholm of index zero (in the sense of \thref{defi:fredholmIndexZero}) by the definition of the essential spectrum of a single operator \eqref{equ:essentialSpec}.
\end{proof}

\begin{thm}\label{thm:jointDiscreteSpecInFiniteSpace}
	Given a character \( \chi \in \Hom_{\Cnum\text{-}\mathrm{alg.}}(\polyAlg, \Cnum) \) with \( \abs{\chi(\operator_{\semigrelemB})} > \essSpecRad(\operator_{\semigrelemB}) \) for some \( \semigrelemB \in H^\circ \), then there exists a finite-dimensional \( \operatorFam \)-invariant vector space \( \finiteSpace \subseteq \complVs \) such that for all \( i \in \Znum \) we have
	\begin{align*}
		\Tor_i^{\polyAlg}(\Cnum_\chi, \complVs_{\operatorFam}) &\cong \Tor_i^{\polyAlg}(\Cnum_\chi, \finiteSpace_{\operatorFam}), \\
		\Ext^i_{\polyAlg}(\Cnum_\chi, \complVs_{\operatorFam}) &\cong \Ext^i_{\polyAlg}(\Cnum_\chi, \finiteSpace_{\operatorFam}).
	\end{align*}
\end{thm}

\begin{proof}
	We construct the same operators \( B_i \) as in the proof of \thref{prop:Fredholm}.
	Then we apply \thref{lem:projectionOnEigenspace} with \( C = \operator_{\semigrelemB} \) and \( \lambda=\chi(\semigrelemB) \).
	As the image of the Riesz projector \( \widetilde{P}_0 \) of \( \operator_{\semigrelemB} \) with eigenvalue \( \chi(\semigrelemB) \) is finite dimensional (see \thref{rem:propOfFredholmAndQuasiCompOp}\ref{item:rieszProjFiniteIm}) and \( \operatorFam \)-invariant, we take it as the subspace \( \finiteSpace_{\operatorFam} \).
\end{proof}

\begin{thm}\label{thm:jointEigenvalues}
	Given a character \( \chi \in \sigma_{\mathrm{T}}(\operator) \) with \( \abs{\chi(\operator_{\semigrelemB})} > \essSpecRad(\operator_{\semigrelemB}) \) for some \( \semigrelemB \in H^\circ \).
	Then the character \( \chi \) is a joint eigenvalue.
\end{thm}

\begin{proof}
	In view of \thref{thm:jointDiscreteSpecInFiniteSpace} the assertion follows from \thref{lem:taylorSpecOnFiniteSpace}.
\end{proof}

\subsection{Application to Transfer Operators}

To apply the established spectral theory to the transfer operators we recall the setting from \Cref{subsec:repOnFinitelyGenCommuMonoids}.
The abelian monoid \( \semigr \subseteq \domCoweights \) is given by a finitely generated submonoid of the dominant coweights corresponding to the root system that belongs to a local building.
We denote the set of generators of \( \semigr \) by \( \Set{h_1,\ldots,h_m} \).
 
In \Cref{equ:commuRelTOp} we have seen that \( \domCoweights \to \mathcal L(\lipschitzFunc), \; \paramA \mapsto \tOp_{\paramA} \) defines a homomorphism of monoids.
Hence the restriction to \( \semigr \subseteq \domCoweights \) is a homomorphism as well.
Moreover, we want to assume that \( \semigr^\circ \coloneqq \innerCoweights \cap \semigr \neq \emptyset \).

We want to study the Taylor spectrum of the operator algebra
\[
\mathcal A_{\semigr} \coloneqq \spr{\tOp_{h_1},\ldots,\tOp_{h_m}}_{\Cnum\text{-}\mathrm{alg.}} = \spr{\tOp_k \mid k \in \semigr}_{\Cnum\text{-}\mathrm{alg.}}.
\]
defined by the transfer operators.
To do this we study the Taylor spectrum of the operator family \( \tOp \coloneqq (\tOp_{h_1},\ldots,\tOp_{h_m}) \) which is a subset \( \sigma_{\mathrm{T}}(\tOp) \subseteq \Hom_{\Cnum\text{-}\mathrm{alg.}}(\Cnum[\semigr], \Cnum) \).
\begin{exmp}
	In the setting of algebraic groups the above situation arises naturally by taking a lattice \( L^\vee \) with \( \corootLat \subseteq L^\vee \subseteq \coweightLat \).
	In such a situation we set \( \semigr = L_+^\vee \coloneqq L^\vee \cap \domCoweights \) as our monoid.
	For example in the case of simply-connected algebraic groups we have \( L^\vee = \corootLat \) and in the setting of a group of adjoint type we have \( \corootLat = L^\vee = P^\vee \).
	This leads to the study of the operator algebra \( \mathcal A_{L_+^\vee} \coloneqq \spr{\tOp_\mu \mid \mu \in L_+^\vee}_{\Cnum\text{-}\mathrm{alg.}} \).
\end{exmp}

Recall that by \thref{rem:barA-characters} the Taylor spectrum is independent of the generators of \( \semigr \) we have chosen.
\thref{thm:jointEigenvalues} immediately yields the following theorem on the existence of a discrete spectrum of dynamical resonances.

\begin{thm}
	Let \( \vartheta \in {}]0,1[{} \), \( \semigr \subset \domCoweights \) a finitely generated submonoid, and \( \mathcal A_{\semigr} \coloneqq \spr{\tOp_{h_1},\ldots,\tOp_{h_m}}_{\Cnum\text{-}\mathrm{alg.}} \) the operator algebra acting on \( \lipschitzFunc \) then
	\[  
	\sigma_{\mathrm{T}}(\tOp) \cap\Set{\chi\in \Hom_{\Cnum\text{-}\mathrm{alg.}}(\Cnum[H], \Cnum) \given \exists k \in \semigr^\circ: \  \abs{\chi(k)} > \vartheta} \subset \sigma_{\mathrm{p}}(\tOp).
	\]
\end{thm}

\begin{thm}\label{thm:spectrumInF1}
	Let \( \vartheta \in {}]0,1[{} \) be a fixed parameter for the metric on \( \secQuot \), \( \semigr \subseteq \domCoweights \) a finitely generated monoid with \( \semigr^\circ \neq \emptyset \).
	We denote the generators of \( \semigr \) by \( \Set{h_1,\ldots,h_m} \).
	Let \( k \in H^\circ \) and \( \tOp = (\tOp_{h_1},\ldots,\tOp_{h_m}) \) be the family of transfer operators acting continuously on \( \lipschitzFunc \).
	Given a character \( \chi \in \Hom_{\Cnum\text{-}\mathrm{alg.}}(\polyAlg, \Cnum) \) such that \( \abs{\chi(\tOp_k)} > \vartheta \) we have
	\begin{align*}
		\Tor_i^{\polyAlg}(\Cnum_\chi, \lipschitzFunc) &\cong \Tor_i^{\polyAlg}(\Cnum_\chi, F_1), \\
		\Ext^i_{\polyAlg}(\Cnum_\chi, \lipschitzFunc) &\cong \Ext^i_{\polyAlg}(\Cnum_\chi, F_1)
	\end{align*}
	for all \( i \in \Znum \).
\end{thm}

\begin{proof}
	We only prove the \( \Tor \)-case. By \thref{prop:parametrixForSemigr} we find \( B_1(\chi),\ldots, B_m(\chi) \in \mathcal A \) such that
	\[
	\sum\limits_{i=1}^m (\tOp_{h_i} - \chi(\tOp_{h_i}))B_i(\chi) = \tOp_{k} - \chi(\tOp_{k}).
	\]
	In \thref{thm:transferOpInteriorQuasiComp} we have seen that \( \tOp_{k} \) is quasi-compact and \( \essSpecRad(\tOp_{k}) \leq \vartheta < \abs{\chi(\tOp_{k})} \).
	This allows us to apply \thref{thm:jointDiscreteSpecInFiniteSpace} and \thref{lem:projectionOnEigenspace}. Thus, we know that we can compute
	\begin{equation}\label{equ:firstFinite}
		\Tor_i^{\polyAlg}(\Cnum_\chi, \lipschitzFunc) \cong \Tor_i^{\polyAlg}(\Cnum_\chi, \finiteSpace_{\tOp})
	\end{equation}
	in some finite-dimensional subspace \(\finiteSpace_{\tOp}\) given by the image of the Riesz projector of \( \tOp_{k} \) onto the eigenvalue \( \chi(\tOp_{k}) \).
	If \( \chi(\tOp_{k}) \) is not an eigenvalue, the space is the zero-space, else the image of the Riesz projector is the generalized eigenspace \( \finiteSpace_{\tOp} = \Hau(\tOp_{k}, \chi(\tOp_{k})) \subseteq \lipschitzFunc \).
	
	On the other hand we know that \( F_1 \) is a finite-dimensional \( \tOp \)-invariant vector space (see \thref{prop:fnFiniteSpaces} and \thref{prop:fnInvariantQuot}\ref{item:fnTOpinv}).
	Again we can apply \thref{thm:jointDiscreteSpecInFiniteSpace} and find
	\begin{equation}\label{equ:secondFinite}
		\Tor_i^{\polyAlg}(\Cnum_\chi, F_1) \cong \Tor_i^{\polyAlg}(\Cnum_\chi, \finiteSpace^\prime_{\tOp}).
	\end{equation}
	Here \( \finiteSpace^\prime_{\tOp} \) is the image of the Riesz projector given by the generalized eigenspace \( \Hau(\tOp_{k}, \chi(\tOp_{k})) \subseteq F_1 \).
		But we have shown in \thref{prop:GenEigenspaceInFOne} that for all eigenvalues \( \chi(\tOp_{k}) > \vartheta \) the generalized eigenspace of \( \tOp_{k} \) on \( \lipschitzFunc \) is contained in \( F_1 \).
	Consequently \( \finiteSpace_{\tOp} = \finiteSpace^\prime_{\tOp} \) and combining \eqref{equ:firstFinite} and \eqref{equ:secondFinite} we obtain
	\[
	\Tor_i^{\polyAlg}(\Cnum_\chi, \lipschitzFunc) \cong \Tor_i^{\polyAlg}(\Cnum_\chi, F_1). \qedhere
	\]
\end{proof}

Applying \thref{thm:spectrumInF1}, \thref{prop:GenEigenspaceInFOne} and \thref{cor:eigenvalIndOfTheta} we obtain the following theorem.

\begin{thm}\label{thm:jointEigenvaluesOnFiniteDimSpace}
	Let \( \chi \in \Hom_{\Cnum\text{-}\mathrm{alg.}}(\polyAlg, \Cnum) \) be a character.
	We set \( \alpha \coloneqq \sup_{\paramA \in H^\circ} \abs{\chi(\tOp_{\paramA})} \) and assume that \( \alpha > 0 \).
	Then the following statements are equivalent.
	\begin{enumerate}[(i)]
		\item
			There exists some \( \vartheta \in (0, \min\Set{1,\alpha}) \) such that \( \chi \in \sigma_{\mathrm{T}}(\tOp) \).
		\item
			For every \( \vartheta \in (0, \min \Set{1,\alpha}) \) \( \chi \) is a joint eigenvalue on \( \lipschitzFunc \).
		\item
			The character \( \chi \) is a joint eigenvalue for the operator algebra
			\[
			\restr{{\mathcal A_{H}}}{{F_{1}(\secQuot)}} \coloneqq \spn*{\restr{\tOp_{\paramA}}{F_1} \mid \paramA \in H}_{\Cnum\text{-}\mathrm{alg.}}.
			\]
			In other words, \( \chi \) is a joint eigenvalue of the family of transfer operators restricted to the invariant subspace \( F_1 \).
	\end{enumerate}
\end{thm}

\begin{proof}
	\mbox{}
	\begin{prooflist}
		\ximpliesy{(ii)}{(i)}
			Trivial.
		\ximpliesy{(i)}{(iii)}
			By assumption \( \chi \in \sigma_{\mathrm{T}}(\tOp) \), which tells us that there is some \( i \in \Set{0,\ldots,n} \) such that \( \Tor_i^{\polyAlg}(\Cnum_\chi, \lipschitzFunc) \neq 0 \).
			By \thref{thm:spectrumInF1} we have \( \Tor_i^{\polyAlg}(\Cnum_\chi, F_1) \neq 0 \).
			But as \( F_1 \) is a finite-dimensional vector space \( \chi \) is a common eigenvalue of the algebra.
			
			Recall that being a joint eigenvalue means
			\[
			\Eig(\tOp, \chi) = \bigcap\limits_{\paramA \in \semigr} \Eig(\tOp_\paramA, \chi(\paramA)) \neq \Set{0}.
			\]
			Now let \( \varphi \in \Eig(\tOp, \chi) \).
			By assumption we find \( \paramA \in \semigr^\circ \) fulfilling \( \abs{\chi(\paramA)} > \vartheta \).
			This gives us
			\[
			\varphi \in \Eig(\tOp_{\paramA}, \chi(\tOp_{\paramA})) \subseteq F_1
			\]
			by \thref{prop:GenEigenspaceInFOne}.
			Hence we have \( \Eig(\tOp, \chi) \subseteq F_1 \).
		\ximpliesy{(iii)}{(ii)}
			As \( F_1 \subseteq \lipschitzFunc \) is an invariant subspace for every \( \vartheta \in {}]0,1[{} \) the assertion follows. \qedhere
	\end{prooflist}
\end{proof}

\begin{cor}
	Let \( \Set{h_1,\ldots,h_m} \subseteq \semigr \subseteq \domCoweights \) be a set of generators for the monoid \( \semigr \).
	By construction the operators \( (\tOp_{h_i})_{i=1,\ldots,m} \) generate the algebra \( \mathcal A_\semigr \).
	Given a joint eigenvalue \( \chi \in \sigma_{\mathrm{T}}(\tOp) \) with \( \abs{\chi(\tOp_{k})} > \vartheta \) for some \( k \in \semigr^\circ \) the corresponding joint eigenspace \( \Eig(\tOp, \chi) \) is finite dimensional.
\end{cor}

\begin{proof}
	This follows from \thref{thm:jointEigenvaluesOnFiniteDimSpace}.
	To see this, recall that for any joint eigenvector \( v \) we have
	\[
	\tOp_{k} v = \chi(\tOp_{k}) v.
	\]
	Since \( \abs{\chi(\tOp_{k})} > \vartheta \) and \( \tOp_{k} \) is quasi-compact, the character \( \chi \) is a normal eigenvalue of \( \tOp_{k} \).
	In \thref{prop:fnFiniteSpaces} we have seen that \( F_1 \) is finite-dimensional.
	Therefore, by \thref{prop:GenEigenspaceInFOne} the corresponding eigenspace \( \Eig(\tOp_{k}, \chi(\tOp_{k})) \subseteq F_1 \) is finite-dimensional as well.
\end{proof}

%%%%%%%%%%%%%%%%%%%%%%%%%%%%%%%%%%%%%%%%%%%%%%%%%%%%%%%%%%%%%%%%%%

\bibliographystyle{amsalpha}
\bibliography{bibo}

\bigskip

%%%%%%%%%%%%%%%%%%%%%%%%%%%%%%%%%%%%%%%%%%%%%%%%%%%%%%%%%%%%%%%%%%

\end{document}